\documentclass[11pt]{article}

\usepackage{amsfonts,amssymb,amsmath,graphicx,graphics}
\usepackage[english]{babel}
 \def\botcaption#1#2{\medskip\centerline{{\scshape #1.}\kern8pt
 {\rm #2}}\bigskip}

 \makeatletter
 \@addtoreset{equation}{subsection}
 \makeatother

 \makeatletter
 \@addtoreset{enunciato}{section}
 \makeatother

 \newcounter{enunciato}[subsection]

 \newtheorem{ittheorem}{Theorem}
 \newtheorem{itlemma}{Lemma}
 \newtheorem{itproposition}{Proposition}
 \newtheorem{itdefinition}{Definition}
 \newtheorem{itassumption}{Assumption}
 \newtheorem{itremark}{Remark}
 \newtheorem{itclaim}{Claim}

 \newenvironment{theorem}{\addtocounter{enunciato}{1}
 \begin{ittheorem}}{\end{ittheorem}}

 \newenvironment{lemma}{\addtocounter{enunciato}{1}
 \begin{itlemma}}{\end{itlemma}}

 \newenvironment{proposition}{\addtocounter{enunciato}{1}
 \begin{itproposition}}{\end{itproposition}}

 \newenvironment{definition}{\addtocounter{enunciato}{1}
 \begin{itdefinition}}{\end{itdefinition}}

  \newenvironment{assumption}{\addtocounter{enunciato}{1}
 \begin{itassumption}}{\end{itassumption}}

 \newenvironment{remark}{\addtocounter{enunciato}{1}
 \begin{itremark}}{\end{itremark}}

 \newenvironment{claim}{\addtocounter{enunciato}{1}
 \begin{itclaim}}{\end{itclaim}}

 \newenvironment{proof}{\noindent {\bf Proof.\,}
 }{\hspace*{\fill}$\square$\medskip}

 \newcommand{\be}[1]{\begin{equation}\label{#1}}
 \newcommand{\ee}{\end{equation}}

 \newcommand{\bl}[1]{\begin{lemma}\label{#1}}
 \newcommand{\el}{\end{lemma}}

 \newcommand{\br}[1]{\begin{remark}\label{#1}}
 \newcommand{\er}{\end{remark}}

 \newcommand{\bt}[1]{\begin{theorem}\label{#1}}
 \newcommand{\et}{\end{theorem}}

 \newcommand{\bd}[1]{\begin{definition}\label{#1}}
 \newcommand{\ed}{\end{definition}}

 \newcommand{\bcl}[1]{\begin{claim}\label{#1}}
 \newcommand{\ecl}{\end{claim}}

 \newcommand{\bp}[1]{\begin{proposition}\label{#1}}
 \newcommand{\ep}{\end{proposition}}

 \newcommand{\bc}[1]{\begin{corollary}\label{#1}}
 \newcommand{\ec}{\end{corollary}}

 \newcommand{\bpr}{\begin{proof}}
 \newcommand{\epr}{\end{proof}}

 \newcommand{\bi}{\begin{itemize}}
 \newcommand{\ei}{\end{itemize}}

 \newcommand{\ben}{\begin{enumerate}}
 \newcommand{\een}{\end{enumerate}}

 \def\botcaption#1#2{\medskip\centerline{{\scshape #1.}\kern8pt
 {\rm #2}}\bigskip}

 \def \ba {\begin{array}}
 \def \ea {\end{array}}

 \def \Z {{\mathbb Z}}
 \def \R {{\mathbb R}}
 
 \def \N {{\mathbb N}}
 \def \P {{\mathbb P}}
 \def \E {{\mathbb E}}
 
 \def \da {\downarrow}

 \def \cO {{\mathcal O}}
 \def \cW {{\mathcal W}}
 \def \cD {{\mathcal D}}
 \def \cL {{\mathcal L}}
 \def \cI {{\mathcal I}}
 \def \cJ {{\mathcal J}}
 \def \cR {{\mathcal R}}
 \def \cA {{\mathcal A}}
 
 \def \cN {{\mathcal N}}
 \def \cE {{\mathcal E}}
 \def \cP {{\mathcal P}}
 \def \cQ {{\mathcal Q}}
 
 \def \cF {{\mathcal F}}

 \def \cV {{\mathcal V}}
 \def \cH {{\mathcal H}}
 \def \cT {{\mathcal T}}
 \def \cZ {{\mathcal Z}}
 \def \cS {{\mathcal S}}

 \def \ind {{1}}
 \def \gep {{\varepsilon}}

 \def \DOM {{\hbox{\footnotesize\rm DOM}}}
 \def \CONE {{\hbox{\footnotesize\rm CONE}}}

\begin{document}

\title{On the localized phase of a copolymer in an emulsion:
supercritical percolation regime}

\author{\renewcommand{\thefootnote}{\arabic{footnote}}
F.\ den Hollander
\footnotemark[1]\,\,\,\footnotemark[2]
\\
\renewcommand{\thefootnote}{\arabic{footnote}}
N.\ P\'etr\'elis
\footnotemark[2]
}

\footnotetext[1]
{Mathematical Institute, Leiden University, P.O.\ Box 9512,
2300 RA Leiden, The Netherlands}\,

\footnotetext[2]
{EURANDOM, P.O.\ Box 513, 5600 MB Eindhoven, The Netherlands}

\maketitle

\begin{abstract}
In this paper we study a two-dimensional directed self-avoiding walk model of
a random copolymer in a random emulsion. The copolymer is a random concatenation
of monomers of two types, $A$ and $B$, each occurring with density $\frac{1}{2}$.
The emulsion is a random mixture of liquids of two types, $A$ and $B$, organised
in large square blocks occurring with density $p$ and $1-p$, respectively, where
$p \in (0,1)$. The copolymer in the emulsion has an energy that is minus $\alpha$
times the number of $AA$-matches minus $\beta$ times the number of $BB$-matches,
where without loss of generality the interaction parameters can be taken from
the cone $\{(\alpha,\beta)\in\R^2\colon\,\alpha\geq |\beta|\}$. To make the model
mathematically tractable, we assume that the copolymer is directed and can only
enter and exit a pair of neighbouring blocks at diagonally opposite corners.

In \cite{dHW06}, a variational expression was derived for the quenched free energy
per monomer in the limit as the length $n$ of the copolymer tends to infinity and
the blocks in the emulsion have size $L_n$ such that $L_n \to \infty$ and $L_n/n
\to 0$. Under this restriction, the free energy is self-averaging with respect to
both types of randomness. It was found that in the supercritical percolation regime
$p \geq p_c$, with $p_c$ the critical probability for directed bond percolation on
the square lattice, the free energy has a phase transition along a curve in the cone
that is independent of $p$. At this critical curve, there is a transition from a phase
where the copolymer is fully delocalized into the $A$-blocks to a phase where it is
partially localized near the $AB$-interface. In the present paper we prove three
theorems that complete the analysis of the phase diagram : (1) the critical
curve is strictly increasing; (2) the phase transition is second order; (3) the
free energy is infinitely differentiable throughout the partially localized phase.

In the subcritical percolation regime $p<p_c$, the phase diagram is much more
complex. This regime will be treated in a forthcoming paper.

\vskip 0.5truecm
\noindent
\emph{AMS} 2000 \emph{subject classifications.} 60F10, 60K37, 82B27.
{\it Acknowledgment.} NP is supported by a postdoctoral fellowship from the
Netherlands Organization for Scientific Research (grant 613.000.438).
FdH and NP are grateful to the Pacific Institute for the Mathematical
Sciences and the Mathematics Department of the University of British Columbia,
Vancouver, Canada, for hospitality: FdH from January to August 2006, NP from
mid-March to mid-April 2006 when the work in this paper started.
\end{abstract}


\section{Introduction and main results}
\label{S1}

\subsection{Background}
\label{S1.1}

The problem considered in this paper is the \emph{localization transition} of
a random copolymer near a random interface. Suppose that we have two immiscible
liquids, say, oil and water, and a copolymer chain consisting of two types of
monomer, say, hydrophobic and hydrophilic. Suppose that it is energetically
favourable for monomers of one type to be in one liquid and for monomers of the
other type to be in the other liquid. At high temperatures the copolymer will
delocalize into one of the liquids in order to maximise its entropy, while at
low temperatures energetic effects will dominate and the copolymer will localize
close to the interface between the two liquids, because in this way it is able
to place more than half of its monomers in their preferred liquid. In the
limit as the copolymer becomes long, we may expect a phase transition.

In the literature most attention has focussed on models with a \emph{single
flat infinite} interface or an \emph{infinite array of parallel flat infinite}
interfaces. Relevant references can be found in P\'etr\'elis \cite{P06}. In the
present paper we continue the analysis of a model introduced in den Hollander
and Whittington \cite{dHW06}, where the interface has a \emph{random shape}.
In particular, the situation was considered in which the square lattice is divided
into large blocks, and each block is independently labelled $A$ (oil) or $B$
(water) with probability $p$ and $1-p$, respectively, i.e., the interface has
a \emph{percolation type structure}. This is a primitive model of an emulsion,
consisting of oil droplets dispersed in water (see Figure \ref{fig-copolemul}).

\begin{figure}
\begin{center}
\includegraphics[scale = 0.3]{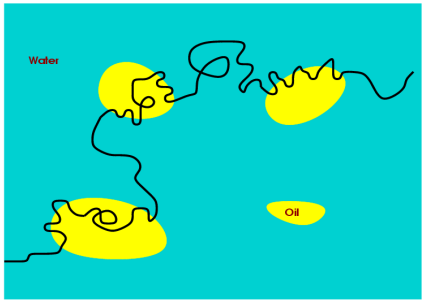}
\end{center}
\caption{An undirected copolymer in an emulsion.}
\label{fig-copolemul}
\end{figure}

The copolymer consists of an i.i.d.\ random concatenation of monomers of type
$A$ (hydrophobic) and $B$ (hydrophilic). It is energetically favourable for
monomers of type $A$ to be in the $A$-blocks and for monomers of type $B$ to
be in the $B$-blocks. Under the restriction that \emph{the copolymer is directed
and can only enter and exit a pair of neighbouring blocks at diagonally opposite
corners}, it was shown that there are phase transitions between phases where
the copolymer is \emph{fully delocalized} away from the interface and phases
where it is \emph{partially localized} near the interface. It turns out that
the phase diagram \emph{does not} depend on $p$ when $p \geq p_c$, the critical
value for directed bond percolation on $\Z^2$, while it \emph{does} depend on
$p$ when $p<p_c$. In the present paper we focus on the supercritical percolation
regime.

Our paper is organised as follows. In the rest of Section \ref{S1} we recall the
definition of the model, state the relevant results from \cite{dHW06}, and formulate
three theorems for the supercritical percolation regime. These theorems are proved in
Sections \ref{S3}, \ref{S4} and \ref{S5}, respectively. Section \ref{S2} recalls
the key variational formula for the free energy, as well as some basic facts
about block pair free energies and path entropies needed along the way.

\subsection{The model}
\label{S1.2}

Each positive integer is randomly labelled $A$ or $B$, with probability $\frac{1}{2}$
each, independently for different integers. The resulting labelling is denoted by
\be{bondlabel}
\omega = \{\omega_i \colon\, i \in \N\} \in \{A,B\}^\N
\ee
and represents the \emph{randomness of the copolymer}, with $A$ denoting a hydrophobic
monomer and $B$ a hydrophilic monomer. Fix $p \in (0,1)$ and $L_n \in \N$. Partition
$\R^2$ into square blocks of size $L_n$:
\be{blocks}
\R^2 = \bigcup_{x \in \Z^2} \Lambda_{L_n}(x), \qquad
\Lambda_{L_n}(x) = xL_n + (0,L_n]^2.
\ee
Each block is randomly labelled $A$ or $B$, with probability $p$, respectively,
$1-p$, independently for different blocks. The resulting labelling is denoted by
\be{blocklabel}
\Omega = \{\Omega(x) \colon\, x \in \Z^2\} \in \{A,B\}^{\Z^2}
\ee
and represents the \emph{randomness of the emulsion}, with $A$ denoting oil and
$B$ denoting water.

Let
\begin{itemize}
\item
$\cW_n =$ the set of $n$-step \emph{directed self-avoiding paths} starting
at the origin and being allowed to move \emph{upwards, downwards and to the right}.
\item
$\cW_{n,L_n} =$ the subset of $\cW_n$ consisting of those paths that enter blocks
at a corner, exit blocks at one of the two corners \emph{diagonally opposite} the
one where it entered, and in between \emph{stay confined} to the two blocks that
are seen upon entering (see Figure \ref{fig-copolemulblock}).
\end{itemize}
The corner restriction, which is unphysical, is put in to make the model mathematically
tractable. We will see that, despite this restriction, the model has physically
relevant behaviour.

\begin{figure}
\begin{center}
\includegraphics[scale = 0.4]{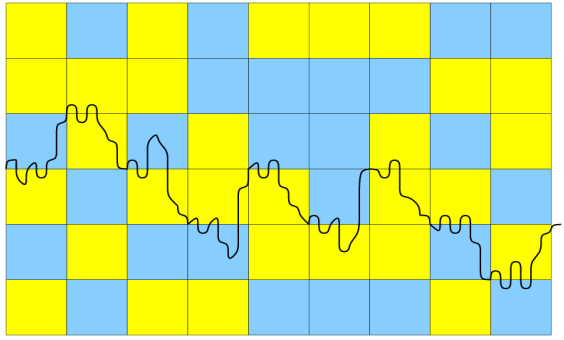}
\end{center}
\caption{A directed self-avoiding path crossing blocks of oil and water diagonally.
The light-shaded blocks are oil, the dark-shaded blocks are water. Each block is
$L_n$ lattice spacings wide in both directions. The path carries hydrophobic and
hydrophilic monomers on the lattice scale, which are not indicated.}
\label{fig-copolemulblock}
\end{figure}


Given $\omega,\Omega$ and $n$, with each path $\pi \in \cW_{n,L_n}$ we associate
an \emph{energy} given by the Hamiltonian
\be{Hamiltonian}
H_{n,L_n}^{\omega,\Omega}(\pi)
= - \sum_{i=1}^n \Big(\alpha\, 1\left\{\omega_i=\Omega^{L_n}_{(\pi_{i-1},\pi_i)}=A\right\}
+ \beta\, 1\left\{\omega_i=\Omega^{L_n}_{(\pi_{i-1},\pi_i)}=B\right\}\Big),
\ee
where $(\pi_{i-1},\pi_i)$ denotes the $i$-th step of the path and $\Omega^{L_n}_{(\pi_{i-1},
\pi_i)}$ denotes the label of the block this step lies in. What this Hamiltonian
does is count the number of $AA$-matches and $BB$-matches and assign them
energy $-\alpha$ and $-\beta$, respectively, where $\alpha,\beta\in\R$. (Note
that the interaction is assigned to bonds rather than to sites: we identify the
monomers with the steps of the path). As we will recall in Section \ref{S2.1},
without loss of generality we may restrict the interaction parameters to the cone
\be{defcone}
\CONE = \{(\alpha,\beta)\in\R^2\colon\,\alpha\geq |\beta|\}.
\ee

Given $\omega,\Omega$ and $n$, we define the \emph{quenched free energy per step}
as
\be{fedef}
\begin{aligned}
f_{n,L_n}^{\omega,\Omega}
&= \frac{1}{n} \log Z_{n,L_n}^{\omega,\Omega},\\
Z_{n,L_n}^{\omega,\Omega}
&= \sum\limits_{\pi\in\cW_{n,L_n}}
\exp\left[-H_{n,L_n}^{\omega,\Omega}(\pi)\right].
\end{aligned}
\ee
We are interested in the limit $n \to \infty$ subject to the restriction
\be{Ln}
L_n \to \infty \qquad \mbox{ and } \qquad \frac{1}{n}L_n \to 0.
\ee
This is a \emph{coarse-graining} limit where the path spends a long time in each
single block yet visits many blocks. In this limit, there is a separation between
a \emph{copolymer scale} and an \emph{emulsion scale}.

In \cite{dHW06}, Theorem 1.3.1, it was shown that
\be{flim}
\lim_{n\to\infty} f_{n,L_n}^{\omega,\Omega}
= f = f(\alpha,\beta;p)
\ee
exists $\omega,\Omega$-a.s.\ and in mean, is finite and non-random, and can be
expressed as a variational problem involving the free energies of the copolymer
in each of the four block pairs it may encounter and the frequencies at which
the copolymer visits each of these block pairs on the coarse-grained block scale.
This variational problem, which is recalled in Section \ref{S2.1}, will be the
starting point of our analysis.

\subsection{Phase diagram for $p \geq p_c$}
\label{S1.3}

In the supercritical regime the \emph{oil blocks percolate}, and so the
coarse-grained path can choose between moving into the oil or running along
the interface between the oil and the water (see Figure \ref{fig-copolemulinfcl}).
We begin by recalling from den Hollander and Whittington \cite{dHW06} the two
main theorems for the supercritical percolation regime (see Figure
\ref{fig-supcritphd}).

\begin{figure}
\begin{center}
\includegraphics[scale = 0.4]{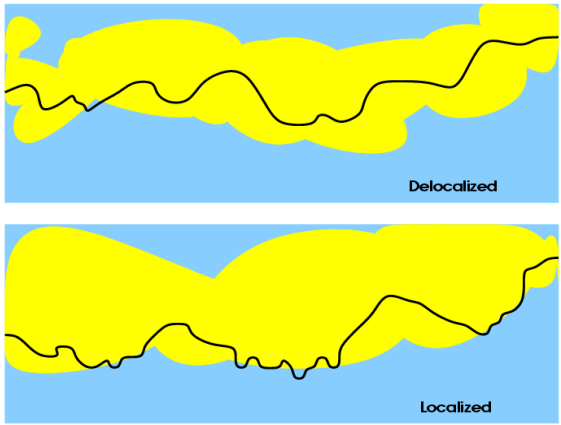}
\end{center}
\caption{Two possible strategies when the oil percolates.}
\label{fig-copolemulinfcl}
\end{figure}

\bt{fesupcr} {\rm (\cite{dHW06}, Theorem 1.4.1)} Let $p \geq p_c$. Then
$(\alpha,\beta)\mapsto f(\alpha,\beta;p)$ is non-analytic along the curve
in $\CONE$ separating the two regions
\be{DL}
\begin{aligned}
\cD &= \hbox{ delocalized phase}
    &= \left\{(\alpha,\beta)\in\CONE \colon f(\alpha,\beta;p)
     = \textstyle{\frac12}\alpha + \varpi\right\},\\
\cL &= \hbox{ localized phase}
    &= \left\{(\alpha,\beta)\in\CONE \colon f(\alpha,\beta;p)
    > \textstyle{\frac12}\alpha + \varpi\right\}.
\end{aligned}
\ee
Here, $\varpi = \lim_{n\to\infty} \frac{1}{n} \log |\cW_{n,L_n}| =
\frac12\log 5$ is the entropy per step of the walk subject to {\rm
(\ref{Ln})}.
\et

\bt{phtrcurve} {\rm (\cite{dHW06}, Theorem 1.4.3)} Let $p \geq p_c$.\\
(i) For every $\alpha \geq 0$ there exists a $\beta_c(\alpha) \in
[0,\alpha]$ such that the copolymer is
\be{betac}
\ba{lll}
&\mbox{delocalized} &\mbox{if }\, -\alpha \leq \beta \leq
\beta_c(\alpha),\\
&\mbox{localized}   &\mbox{if }\, \beta_c(\alpha) < \beta \leq \alpha.
\ea
\ee
(ii) $\alpha\mapsto\beta_c(\alpha)$ is independent of $p$, continuous,
non-decreasing and concave on $[0,\infty)$. There exist $\alpha^* \in
(0,\infty)$ and $\beta^* \in [\alpha^*,\infty)$ such that
\be{phtrlims}
\ba{lll}
&\beta_c(\alpha) = \alpha &\mbox{if }\,\alpha \leq \alpha^*,\\
&\beta_c(\alpha) < \alpha &\mbox{if }\,\alpha > \alpha^*,
\ea
\ee
and
\be{slopedis}
\lim_{\alpha\da\alpha^*} \frac{\alpha-\beta_c(\alpha)}{\alpha-\alpha^*}
\in [0,1), \qquad \lim_{\alpha \to \infty} \beta_c(\alpha) = \beta^*.
\ee
\et

The intuition behind Theorem \ref{fesupcr} is as follows (see Figure
\ref{fig-copolemulinfcl}). Suppose that $p>p_c$. Then the $A$-blocks percolate.
Therefore the copolymer has the option of moving to the infinite cluster of
$A$-blocks and staying inside that infinite cluster forever, thus seeing only
$AA$-blocks. In doing so, it loses an entropy of at most $o(n/L_n) = o(n)$ (on
the coarse-grained scale), it gains an energy $\frac{1}{2}\alpha n+o(n)$ (on
the lattice scale, because only half of its monomers are matched), and it gains
an entropy $\varpi n+o(n)$ (on the lattice scale, because it crosses blocks
diagonally). Alternatively, the path has the option of running along the boundary
of the infinite cluster (at least part of the time), during which it sees
$AB$-blocks and (when $\beta\geq 0$) gains more energy by matching more than
half of its monomers. Consequently,
\be{feineq}
f(\alpha,\beta;p) \geq \textstyle{\frac12}\alpha + \varpi.
\ee
The boundary between the two regimes in (\ref{DL}) corresponds to the crossover
from \emph{full delocalization} into the $A$-blocks to \emph{partial localization}
near the $AB$-interfaces. The critical curve does \emph{not} depend on $p$ as long
as $p>p_c$. Because $p\mapsto f(\alpha,\beta;p)$ is continuous (see Theorem
\ref{feiden}(iii) in Section \ref{S2.1}), the same critical curve occurs at $p=p_c$.

\begin{figure}
\begin{center}
\setlength{\unitlength}{0.45cm}
\begin{picture}(12,12)(0,-3)
\put(0,0){\line(12,0){12}}
\put(0,0){\line(0,8){8}}
\put(0,0){\line(0,-3){3}}
\put(0,0){\line(-4,0){4}}
{\thicklines
\qbezier(2,2)(4,3.2)(9,4.5)
\thicklines
\qbezier(0,0)(1,1)(2,2)
}
\qbezier[60](2,2)(4.5,3.5)(9.5,6.5)
\qbezier[60](0,5.5)(5,5.5)(10,5.5)
\qbezier[60](2,2)(4.5,4.5)(7,7)
\qbezier[15](2,0)(2,1)(2,2)
\qbezier[70](4,-4)(2,-2)(0,0)
\put(-.7,-1){$0$}
\put(12.5,-0.2){$\alpha$}
\put(-0.1,8.5){$\beta$}
\put(2,2){\circle*{.3}}
\put(1.7,-1){$\alpha^*$}
\put(-1.4,5.4){$\beta^*$}
\put(10.5,4.5){$\beta_c(\alpha)$}
\put(5.6,4.3){$\cal L$}
\put(6,2.3){$\cal D$}
\end{picture}
\end{center}
\caption{Qualitative picture of $\alpha\mapsto\beta_c(\alpha)$ for $p \geq p_c$.}
\label{fig-supcritphd}
\end{figure}
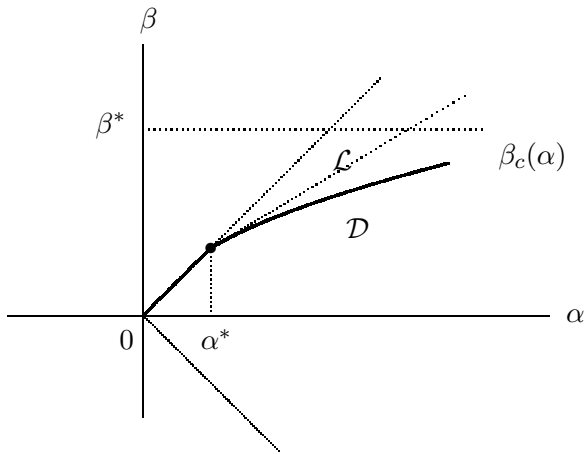


The proof of Theorem \ref{phtrcurve} relies on a representation of $\cD$ and $\cL$
in terms of the single interface (!) free energy (see Proposition \ref{p:phtrinfchar}
in Section \ref{S2.3}). This representation, which is key to the analysis of the
critical curve, expresses the fact that localization occurs for the emulsion free
energy only when the single interface free energy is \emph{sufficiently deep inside}
its localized phase. This gap is needed to compensate for the loss of entropy
associated with running along the interface and crossing at a steeper angle.

The intuition behind Theorem \ref{phtrcurve} is as follows (see Figure
\ref{fig-supcritphd}). Pick a point $(\alpha,\beta)$ inside $\cD$. Then the
copolymer spends almost all of its time deep inside the $A$-blocks. Increase
$\beta$ while keeping $\alpha$ fixed. Then there will be a larger energetic
advantage for the copolymer to move some of its monomers from the $A$-blocks to
the $B$-blocks by \emph{crossing the interface inside the $AB$-block pairs}. There
is some entropy loss associated with doing so, but if $\beta$ is large enough, then
the energetic advantage will dominate, so that $AB$-localization sets in. The
value at which this happens depends on $\alpha$ and is strictly positive.
Since the entropy loss is finite, for $\alpha$ large enough the energy-entropy
competition plays out not only below the diagonal, but also below a horizontal
asymptote. On the other hand, for $\alpha$ small enough the loss of entropy
dominates the energetic advantage, which is why the critical curve has
a piece that lies on the diagonal. The larger the value of $\alpha$ the
larger the value of $\beta$ where $AB$-localization sets in. This explains
why the critical curve is non-decreasing. At the critical curve the single
interface free energy is already inside its localized phase. This explains why
the critical curve has a slope discontinuity at $\alpha^*$.

\subsection{Main results}
\label{S1.4}

In the present paper we prove three theorems, which complete the analysis of the
phase diagram in Figure \ref{fig-supcritphd}.

\bt{phtrmon}
Let $p \geq p_c$. Then $\alpha\mapsto\beta_c(\alpha)$ is strictly
increasing on
$[0,\infty)$.
\et

\bt{phtr2nd}
Let $p \geq p_c$. Then for every $\alpha \in (\alpha^*,\infty)$ there exist
$0<C_1<C_2<\infty$ and $\delta_0>0$ (depending on $p$ and $\alpha$) such that
\be{f2ndbds}
C_1\,\delta^2 \leq f\left(\alpha,\beta_c(\alpha)+\delta;p\right)
- f\left(\alpha,\beta_c(\alpha);p\right)
\leq C_2\,\delta^2 \qquad \forall\,\delta \in (0,\delta_0].
\ee
\et

\bt{locinfdiff}
Let $p \geq p_c$. Then, under Assumption {\rm \ref{assum}}, $(\alpha,\beta)\mapsto
f(\alpha,\beta;p)$ is infinitely differentiable throughout $\cL$.
\et

\noindent
Assumption \ref{assum} states that a certain intermediate single-interface free
energy has a finite curvature. We believe this assumption to be true, but have
not managed to prove it. See the end of Section \ref{Sustrconv} for a motivation
and for a way to weaken it.

Theorem \ref{phtrmon} implies that the critical curve never reaches the horizontal
asymptote, which in turn implies that $\alpha^*<\beta^*$ and that the slope in
(\ref{slopedis}) is $>0$. Theorem \ref{phtr2nd} shows that the phase transition
is \emph{second order off the diagonal}. (In contrast, we know that the phase
transition is \emph{first order on the diagonal}. Indeed, the free energy equals
$\frac12\alpha+\varpi$ on and below the diagonal segment between $(0,0)$ and
$(\alpha^*,\alpha^*)$, and equals $\frac12\beta+\varpi$ on and above this segment
as is evident from interchanging $\alpha$ and $\beta$.) Theorem \ref{locinfdiff}
tells us that the critical curve is the only location in $\CONE$ where a phase
transition of finite order occurs. Theorems \ref{phtrmon}, \ref{phtr2nd} and
\ref{locinfdiff} are proved in Sections \ref{S3}, \ref{S4} and \ref{S5}, respectively.
Their proofs rely on \emph{perturbation arguments}, in combination with \emph{exponential
tightness of the excursions away from the interface} inside the localized phase.

The analogues of Theorems \ref{phtr2nd} and \ref{locinfdiff} for the single
flat infinite interface were derived in Giacomin and Toninelli \cite{GT06a},
\cite{GT06b}. For that model the phase transition is shown to be \emph{at least
of second order}, i.e., only the quadratic upper bound is proved. Numerical
simulation indicates that the transition may well be of higher order.

The mechanisms behind the phase transition in the two models are different. While
for the single interface model the copolymer makes long excursions away from the
interface and dips below the interface during a fraction of time that is at most
of order $\delta^2$, in our emulsion model the copolymer runs along the interface
during a fraction of time that is of order $\delta$, and in doing so stays close
to the interface. Morover, because near the critical curve for the emulsion model
the single interface model is already inside its localized phase, there is a
variation of order $\delta$ in the single interface free energy. Thus, the $\delta^2$
in the emulsion model is the product of two factors $\delta$, one coming from the
time spent running along the interface and one coming from the variation of the
constituent single interface free energy away from its critical curve. See Section
\ref{S4} for more details.

In the proof of Theorem \ref{locinfdiff} we use some of the ingredients of the
proof in Giacomin and Toninelli \cite{GT06b} of the analogous result for the single
interface model. However, in the emulsion model there is an extra complication, namely,
the speed per step to move one unit of space forward may vary (because steps are
up, down and to the right), while in the single interface model this is fixed at
one (because steps are up-right and down-right). We need to control the infinite
differentiability with respect to this speed variable. This is done by considering
the Fenchel-Legendre transform of the free energy, in which the dual of the speed
variable enters into the Hamiltonian rather than in the set of paths. Moreover,
since the block pair free energies and the total free energy are both given by
variational problems, we need to show \emph{uniqueness of maximisers} and prove
\emph{non-degeneracy of the Jacobian matrix at these maximisers} in order to be
able to apply implicit function theorems. See Section \ref{S5} for more details.


\section{Preparations}
\label{S2}

In Sections \ref{S2.1}--\ref{S2.3} we recall a few key facts from den Hollander and
Whittington \cite{dHW06} that will be crucial for the proofs. Section \ref{S2.1} gives
the variational formula for the free energy, Section \ref{S2.2} states two elementary
lemmas about path entropies, while Section \ref{S2.3} states two lemmas for the
block pair free energies and a proposition \emph{characterising the localized phase
of the emulsion free energy in terms of the single interface free energy}. Section
\ref{S2.4} states a lemma about the tail behaviour of the single interface free energy
and the block pair free energies, showing that long paths wash out the effect of entropy.

\subsection{Variational formula for the free energy}
\label{S2.1}

To formulate the key variational formula for the free energy that serves as our
starting point, we need three ingredients.

\medskip\noindent
{\bf I.} For $L\in\N$ and $a \geq 2$ (with $aL$ integer), let $\cW_{aL,L}$ denote
the set of $aL$-step directed self-avoiding paths starting at $(0,0)$, ending
at $(L,L)$, and in between not leaving the two adjacent blocks of size $L$
labelled $(0,0)$ and $(-1,0)$ (see Figure \ref{fig-blockcross}). For $k,l\in\{A,B\}$,
let
\be{feklaL}
\begin{aligned}
\psi^\omega_{kl}(aL,L) &= \frac{1}{aL} \log Z^{\omega}_{aL,L},\\
Z^{\omega}_{aL,L} &= \sum_{\pi\in\cW_{aL,L}}
\exp\big[-H^{\omega,\Omega}_{aL,L}(\pi)\big]
\hbox{ when } \Omega(0,0)=k \hbox{ and } \Omega(0,-1)=l,\\
\end{aligned}
\ee
denote the free energy per step in a $kl$-block when the number of steps inside
the block is $a$ times the size of the block. Let
\be{fekl}
\lim_{L\to\infty} \psi^\omega_{kl}(aL,L) = \psi_{kl}(a)
= \psi_{kl}(\alpha,\beta;a).
\ee
Note here that $k$ labels the type of the block that is diagonally crossed, while
$l$ labels the type of the block that appears as its neighbour at the starting
corner (see Figure \ref{fig-blockcross}). We will recall in Section \ref{S2.3}
that the limit exists $\omega$-a.s.\ and in mean, and is non-random. Both $\psi_{AA}$
and $\psi_{BB}$ take on a simple form, whereas $\psi_{AB}$ and $\psi_{BA}$ do not.

\begin{figure}
\vspace{0.5cm}
\begin{center}
\setlength{\unitlength}{0.25cm}
\begin{picture}(10,10)(0,-6)
{\thicklines
\qbezier(0,6)(3,6)(6,6)
\qbezier(6,0)(6,3)(6,6)
\qbezier(0,0)(3,0)(6,0)
\qbezier(6,-6)(6,-3)(6,0)
}
\qbezier[40](0,0)(0,3)(0,6)
\qbezier[40](0,0)(0,-3)(0,-6)
\qbezier[40](0,-6)(3,-6)(6,-6)
\put(6,6){\circle*{0.5}}
\put(6,-6){\circle*{0.5}}
\put(0,0){\circle*{0.5}}
\put(-4,0){$(0,0)$}
\put(4,7){$(L,L)$}
\put(4,-7.8){$(L,-L)$}
\qbezier[45](0,0)(3,3)(6,6)
{\thicklines
\qbezier(2.5,3)(2.75,3)(3,3)
\qbezier(3,2.5)(3,2.75)(3,3)
}
\end{picture}
\end{center}
\caption{Two neighbouring blocks. The dashed line with arrow indicates that the
coarse-grained path makes a step diagonally upwards. The path enters at $(0,0)$,
exits at $(L,L)$, and in between stays confined to the two blocks.}
\label{fig-blockcross}
\end{figure}
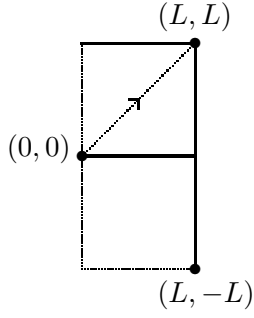

\medskip\noindent
{\bf II.} Let $\cW$ denote the class of all \emph{coarse-grained paths} $\Pi=
\{\Pi_j\colon\,j\in\N\}$ that step diagonally from corner to corner (see Figure 4,
where each dashed line with arrow denotes a single step of $\Pi$). For $n\in\N$,
$\Pi\in\cW$ and $k,l\in\{A,B\}$, let
\be{fracdef}
\rho^\Omega_{kl}(\Pi,n) = \frac{1}{n}\sum_{j=1}^n
1\,\left\{
\begin{array}{l}
\hbox{$(\Pi_{j-1},\Pi_j)$ diagonally crosses a $k$-block in $\Omega$ that has
an $l$-block}\\
\hbox{in $\Omega$ appearing as its neighbour at the starting corner}
\end{array}
\right\}.
\ee
Abbreviate
\be{rhoklp}
\rho^\Omega(\Pi,n) =
\left(\rho^\Omega_{kl}(\Pi,n)\right)_{k,l\in\{A,B\}},
\ee
which is a $2 \times 2$ matrix with non-negative elements that sum up to 1.
Let $\cR^\Omega(\Pi)$ denote the set of all limits points of the sequence
$\{\rho^\Omega(\Pi,n)\colon\,n\in\N\}$, and put
\be{Ldef}
\cR^\Omega = \mbox{the closure of the set } \bigcup_{\Pi\in\cW}
\cR^\Omega(\Pi).
\ee
Clearly, $\cR^\Omega$ exists for all $\Omega$. Moreover, since $\Omega$ has
a trivial sigma-field at infinity (i.e., all events not depending on finitely
many coordinates of $\Omega$ have probability 0 or 1) and $\cR^\Omega$ is measurable
with respect to this sigma-field, we have
\be{Rdef}
\cR^\Omega = \cR(p) \qquad \Omega-a.s.
\ee
for some \emph{non-random closed} set $\cR(p)$. This set, which depends on
the parameter $p$ controlling $\Omega$, is the set of all possible limit points
of the frequencies at which the four pairs of adjacent blocks can be seen along
an infinite coarse-grained path. The elements of $\cR(p)$ are matrices
\be{matr}
\left(\ba{ll} \rho_{AA} &\rho_{AB}\\ \rho_{BA} &\rho_{BB} \ea\right)
\ee
whose elements are non-negative and sum up to 1. In \cite{dHW06}, Proposition 3.2.1,
it was shown that $p \mapsto \cR(p)$ is continuous in the Hausdorff metric and that,
for $p \geq p_c$, $\cR(p)$ contains matrices of the form
\be{Rpa}
M_\gamma = \left(\ba{ll} 1-\gamma &\gamma\\ 0 &0 \ea\right)
\quad \mbox{ for } \gamma\in C \subset (0,1) \mbox{ closed}.
\ee

\begin{figure}
\vspace{1cm}
\begin{center}
\setlength{\unitlength}{0.3cm}
\begin{picture}(10,10)(4,0)
{\thicklines
\qbezier(0,16)(8,16)(16,16)
\qbezier(16,0)(16,8)(16,16)
\qbezier(0,0)(8,0)(16,0)
\qbezier(0,0)(0,8)(0,16)
\qbezier(4,0)(4,8)(4,16)
\qbezier(8,0)(8,8)(8,16)
\qbezier(12,0)(12,8)(12,16)
\qbezier(0,4)(8,4)(16,4)
\qbezier(0,8)(8,8)(16,8)
\qbezier(0,12)(8,12)(16,12)
}
\qbezier[30](0,4)(2,6)(4,8)
\qbezier[30](4,8)(6,10)(8,12)
\qbezier[30](8,12)(10,10)(12,8)
\qbezier[30](12,8)(14,6)(16,4)
{\thicklines
\qbezier(1.7,6)(1.85,6)(2,6)
\qbezier(2,5.7)(2,5.85)(2,6)
\qbezier(5.7,10)(5.85,10)(6,10)
\qbezier(6,9.7)(6,9.85)(6,10)
\qbezier(9.9,9.8)(10.05,9.8)(10.2,9.8)
\qbezier(10.2,10.1)(10.2,9.95)(10.2,9.8)
\qbezier(13.9,5.8)(14.05,5.8)(14.2,5.8)
\qbezier(14.2,6.1)(14.2,5.95)(14.2,5.8)
}
\put(1,14.5){$A$}
\put(5,14.5){$B$}
\put(9,14.5){$A$}
\put(13,14.5){$A$}
\put(1,10.5){$B$}
\put(5,10.5){$A$}
\put(9,10.5){$B$}
\put(13,10.5){$B$}
\put(1,6.5){$B$}
\put(5,6.5){$A$}
\put(9,6.5){$A$}
\put(13,6.5){$A$}
\put(1,2.5){$B$}
\put(5,2.5){$B$}
\put(9,2.5){$A$}
\put(13,2.5){$B$}
\put(0.05,4){\circle*{.6}}
\put(0.05,12){\circle*{.6}}
\put(4.05,0){\circle*{.6}}
\put(4.05,8){\circle*{.6}}
\put(4.05,16){\circle*{.6}}
\put(8.05,4){\circle*{.6}}
\put(8.05,12){\circle*{.6}}
\put(12.05,0){\circle*{.6}}
\put(12.05,8){\circle*{.6}}
\put(12.05,16){\circle*{.6}}
\put(16.05,4){\circle*{.6}}
\put(16.05,12){\circle*{.6}}
\end{picture}
\end{center}
\caption{$\Pi$ sampling $\Omega$. The dashed lines with arrows indicate the
steps of $\Pi$. The block pairs encountered in this example are $BB$, $AA$,
$BA$ and $AB$.}
\label{fig-coarsesampl}
\end{figure}

\medskip\noindent
{\bf III.} Let $\cA$ be the set of $2 \times 2$ matrices whose elements are $\geq 2$.
The elements of these matrices are used to record the average number of steps made
by the path inside the four block pairs divided by the block size.

\medskip
With I--III in hand, we can state the variational formula for the free energy.
Define
\be{fctV}
V\colon\,\big((\rho_{kl}),(a_{kl})\big) \in \cR(p) \times \cA
\mapsto\frac{\sum_{kl} \rho_{kl} a_{kl} \psi_{kl}(a_{kl})}
{\sum_{kl} \rho_{kl} a_{kl}}.
\ee

\bt{feiden} {\rm (\cite{dHW06}, Theorem 1.3.1)}\\
(i) For all $(\alpha,\beta)\in\R^2$ and $p \in (0,1)$,
\be{sa}
\lim_{n\to\infty} f_{n,L_n}^{\omega,\Omega} = f = f(\alpha,\beta;p)
\ee
exists $\omega,\Omega$-a.s.\ and in mean, is finite and non-random, and is
given by
\be{fevar}
f = \sup_{(\rho_{kl}) \in \cR(p)}\,\sup_{(a_{kl}) \in \cA}\,
V\big((\rho_{kl}),(a_{kl})\big).
\ee
(ii) $(\alpha,\beta)\mapsto f(\alpha,\beta;p)$ is convex on $\R^2$
for all $p \in (0,1)$.\\
(iii) $p\mapsto f(\alpha,\beta;p)$ is continuous on $(0,1)$ for all
$(\alpha,\beta)\in\R^2$.\\
(iv) For all $(\alpha,\beta)\in\R^2$ and $p \in (0,1)$,
\be{symms}
\begin{aligned}
f(\alpha,\beta;p) &= f(\beta,\alpha;1-p),\\
f(\alpha,\beta;p) &= \textstyle{\frac12}(\alpha+\beta) +
f(-\beta,-\alpha;p).
\end{aligned}
\ee
\et

\noindent
Part (iv) is the reason why without loss of generality we may restrict the parameters
to the cone in (\ref{defcone}).

The behaviour of $f$ as a function of $(\alpha,\beta)$ is different for $p \geq p_c$
and $p < p_c$, where $p_c \approx 0.64$ is \emph{the critical percolation density
for directed bond percolation on the square lattice}. The reason is that the
coarse-grained paths $\Pi$, which determine the set $\cR(p)$, sample $\Omega$
just like paths in directed bond percolation on the square lattice rotated by 45
degrees sample the percolation configuration (see Figure \ref{fig-coarsesampl}).

\subsection{Path entropies}
\label{S2.2}

The two lemmas in this section identify the path entropies associated with crossing
a block and running along an interface. They are based on straightforward computations
and are crucial for the analysis of the model.

Let
\be{DOMdef}
\DOM = \{(a,b)\colon\, a \geq 1+b, b\geq 0\}.
\ee
For $(a,b) \in \DOM$, let $N_L(a,b)$ denote the number of $aL$-step self-avoiding
directed paths from $(0,0)$ to $(bL,L)$ whose vertical displacement stays within
$(-L,L]$ ($aL$ and $bL$ are integer). Let
\be{kappa}
\kappa(a,b) = \lim_{L\to\infty} \frac{1}{aL} \log N_L(a,b).
\ee

\bl{l:ka} {\rm (\cite{dHW06}, Lemma 2.1.1)}\\
(i) $\kappa(a,b)$ exists and is finite for all $(a,b)\in\DOM$.\\
(ii) $(a,b) \mapsto a\kappa(a,b)$ is continuous and strictly concave
on $\DOM$ and analytic on the interior of $\DOM$.\\
(iii) For all $a \geq 2$,
\be{ka}
a \kappa(a,1) = \log 2 + \textstyle{\frac 12} \left[a \log a - (a-2)
\log (a-2)\right].
\ee
(iv) $\sup_{a\geq 2} \kappa(a,1) = \kappa(a^*,1) = \frac12 \log 5$ with unique maximiser
$a^* = \frac52$.\\
(v) $(\frac{\partial}{\partial a}\kappa)(a^*,1)=0$ and $a^*(\frac{\partial}{\partial b}
\kappa)(a^*,1)=\frac12\log\frac95$.\\
(vi) $(\frac{\partial^2}{\partial a^2}\kappa)(a^*,1)=-\textstyle{\frac{8}{25}}$,
$(\frac{\partial^2}{\partial b^2}\kappa)(a^*,1)=-\textstyle{\frac{262}{225}}$ and
$(\frac{\partial^2}{\partial a \partial b}\kappa)(a^*,1)=-\textstyle{\frac{2}{25}}
\log\textstyle{\frac95}+\textstyle{\frac{44}{75}}$.
\el

\noindent
Part (vi), which was not stated in \cite{dHW06}, follows from a direct computation via
\cite{dHW06}, Equations (2.1.5), (2.1.8) and (2.1.9).

For $\mu \geq 1$, let $\hat N_L(\mu)$ denote the number of $\mu L$-step
self-avoiding paths from $(0,0)$ to $(L,0)$ with no restriction on the
vertical displacement ($\mu L$ is integer). Let
\be{kappamudef}
\hat \kappa(\mu) = \lim_{L\to\infty} \frac{1}{\mu L} \log \hat
N_L(\mu).
\ee

\bl{l:kamu} {\rm (\cite{dHW06}, Lemma 2.1.2)}\\
(i) $\hat\kappa(\mu)$ exists and is finite for all $\mu\geq 1$.\\
(ii) $\mu\mapsto\mu\hat\kappa(\mu)$ is continuous and strictly concave
on $[1,\infty)$ and analytic on $(1,\infty)$.\\
(iii) $\hat\kappa(1)=0$ and $\mu\hat\kappa(\mu)\sim\log\mu$ as
$\mu\to\infty$.\\
(iv) $\sup_{\mu\geq 1} \mu[\hat\kappa(\mu)-\frac12\log 5]<\frac12\log\frac95$.
\el

\subsection{Free energies per pair of blocks}
\label{S2.3}

In this section we identify the block pair free energies. In \cite{dHW06},
Proposition 2.2.1, we showed that $\omega$-a.s.\ and in mean,
\be{psiAABB}
\psi_{AA}(a) = \textstyle{\frac12}\alpha+\kappa(a,1)
\qquad \mbox{ and } \qquad
\psi_{BB}(a) = \textstyle{\frac12}\beta+\kappa(a,1).
\ee
Both are easy expressions, because $AA$-blocks and $BB$-blocks have no interface.

To compute $\psi_{AB}(a)$ and $\psi_{BA}(a)$, we first consider the free energy
per step when the path moves in the vicinity of a \emph{single linear interface}
$\cI$ separating a liquid $A$ in the upper halfplane from a liquid $B$ in the lower
halfplane including the interface itself. To that end, for $c \geq b >0$, let
$\cW_{cL,bL}$ denote the set of $cL$-step directed self-avoiding paths starting at
$(0,0)$ and ending at $(bL,0)$. Define
\be{feinf}
\psi^{\omega,\cI}_L(c,b) = \frac{1}{cL} \log Z^{\omega,\cI}_{cL,bL}
\ee
with
\be{Zinf}
\begin{aligned}
Z^{\omega,\cI}_{cL,bL}
&= \sum_{\pi\in\cW_{cL,bL}}
\exp\left[-H^{\omega,\cI}_{cL}(\pi)\right],\\
H^{\omega,\cI}_{cL}(\pi)
&= - \sum_{i=1}^{cL}\Big(\alpha\, 1\{\omega_i=A,(\pi_{i-1},\pi_i)>0\}
+\beta\, 1\{\omega_i=B,(\pi_{i-1},\pi_i) \leq 0\}\Big),
\end{aligned}
\ee
where $(\pi_{i-1},\pi_i)>0$ means that the $i$-th step lies in the upper halfplane and
$(\pi_{i-1},\pi_i) \leq 0$ means that the $i$-th step lies in the lower halfplane or
in the interface (see Figure \ref{fig-linintdef}).

\begin{figure}
\begin{center}
\includegraphics[scale = 0.5]{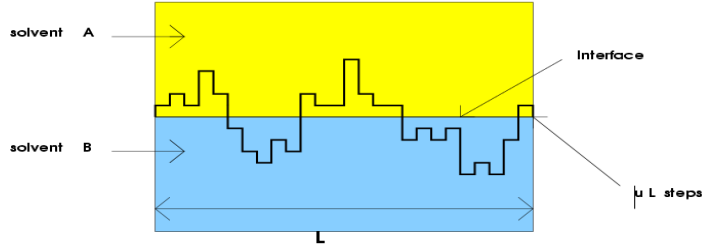}
\end{center}
\caption{Illustration of (\ref{feinf}--\ref{Zinf}) for $c=\mu$ and $b=1$.}
\label{fig-linintdef}
\end{figure}

For $a \in [2,\infty)$, let
\be{DOMadef}
\DOM(a) = \{(c,b)\in\R^2\colon\, 0\leq b \leq 1,\,c\geq b,\,a-c \geq 2-b\}.
\ee

\bl{l:feinflim} {\rm (\cite{dHW06}, Lemma 2.2.1)} For all $(\alpha,\beta)\in\R^2$
and $c \geq b > 0$,
\be{fesainf}
\lim_{L\to\infty} \psi^{\omega,\cI}_L(c,b) = \phi^{\cI}(c/b)
= \phi^{\cI}(\alpha,\beta;c/b)
\ee
exists $\omega$-a.s.\ and in mean, and is non-random.
\el

\bl{l:linkinf}{\rm (\cite{dHW06}, Lemma 2.2.2)} For all $(\alpha,\beta)\in\R^2$
and $a\geq 2$,
\be{psiinflink}
\begin{aligned}
a\psi_{AB}(a) &= a\psi_{AB}(\alpha,\beta;a)\\
&= \sup_{(c,b) \in \DOM(a)}
\big\{ c\phi^{\cI}(c/b)+(a-c)\left[\tfrac{1}{2}\alpha+\kappa(a-c,1-b)\right] \big\}.
\end{aligned}
\ee
\el

\bl{psiprop} {\rm (\cite{dHW06}, Lemma 2.2.3)} Let $k,l\in\{A,B\}$.\\
(i) For all $(\alpha,\beta)\in\R^2$, $a\mapsto a\psi_{kl}(\alpha,\beta;a)$ is
continuous and concave on $[2,\infty)$.\\
(ii) For all $a\in [2,\infty)$,
$\alpha\mapsto\psi_{kl}(\alpha,\beta;a)$ and $\beta\mapsto\psi_{kl}(\alpha,\beta;a)$
are continuous and non-decreasing on $\R$.
\el

\noindent
The idea behind Lemma \ref{l:linkinf} is that the copolymer follows the $AB$-interface
over a distance $bL$ during $cL$ steps and then wanders away from the $AB$-interface to
the diagonally opposite corner over a distance $(1-b)L$ during $(a-c)L$ steps. The optimal
strategy is obtained by maximising over $b$ and $c$ (see Figure \ref{fig-twoblockstrats}).
A similar expression holds for $\psi_{BA}$.

\begin{figure}
\vspace{1cm}
\begin{center}
\setlength{\unitlength}{0.3cm}
\begin{picture}(10,10)(4,-5)
{\thicklines
\qbezier(0,6)(3,6)(6,6)
\qbezier(6,0)(6,3)(6,6)
\qbezier(0,0)(3,0)(6,0)
\qbezier(6,-6)(6,-3)(6,0)
}
\qbezier[40](0,0)(0,3)(0,6)
\qbezier[40](0,0)(0,-3)(0,-6)
\qbezier[40](0,-6)(3,-6)(6,-6)
{\thicklines
\qbezier(2.7,3)(2.85,3)(3,3)
\qbezier(3,2.7)(3,2.85)(3,3)
\qbezier(13.15,2.9)(13.3,2.95)(13.45,3)
\qbezier(13.45,2.7)(13.45,2.85)(13.45,3)
}
\put(1.5,4.5){$A$}
\put(1.5,-1.5){$B$}
\qbezier[60](0,0)(3,3)(6,6)
{\thicklines
\qbezier(9,6)(12,6)(15,6)
\qbezier(15,0)(15,3)(15,6)
\qbezier(9,0)(12,0)(15,0)
\qbezier(15,-6)(15,-3)(15,0)
}
\qbezier[40](9,0)(9,3)(9,6)
\qbezier[40](9,0)(9,-3)(9,-6)
\qbezier[40](9,-6)(12,-6)(15,-6)
\put(10.5,4.5){$A$}
\put(10.5,-1.5){$B$}
\qbezier[50](12,0.2)(13.5,3.1)(15,6)
\qbezier[30](9,.2)(10.5,.2)(12,.2)
\put(0,0){\circle*{.5}}
\put(6,6){\circle*{.5}}
\put(6,-6){\circle*{.5}}
\put(9,0){\circle*{.5}}
\put(15,6){\circle*{.5}}
\put(15,-6){\circle*{.5}}
\end{picture}
\end{center}
\caption{Two possible strategies inside an $AB$-block: The path can either move
straight across or move along the interface for awhile and then move across.
Both strategies correspond to a coarse-grained step diagonally upwards as in
Figure \ref{fig-coarsesampl}.}
\label{fig-twoblockstrats}
\end{figure}
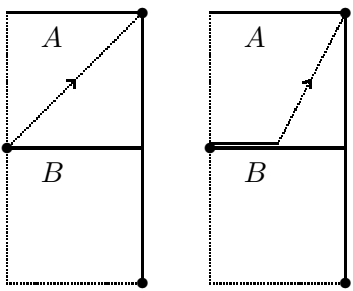

The key result behind the analysis of the critical curve in Figure \ref{fig-supcritphd}
is the following proposition, whose proof relies on Lemmas \ref{l:feinflim}--\ref{psiprop}.

\bp{p:phtrinfchar} {\rm (\cite{dHW06}, Proposition 2.3.1)}\\
Let $p\geq p_c$. Then $(\alpha,\beta)\in\cL$ if and only if
\be{phinfcr}
\sup_{\mu \geq 1}
\mu\left[\phi^{\cI}(\alpha,\beta;\mu)-\tfrac12\alpha-\tfrac12\log 5\right]
> \tfrac12\log\tfrac95.
\ee
\ep

\noindent
Note that $\frac12\alpha+\frac12\log 5$ is the free energy per step when the
copolymer diagonally crosses an $A$-block. What Proposition \ref{p:phtrinfchar}
says is that for the copolymer in the emulsion to localize, the excess free energy
of the copolymer along the interface must be sufficiently large to compensate for
the loss of entropy of the copolymer coming from the fact that it must diagonally
cross the block at a steeper angle (see Figure \ref{fig-twoblockstrats}).

We have
\be{phiublb}
\tfrac12\alpha +\hat\kappa(\mu) \leq \phi^\cI(\mu) \leq \alpha + \hat\kappa(\mu),
\ee
where $\hat\kappa(\mu)$ is the entropy defined in (\ref{kappamudef}). The upper
bound and the gap in Lemma \ref{l:kamu}(iv) are responsible for the linear piece
of the critical curve in Figure \ref{fig-supcritphd}. In analogy with Lemma \ref{l:kamu},
we note further that, for all $(\alpha,\beta)\in\R^2$, $\phi^\cI(\mu)$ is finite
for all $\mu\geq 1$, $\mu\mapsto \mu \phi^\cI(\mu)$ is continuous and concave on
$[1,\infty)$, and $\phi^{\cI}(1)=\tfrac12\beta$.

\subsection{Tail behaviour of free energies for long paths}
\label{S2.4}

In this section we show that long paths wash out the effect of entropy. This will
be needed later for compactification arguments.

Let $\P_{\mu L}^{\omega,\cI}$ denote the law of the copolymer of length $\mu L$
in the single interface model with the energy shifted by $-\frac{\alpha}{2}$,
i.e.,
\be{colawdef}
\P_{\mu L}^{\omega,\cI}(\pi) = \frac{1}{Z_{\mu L,L} ^{\omega,\cI}}\,
\exp\left[-H_{\mu L}^{\omega,\cI}(\pi)\right], \qquad \pi \in \cW_{\mu L,L},
\ee
with
\be{newham}
H^{\omega,\cI}_{\mu L}(\pi)
= - \sum_{i=1}^{\mu L}\Big(-\alpha\, 1\{\omega_i=A\}
+\beta\, 1\{\omega_i=B\}\Big)\,1\{(\pi_{i-1},\pi_i)\leq0\}.
\ee
Let
\be{freeengs}
\phi^{\cI}(\mu) = \phi^{\cI}(\alpha,\beta;\mu)
= \lim_{L\to\infty} \phi_{\mu L}^{\omega,\cI} \quad \omega-a.s.
\quad \mbox{ with } \quad
\phi_{\mu L}^{\omega,\cI} = \phi_{\mu L}^{\omega,\cI}(\alpha,\beta)
= \frac{1}{\mu L} \log Z_{\mu L,L} ^{\omega,\cI}
\ee
(compare with (\ref{Zinf})). Henceforth we adopt this shift, but we retain the same
notation. The reader must keep this in mind throughout the sequel!

\bl{mulim}
For any $\beta_0>0$,\\
(i) $\lim_{\mu\to\infty} \phi^\cI(\alpha,\beta;\mu)=0$,\\
(ii) $\lim_{a\to\infty} \psi_{AB}(\alpha,\beta;a)=0$,\\
uniformly in $\alpha\geq \beta$ and $\beta\leq\beta_0$.
\el

\bpr
(i) Recall the definition of $\cW_{\mu L,L}$ in Section \ref{S2.3}. Abbreviate
$\chi_i=1\{\omega_i=B\}-1\{\omega_i=A\}$. Because $\alpha \geq\beta$ and
$\beta\leq\beta_0$, we have
\be{fest17}
\begin{aligned}
\phi^\cI(\alpha,\beta;\mu) \leq
&\lim_{L\to \infty} \frac{1}{\mu L} \log \sum_{\pi\in \cW_{\mu L,L}}
\exp\left[\beta\sum_{i=1}^{\mu L}
\chi_i 1\{(\pi_{i-1},\pi_i)\leq 0\}\right]\\
\leq & \hat\kappa(\mu)+ \beta_0
\limsup_{L\to \infty} \frac{1}{\mu L} \max_{\pi\in \cW_{\mu L,L}}
\left\{\sum_{i=1}^{\mu L}\chi_i 1\{(\pi_{i-1},\pi_i)\leq 0\}\right\}.
\end{aligned}
\ee
We know from Lemma \ref{l:kamu}(iii) that $\lim_{\mu\to\infty}\hat\kappa(\mu)=0$.
Therefore it suffices to show that for every $\gep>0$ there exists a $\mu_0(\gep)
\geq 2$ such that
\be{eqt}
\limsup_{L\to \infty} \frac{1}{\mu L} \max_{\pi\in\cW_{\mu L,L}}
\left\{\sum_{i=1}^{\mu L}
\chi_i 1\{(\pi_{i-1},\pi_i)\leq 0\}\right\} \leq \gep \quad \omega-a.s.
\qquad\forall\,\mu\geq\mu_0(\gep).
\ee

The random variables $\chi_i$ are i.i.d.\ $\pm 1$ with probability $\frac12$.
Let $I_j$ be the set of indices $i$ in the $j$-th excursion of $\pi$ on or
below the interface. Then $\sum_{i=1}^{\mu L} \chi_i 1\{(\pi_{i-1},\pi_i)\leq 0\}
= \sum_j \sum_{i\in I_j} \chi_i$. Let $\cF_{\mu,L}$ denote the family of all
possible sequences $I=(I_j)$ as $\pi$ runs over the set $\cW_{\mu L,L}$, and
write $|I|=\sum_j |I_j|$. For $0<\gep\leq 1$, consider the quantity
\be{fest18}
p_{\mu,L,\gep}=\P\left(\exists I\in \cF_{\mu,L}\colon\,\sum_j
\sum_{i \in I_j} \chi_i \geq \gep \mu L\right),
\ee
where $\P$ denotes the probability law of $\omega$. By the Markov inequality,
there exists a $C>0$ such that
\be{eq8}
\P\left(\sum_{i=1}^N \chi_i \geq \gep RN \right) \leq e^{-C\gep^2RN}
\qquad \forall\,N,R\geq 1,\,\forall\,0 < \gep \leq 1.
\ee
Since $|I|\leq\mu L$ for all $I\in\cF_{\mu,L}$, we can apply (\ref{eq8})
with $N=|I|$ and $R=\mu L/|I|$ to estimate
\be{fest19}
p_{\mu,L,\gep}
\leq \sum_{I\in\cF_{\mu,L}} \P\left(\sum_j \sum_{i\in I_j} \chi_i
\geq \gep \frac{\mu L}{|I|}|I|\right)
\leq |\cF_{\mu,L}|\,e^{-C\gep^2\mu L}.
\ee
Since
\be{Fcardest}
|\cF_{\mu,L}| \leq \binom{\mu L}{L}^2
= \exp\left[C(\mu)L+o(L)\right] \qquad \mbox{ as } L\to\infty,
\ee
with $C(\mu)\sim\log\mu$ as $\mu\to\infty$, there exists a $C'>0$ such that,
for $\mu \geq 2$ and $L$ large enough, $|\cF_{\mu,L}| \leq \exp[LC'\log\mu]$
and hence $p_{\mu,L,\gep}\leq\exp[L(C'\log\mu-C\gep^2\mu)]$. Thus, there
exists a $\mu_0(\gep) \geq 2$ such that for $\mu\geq\mu_0(\gep)$,
\be{fest20}
\sum_{L=1}^\infty p_{\mu,L,\gep} < \infty.
\ee
The Borel-Cantelli lemma now allows us to assert that, $\omega$-a.s.\ for
$\mu\geq\mu_0(\gep)$ and $L$ large enough, the inequality $\sum_j\sum_{i\in I_j}
\chi_i \leq \gep \mu L$ holds uniformly in $I\in\cF_{\mu,L}$. Hence (\ref{eqt})
is true indeed.

\medskip\noindent
(ii) This follows from a similar argument. The counterpart of equation
\eqref{fest17} is (recall \eqref{DOMdef}-\eqref{kappa})
\be{festiv}
\psi_{AB}(\alpha,\beta;a) \leq \kappa(a,1)+ \beta_0
\limsup_{L\to \infty} \frac{1}{a L} \max_{\pi\in N_L(a,1)}
\left\{\sum_{i=1}^{a L}\chi_i 1\{(\pi_{i-1},\pi_i)\leq 0\}\right\}.
\ee
Lemma \ref{l:ka}(iii) implies that $\kappa(a,1)\to 0$ as $a\to \infty$, while the proof that $\omega$-a.s.
the second term
in the r.h.s.\ of \eqref{festiv} tends to $0$ is the same as in (i).
\epr


\section{Proof of Theorem \ref{phtrmon}}
\label{S3}

In Section \ref{S3.1} we derive a proposition stating that the excursions
away from the interface are exponentially tight in the localized phase. In
Section \ref{S3.2} we use this proposition to prove Theorem \ref{phtrmon}.

\subsection{Tightness of excursions}
\label{S3.1}

We will call the triple $(\alpha,\beta,\mu) \in \CONE \times [1,\infty)$
\emph{weakly localized} if (recall Proposition \ref{p:phtrinfchar} and
(\ref{colawdef}--\ref{freeengs}))
\be{loctriple}
\alpha \in (\alpha^*,\infty) \quad \mbox{and} \quad
\sup_{\nu \geq 1} \nu\left[\phi^\cI(\alpha,\beta;\nu)-\varpi\right]
= \mu\left[\phi^\cI(\alpha,\beta;\mu)-\varpi\right]
\geq \varsigma
\ee
with
\be{kldefs}
\varpi=\tfrac12\log 5 \qquad \hbox{and} \qquad
\varsigma=\tfrac12\log\tfrac95.
\ee
Let $l_{\mu L}$ denote the number of strictly positive excursions in $\pi\in\cW_{\mu L,L}$.
For $k=1,\dots,l_{\mu L}$, let $\tau_k$ denote the length of the $k$-th such excursion in
$\pi$.

\bp{extight}
Let $(\alpha,\beta,\mu)$ be a weakly localized triple. Then for every $C>0$ there exists an
$M_0=M_0(C)$ such that for $M\geq M_0$,
\be{extightlim}
\lim_{L\to\infty} \E \left( \P_{\mu L}^{\omega,\cI} \left(
\sum_{k=1}^{l_{\mu L}} \tau_k 1\{\tau_k\geq M\} \geq C \mu L
\right)\right) = 0.
\ee
\ep

\bpr
Along the way we need the following concentration inequality for the free energy of the
single interface. Let $\phi_{\mu L}^{\omega,\cI} = (1/\mu L)\log Z^{\omega,\cI}_{\mu L,L}$
(recall (\ref{Zinf})).

\bl{conc}
There exist $C_1,C_2>0$ such that for all $\gep>0$, $(\alpha,\beta,\mu) \in \CONE \times
[1,\infty)$ and $L \in \N$,
\be{concest}
\P\left(\big|\phi_{\mu L}^{\omega,\cI}(\alpha,\beta)-\E\left(\phi_{\mu L}^{\omega,\cI}
(\alpha,\beta)\right)\big| \geq \gep\right)
\leq C_1 \exp\left[-\gep^{2}\mu L/C_2 (\alpha+\beta)\right].
\ee
\el

\bpr
See Giacomin and Toninelli \cite{GT06b}. The argument for their single interface model
readily extends to our single interface model.
\epr

\noindent
{\bf Step 1.}
Throughout the proof, $(\alpha,\beta,\mu)$ is a weakly localized triple and $C\in (0,1)$.
Fix $M$. For $\pi\in \cW_{\mu L,L}$, we let
\be{ql}
K_L=K_L(\pi) = \{k\in\{1,\dots,l_{\mu L}\}\colon\,\tau_k\geq M\}.
\ee
We also define
\be{QRKdefs}
\begin{aligned}
\widetilde{\cW}_L &=\Big\{\pi\in\cW_{\mu L,L}\colon\,\sum_{k\in K_L}
\tau_k \geq C \mu L\Big\},\\
\cQ_L &= \{C\mu L,\dots,\mu L\} \times \{1,\dots,L\} \times
\{1,\dots,\mu L/M\}.
\end{aligned}
\ee
Note that $\widetilde{\cW}_L$ is the union of the events $(A_{s,r,t})_{(s,r,t)\in \cQ_L}$
with
\be{defi}
A_{s,r,t} = \Big\{\sum_{k\in K_L} \tau_k = s\Big\}
\cap \Big\{\sum_{k\in K_L} \tau_k/\mu_k = r \Big\}
\cap \big\{|K_L| = t\big\},
\ee
where $\mu_k$ is the number of steps divided by the number of horizontal steps
in the $k$-th strictly positive excursion. Let $v=(v_k^1,v_k^2)_{k\in K_L}$
denote the starting points and ending points of the successive positive excursions
of length $\geq M$. If $V_L$ denotes all possible values of $v$, then $A_{s,r,t}$
is the union of the events $(A_{s,r,t}^v)_{v\in V_L}$. We will estimate
$\E(\P_{\mu L}^{\omega,\cI}
(A_{s,r,t}^v))$.

\medskip\noindent
{\bf Step 2.}
We want to bound from above the quantity
\be{newq}
\E\left( \P_{\mu L}^{\omega,\cI}\left(A_{s,r,t}^v\right)\right)
= \E\left(\left(\textstyle{\sum_{\pi \in A_{s,r,t}^v}}
e^{-H^{\omega,\cI}_{\mu L}(\pi)}\right)\,
e^{-\mu L \phi_{\mu L}^{\omega,\cI}}\right).
\ee
To that end, we concatenate the excursions of $\pi$ in $[v^2_{k-1},v^1_k]$, $k\in
\{1,\dots,t\}$, as follows. Since these excursions start and end at the interface,
either with a horizontal step or with a vertical step up, we concatenate them by adding
a strictly positive excursion of 3 steps between them. The latter has no effect on the
Hamiltonian. We also concatenate the strictly positive excursions in $[v^1_k,v^2_k]$,
$k\in\{1,\dots,t\}$, by adding 1 horizontal step between them. Thus, if we abbreviate
$S_1=\mu L-s+3t$ and $S_2=L-r+t$, and if we denote by $\omega_v$ the concatenation of
the $\omega_i$ in $[v^2_{k-1},v^1_k]$, $k\in\{1,\dots,t\}$, then we have
\be{es1}
\textstyle{\sum_{\pi \in A_{s,r,t}^v}
e^{-H^{\omega,\cI}_{\mu L}(\pi)}} \leq
\textstyle{\sum_{\pi \in \cW_{S_1,S_2}}
e^{-H^{\omega_v,\cI}_{S_1}(\pi)}\
K(s+t,r+t)},
\ee
where $K(a,b)$ is the number of strictly positive excursions of length $a$ that make
$b$ horizontal steps. A standard superadditivity argument gives
\be{supaddK}
K(s+t,r+t) \leq e^{(s+t)\hat\kappa(\frac{s+t}{r+t})}
\ee
with $\hat\kappa$ the entropy function defined in (\ref{kappamudef}). Put $\hat{\mu}
=S_1/S_2$. Then with \eqref{supaddK} we can rewrite \eqref{es1} as
\be{esti3}
\textstyle{\sum_{\pi \in A_{s,r,t}^v}
e^{-H^{\omega,\cI}_{\mu L}(\pi)}} \leq
e^{S_1\; \phi^{\omega_v,\cI}_{\hat{\mu} S_2}}\,
e^{(s+t)\,\hat\kappa(\frac{s+t}{r+t})}.\\
\ee
At this stage, two cases need to be distinguished. Fix $\eta>0$.

\medskip\noindent
[Case $S_1\geq \eta L$.]
Let
\be{A1A2intr}
\begin{aligned}
A_1 &= \left\{\phi_{\mu L}^{\omega,\cI}\leq
\E\big(\phi_{\mu L}^{\omega,\cI}\big)-\gep\right\},\\
A_2 &= \big\{\phi^{\omega_v,\cI}_{\hat{\mu} S_2}\geq
\E\left(\phi_{\hat{\mu}S_2}^{\omega_v,\cI}\big)+\gep\right\}.
\end{aligned}
\ee
Since $\mu L\geq\hat{\mu}S_2=S_1\geq \eta L$, Lemma \ref{conc} gives the large
deviation inequality
\be{eq1N}
\max\{\P(A_1),\P(A_2)\} \leq C_1 \exp\left[-\gep^{2}\eta L/C_2 (\alpha+\beta)\right].
\ee
By superadditivity, we have $\E(\phi^{\omega_v,\cI}_{\hat{\mu} S_2}) \leq \sup_{L\in\N}
\E(\phi^{\omega_v,\cI}_{\hat{\mu}L}) = \phi^\cI(\hat{\mu})$. Moreover, for $L$ large
enough, we have $\E(\phi_{\mu L}^{\omega,\cI}) \geq \phi^{\cI}(\mu)-\gep$. Hence,
it follows from (\ref{esti3}--\ref{eq1N}) that
\be{esti1}
\begin{aligned}
\E\left( \P_{\mu L}^{\omega,\cI}\left(A_{s,r,t}^v\right)\right)
&= \E\left(\left(\textstyle{\sum_{\pi \in A_{s,r,t}^v}}
e^{-H^{\omega,\cI}_{\mu L}(\pi)}\right)\,
e^{-\mu L \phi_{\mu L}^{\omega,\cI}}\right)\\
&\leq \mathbb{P}(A_1)+\mathbb{P}(A_2)+\E\left(\left(\textstyle{\sum_{\pi \in A_{s,r,t}^v}}
e^{-H^{\omega,\cI}_{\mu L}(\pi)}\right)\,
e^{-\mu L \phi_{\mu L}^{\omega,\cI}}\ \ind_{A_1^c\cap A_2^c}\right)\\
&\leq 2 C_1 e^{-\gep^{2}\eta L/C_2(\alpha+\beta)}
+e^{S_1(\phi^{\cI}(\hat{\mu})+\gep)} \,
e^{-\mu L(\phi^{\cI}(\mu)-2\gep)}\, e^{(s+t)\,\hat\kappa(\frac{s+t}{r+t})}.
\end{aligned}
\ee

\medskip\noindent
[Case $S_1\leq \eta L$.]
Note that, for $(\alpha,\beta)\in \CONE$, the trivial inequality $\phi^{\omega,\cI}_{\mu L}
\leq\alpha+\hat{\kappa}(\mu)$ (compare with (\ref{phiublb})) and Lemma \ref{l:kamu} (iii)
are sufficient to assert that there exists an $R_\alpha>0$ such that $\phi^{\omega,
\cI}_{\mu L}\leq R_\alpha$ for all $\mu\geq 1$, $L\in\N$ and $\omega$. Therefore also
$\phi^{\omega_v,\cI}_{\hat{\mu}S_2}\leq R_\alpha$ for all $\hat{\mu}\geq 1$, $S_2\in\N$
and $\omega_v$, and so it follows from (\ref{esti3}--\ref{eq1N}) that
\be{esti2}
\begin{aligned}
\E\left( \P_{\mu L}^{\omega,\cI}\left(A_{s,r,t}^v\right)\right)
&= \E\left(\left(\textstyle{\sum_{\pi \in A_{s,r,t}^v}}
e^{-H^{\omega,\cI}_{\mu L}(\pi)}\right)\,
e^{-\mu L \phi_{\mu L}^{\omega,\cI}}\right)\\
&=\mathbb{P}(A_1)+\E\left(\left(\textstyle{\sum_{\pi \in A_{s,r,t}^v}}
e^{-H^{\omega,\cI}_{\mu L}(\pi)}\right)\,
e^{-\mu L \phi_{\mu L}^{\omega,\cI}}\ \ind_{A_1^c}\right)\\
&\leq C_1 e^{-\gep^{2}\mu L/C_2 \beta}
+e^{S_1 R_{\alpha}} \,
e^{-\mu L(\phi^{\cI}(\mu)-2\gep)}\, e^{(s+t)\,\hat\kappa(\frac{s+t}{r+t})}.
\end{aligned}
\ee

\medskip\noindent
{\bf Step 3.}
To bound the quantity $S_1\phi^\cI(\hat{\mu})=S_1\phi^\cI(S_1/S_2)$ in \eqref{esti1},
we define $x=s/\mu L$ and $\tilde\mu=s/r$. Then $S_1=\mu L (1-x)+3t$ and
$S_2=L(1-x\mu/\tilde\mu)+t$. Since $(\alpha,\beta,\mu)$ is a weakly localized triple
(recall (\ref{loctriple})), we have $S_1\phi^\cI(S_1/S_2)\leq \mu S_2 \phi^\cI(\mu)
+\varpi(S_1-\mu S_2)$, with $\varpi$ given in (\ref{kldefs}). This can be further
estimated by
\be{equaplus}
S_1 \phi^\cI(S_1/S_2)
\leq \mu L \phi^\cI(\mu) -\varpi x\mu L + x \frac{\mu^2}{\tilde{\mu}}
L [\varpi-\phi^\cI(\mu)]
+ t \left[\mu \phi^\cI(\mu)+\varpi(3-\mu)\right]
\ee
\be{equa2*}
\hspace{-6.2cm}\leq \mu L \phi^\cI(\mu) - \tfrac56\varpi x\mu L,
\ee
where we use that $\varpi-\phi^\cI(\mu) \leq 0$, $t \leq \mu L/M$, and $M$ is large
enough (by assumption). Next, let $\mu_0$ be such that $\hat\kappa(\nu)\leq \frac{\varpi}{2}$
for all $\nu\geq\frac{\mu_0}{2}$ (which is possible by Lemma \ref{l:kamu}(iii)).

\medskip\noindent
[Case $\tilde{\mu}\geq \mu_0$.] Since $s \geq c\mu L$
and $t\leq \mu L/M$, if $\tilde{\mu}\geq \mu_0$, then $(s+t)/(r+t)\geq \tilde{\mu}/(1+t/r)
\geq \frac{\mu_0}{2}$. Since $s+t\leq x\mu L+ \mu L/M$, it follows from (\ref{equa2*}) that for
$M$ large enough,
\be{cha2}
S_1\, \phi^\cI(S_1/S_2)+(s+t)\,
\hat\kappa\left(\frac{s+t}{r+t}\right)
\leq \mu L \phi^\cI(\mu) - \tfrac16\varpi x\mu L.
\ee

\medskip\noindent
[Case $\tilde{\mu}\leq \mu_0$.] For $\tilde{\mu}<\mu_0$, we first note that, by Lemma \ref{l:kamu}(iv)
and (\ref{loctriple}), there exists a $z>0$ such that
\be{aciter}
\sup_{y\geq 1} y[\hat\kappa(y)-\varpi] = \mu (\phi^\cI(\mu)-\varpi) - z.
\ee
Therefore, picking $y=(s+t)/(r+t)$ in (\ref{aciter}), we get
\be{equa2}
\begin{aligned}
(s+t)\hat\kappa\left(\frac{s+t}{r+t}\right)
&\leq \mu (r+t)\phi^\cI(\mu)+\varpi[(s+t)-\mu(r+t)]-z (r+t)\\
&\leq \mu r \phi^\cI(\mu)+\varpi(s-\mu r)- z r + \frac{C'L}{M}\\
&= x\frac{\mu^2 L}{\tilde{\mu}} \phi^\cI(\mu) +
\varpi x \mu L\left(1-\frac{\mu}{\tilde{\mu}}\right)
-z\frac{x\mu L}{\tilde{\mu}} + \frac{C'L}{M},
\end{aligned}
\ee
where $C'=C'(\mu)>0$ and the second line uses $t\leq \mu L/M$. Summing \eqref{equaplus}
and \eqref{equa2}, we obtain that for $M$ large enough,
\be{equa3}
S_1 \phi^\cI(S_1/S_2) + (s+t)\hat\kappa\left(\frac{s+t}{r+t}\right)
\leq \mu L \phi^\cI(\mu)-z\frac{x\mu L}{\tilde{\mu}} +\frac{C'L}{M}.
\ee
Since $x\geq C$ and $\tilde{\mu}\leq \mu_0$, we can choose $M$ large enough such that the
r.h.s.\ of \eqref{equa3} is bounded from above by $\mu L \phi^\cI(\mu)-\frac{z C}{2\tilde{\mu_0}}
\mu L$.
\medskip

\noindent
Setting $C_3=\inf\{z C/2\tilde{\mu_0},\varpi C/6\}$, we obtain that the r.h.s.\ of
\eqref{cha2} and \eqref{equa3} are both bounded from above by $\mu L \phi^\cI(\mu) -C_3\mu L$.

\medskip\noindent
{\bf Step 4.}
In the case $S_1\geq \eta L$, \eqref{esti1} becomes
\be{esti3*}
\E\left(\P_{\mu L}^{\omega,\cI}(A_{s,r,t}^v)\right)
\leq 2C_1 e^{-\gep^{2}\eta L/C_2(\alpha+\beta)} + e^{\mu L (-C_3+3\gep)},
\ee
while in the case $S_1\leq \eta L$ we choose $\eta\leq C_3/2 R_\alpha$, and \eqref{esti2}
becomes
\be{esti4}
\E\left(\P_{\mu L}^{\omega,\cI}(A_{s,r,t}^v)\right)
\leq C_1 e^{-\gep^2\mu L/C_2(\alpha+\beta)} + e^{\mu L (-\frac12C_3+2\gep)}.
\ee
Thus, there are $C_4,C_5>0$ such that, for $\gep$ small enough,
\be{estext}
\E\left(\P_{\mu L}^{\omega,\cI}(A_{s,r,t}^v)\right)\leq C_4 e^{-C_5 \mu L}.
\ee
Therefore it remains to estimate the number of possible values of $(s,r,t)$ and $v$.
Since $(s,r,t)\in\{1,\dots,\mu L\}^3$, there are at most $(\mu L)^3$ such triples. At
fixed $t$, choosing $v$ amounts to choosing $t$ starting points and $t$ ending points
for the excursions, which can be done in at most $\binom{\mu L}{2t}\leq\binom{\mu L}{2\mu L/M}$
ways when $M\geq 4$. By Stirling's formula there exists a $C''>0$ such that for all
$M \geq 4$ and $L \in \N$,
\be{stir}
\binom{\mu L}{2\mu L/M} \leq C''\sqrt{\mu L}\,e^{d(M)\mu L}
\quad \mbox{ with } \quad
d(M)=-\tfrac{2}{M}\log\left(\tfrac2M \right)-\left(1-\tfrac{2}{M}\right)
\log\left(1-\tfrac{2}{M}\right).
\ee
Since $\lim_{M\to\infty} d(M)=0$, we have $d(M)\leq C_5/2$ for some $C_5>0$ and $M$
large enough.
Therefore
\be{stircons}
\sum_{(s,r,t)\in \cQ_L} \sum_v
\E\left(\P_{\mu L}^{\omega,\cI}(A_{s,r,t}^v)\right)
\leq C_4\,C''\,(\mu L)^{7/2}\, e^{-C_5 \mu L/2}.
\ee
Since the l.h.s.\ equals the expectation in \eqref{extightlim}, we have
completed the proof.
\epr

\subsection{Proof of Theorem \ref{phtrmon}}
\label{S3.2}

The proof uses Lemma \ref{l:ka} and Proposition \ref{extight}.

\medskip\noindent
{\bf Step 1.}
Since $\alpha\mapsto\beta_c(\alpha)$ is non-decreasing and bounded from above
(by Theorem \ref{phtrcurve}(ii)), it converges to a limit $\beta^*$ as $\alpha
\to\infty$. Equation \eqref{phinfcr}, which gives a criterium for the localization
of the copolymer at $AB$-interfaces, implies that
\be{curvecrit}
\sup_{\mu \geq 1} \mu[\phi^{\cI}(\alpha,\beta_c(\alpha);\mu)-\varpi] =
\varsigma
\qquad \forall\,\alpha \geq 0
\ee
with $\varpi,\varsigma$ defined in (\ref{kldefs}) (recall the energy shift made in
(\ref{colawdef}--\ref{freeengs})). Lemma \ref{mulim} asserts that $\phi^{\cI}(\alpha,
\beta_c(\alpha);\mu)$ tends to zero as $\mu\to\infty$, uniformly in $\alpha\geq 0$.
Since $\phi^{\cI}(\alpha,\beta_c(\alpha);1)=0$ for all $\alpha>0$ (the path lies
in the interface), it follows that the supremum in \eqref{curvecrit} is attained at
some $\mu_\alpha>1$. Therefore, if we can prove that
\be{phidiff}
\phi^{\cI}(\alpha',\beta_{c}(\alpha);\mu_\alpha)
>\phi^{\cI}(\alpha,\beta_{c}(\alpha);\mu_\alpha) \qquad \forall\,\alpha>\alpha',
\ee
then
\be{ineqchain}
\sup_{\mu \geq 1}
\mu[\phi^{\cI}(\alpha',\beta_{c}(\alpha);\mu)-\varpi]
\geq
\mu_\alpha[\phi^{\cI}(\alpha',\beta_{c}(\alpha);\mu_\alpha)-\varpi]
> \mu_\alpha[\phi^{\cI}(\alpha,\beta_{c}(\alpha);\mu_\alpha)-\varpi]
=\varsigma,
\ee
and hence $\beta_{c}(\alpha)>\beta_c(\alpha')$.

\medskip\noindent
{\bf Step 2.}
Let $\alpha'>\alpha$ and
\be{diffel}
\begin{aligned}
D &= \phi^{\cI}(\alpha',\beta_{c}(\alpha);\mu_\alpha)
-\phi^{\cI}(\alpha,\beta_{c}(\alpha);\mu_\alpha)\\
&= \lim_{L\to\infty} \frac{1}{\mu_\alpha L} \left[
\log \sum_{\pi\in\cW_{\mu_\alpha L,L}} e^{-H_{\mu_\alpha
L}^{\omega,\cI}
(\alpha',\beta_c(\alpha);\pi)}
- \log \sum_{\pi\in\cW_{\mu_\alpha L,L}} e^{-H_{\mu_\alpha
L}^{\omega,\cI}
(\alpha,\beta_c(\alpha);\pi)}\right]\\
&= \lim_{L\to\infty} \frac{1}{\mu_\alpha L}
\log \E_{\mu_\alpha L}^{\omega,\cI}
\left(\exp\left[(\alpha-\alpha')\sum_{i=1}^{\mu_\alpha L}
1\{\omega_i=A,(\pi_{i-1},\pi_i)\leq 0\}\right]\right),
\end{aligned}
\ee
where the expectation is w.r.t.\ the law of the copolymer with parameters $\alpha$
and $\beta_c(\alpha)$, which are both suppressed from the notation. For $\gep>0$,
let $A_{\gep,L}=\{\pi\colon\,\sum_{i=1}^{\mu_\alpha L} 1\{\omega_i=A,(\pi_{i-1},\pi_i)
\leq 0\} \geq \gep\mu_\alpha L\}$. Then we may estimate
\be{Dest}
D \geq \limsup_{L\to\infty} \frac{1}{\mu_\alpha L} \log \Big[
e^{(\alpha-\alpha')\gep\mu_\alpha L}\, \P_{\mu_\alpha
L}^{\omega,\cI}(A_{\gep,L})
+ \P_{\mu_\alpha L}^{\omega,\cI}([A_{\gep,L}]^c)\Big].
\ee
We will prove that, for $\gep$ small enough, there is a subsequence $(L_m)_{m\in\N}$
such that $\lim_{m\to\infty}$ $\P_{\mu_\alpha L_m}^{\omega,\cI}([A_{\gep,L_m}]^c)=0$
$\omega$-a.s. This willl imply that $D\geq (\alpha-\alpha')\gep$ and complete the
proof.

\medskip\noindent
{\bf Step 3.}
We recall that $l_{\mu_\alpha L}$ denotes the number of strictly positive excursions
in $\pi\in\cW_{\mu_\alpha L,L}$. By Proposition \ref{extight}, $\omega$-a.s.,
$\P_{\mu_\alpha L}^{\omega,\cI}(\sum_{k=1}^{l_{\mu_\alpha L}}\tau_k 1\{\tau_k\geq M\}
\geq C\mu_\alpha L)$ tends to zero as $L\to\infty$ along a subsequence.
Moreover, $\omega$-a.s., $\sum_{i=1}^{\mu_\alpha L} 1\{\omega_i=A\} \geq \frac12\mu_\alpha L
- C \mu_\alpha L$ for $L$ large enough. Thus, putting $s=\frac12-2C-\gep$, for $L$ large
enough we have the inclusion
\be{Acompincl}
[A_{\gep,L}]^c \subset
\left\{\sum_{k=1}^{l_{\mu_\alpha L}} \tau_k 1\{\tau_k\geq M\} \geq
C\mu_\alpha L\right\}
\, \cup\, \Bigg\{\left\{\sum_{i=1}^{\mu_\alpha L}1\{\omega_i=A\}
1\{\Theta_i^M=1\}
\geq s\mu_\alpha L\right\} \cap [A_{\gep,L}]^c\Bigg\},
\ee
where $\Theta_i^M$ is the indicator of the event the $i$-th step lies in a strictly positive
excursion of length $\leq M$.

From now on we fix $C=\frac18$ and $\gep\leq\frac18$, implying that $s\geq \frac18$. We
also fix $M$ such that Proposition \ref{extight} holds for $C=\frac18$. The proof will
be completed once we show that
\be{Beventlim}
\lim_{L\to\infty} \P_{\mu_\alpha L}^{\omega,\cI}(B_{\gep,L}) = 0
\qquad \omega-a.s.,
\ee
where
\be{Bevendef}
B_{\gep,L} = \left\{\pi\colon\,\sum_{i=1}^{\mu_\alpha L}1\{\omega_i=A\}
1\{\Theta_i^M=1\}
\geq s\mu_\alpha L\right\}\cap [A_{\gep,L}]^c.
\ee
Each path of $B_{\gep,L}$ puts at least $s \mu_\alpha L$ monomers labelled by $A$ in
strictly positive excursions of length $\leq M$ and at most $\gep \mu_\alpha L$ monomers
labelled by $A$ in non-positive excursions.

\medskip\noindent
{\bf Step 4.}
For $\pi \in B_{\gep,L}$, let $\cE_L(\pi)$ label the excursions of $\pi$ that are strictly
positive, have length $\leq M$ and contain at least $1$ monomer labelled by $A$. Abbreviate
$r_L(\pi)=|\cE_L(\pi)| \geq s\mu_\alpha L/M$. Partition $\cE_L(\pi)$ into two parts:
\begin{itemize}
\item[--]
$\cE_L^1(\pi)$: those excursions whose preceding and subsequent
non-positive excursions
do not contain an $A$.
\item[--]
$\cE_L^2(\pi)$: those excursions whose preceding and/or subsequent
non-positive excursions
contain an $A$.
\end{itemize}
The total number of non-positive excursions containing an $A$ is bounded from above by
$\gep \mu_\alpha L$. Since a non-positive excursion can be at most once preceding and
once subsequent, we have $|\cE_L^{1}(\pi)| \geq (s/M -2\gep)\mu_\alpha L$. We will
discard the excursions in $\cE^2_L(\pi)$. Morover, to avoid overlap, we will keep from
$\cE^1_L(\pi)$ only half of the excursions. Call the remainder $\tilde\cE^1_L(\pi)$, and
abbreviate $\tilde r_L(\pi)=|\tilde\cE^1_L(\pi)|$. Then $\tilde r_L(\pi) \geq r\mu_\alpha L$
with $r= (s/2M-\gep)\mu_\alpha L$.

Next, for $\pi\in B_{\gep,L}$, let $\chi(\pi)$ denote the partition of $\{1,\dots,\mu_\alpha L\}$
into $2\tilde r_L(\pi)+1$ intervals, i.e., $(I_t)_{t=0}^{2\tilde r_L}$ with $I_{2(j-1)+1}$ ,
$j\in\{1,2,\dots,\tilde r_L\}$, the interval occupied by the $j$-th excursion of $\tilde
\cE_L^1(\pi)$ and its preceding and subsequent non-positive excursions. The partition $\chi(\pi)$
also contains $2\tilde r_L+1$ integers $(i_t)_{t=0}^{2\tilde r_L}$ with $i_t$, $i\in\{0,1,\dots,
2\tilde r_L\}$, the number of horizontal steps the path $\pi$ makes in $I_t$.

Let $K_L^\omega$ be the set of possible outcomes of $\chi(\pi)$ as $\pi$ runs over $B_{\gep,L}$.
For $\chi \in K_L^\omega$, let $t(\chi)$ denote the family of possible paths over the
even intervals $I_0,I_2,\dots,I_{2\tilde r(\chi)}$. The paths of $t(\chi)$ do not put more
than $\gep \mu_\alpha L$ monomers of type $A$ on or below the interface, put exactly one
excursion of type $1$ in each interval $I_{2j}$, $j\in\{1,\dots,2\tilde r(\chi)\}$, no excursion
of type $1$ in $I_0$ and at most one excursion in $I_{2\tilde r(\chi)}$. For $j\in\{1,\dots,
\tilde r(\chi)\}$, let $t_j(\chi)$ be the set of paths on $I_{2j-1}$ that make $i_{2j-1}$
horizontal steps, perform exactly one excursion of type $1$, and have their preceding and
subsequent non-positive excursions without an $A$. Then we have the formula
\be{PBform}
\P_{\mu_\alpha L}^{\omega,\cI} \big(B_{\gep,L}\big)
= \frac{ \sum_{\chi\in K_L^\omega} \Big[\big(\sum_{\pi' \in t(\chi)}
e^{-H^{\omega,\cI}(\pi')}\big)\,
\prod_{j=1}^{\tilde r(\chi)} \big(\sum_{\pi_j \in t_j(\chi)}
e^{-H^{\omega,\cI}(\pi_j)}\big)\Big]}{\sum_{\pi\in \cW_{\mu_\alpha L,L}
}e^{-H^{\omega,\cI}(\pi)}}.
\ee

\medskip\noindent
{\bf Step 5.}
For $j\in\{1,\dots,\tilde r(\chi)\}$, let $s_j(\chi)$ be the set of non-positive excursions
of $|I_{2j-1}|$ steps of which $i_{2j-1}$ are horizontal. Then we may estimate
\be{excursion}
\begin{aligned}
&\P_{\mu_\alpha L}^{\omega,\cI}\big(B_{\gep,L}\big)
\leq \gep \mu_\alpha L \binom{\mu_\alpha L}{\gep \mu_\alpha L}\\
&\times \frac{\sum_{\chi\in K_L^\omega}
\Big[\big(\sum_{\pi' \in t(\chi)} e^{-H^{\omega,\cI}(\pi')}\big)\,
\prod_{j=1}^{\tilde r(\chi)}
\big(\sum_{\pi_j \in t_j(\chi)} e^{-H^{\omega,\cI}(\pi_j)}\big)\Big]}
{\sum_{\chi\in K_L^\omega}
\Big[\big(\sum_{\pi' \in t(\chi)} e^{-H^{\omega,\cI}(\pi')}\big)\,
\prod_{j=1}^{\tilde r(\chi)}
\big(\sum_{\pi_j \in t_j(\chi)} e^{-H^{\omega,\cI}(\pi_j)}
+\sum_{\pi_j \in s_j(\chi)} e^{-H^{\omega,\cI}(\pi_j)}\big)\Big]}.
\end{aligned}
\ee
Here, the prefactor comes from the fact that a path with more than one non-positive
excursion containing an $A$ may be associated with more than one family $(\chi,t(\chi))$ in
the sum in the denominator of (\ref{PBform}). However, a path $t(\chi)$ cannot have
more than $\gep \mu_\alpha L$ excursions of such type. Since the number of excursions is
bounded from above by $\mu_\alpha L$, we can assert that each path can appear at most
$\gep \mu_\alpha L \binom{\mu_\alpha L}{\gep \mu_\alpha L}$ times in the denominator.

At this stage it suffices to show that there exists a $C>0$, depending only on $\alpha,\alpha'$
and $M$, such that for all $\chi\in K_L^\omega$ and $j\in\{1,\dots,\tilde r(\chi)\}$,
\be{sumHineq}
\sum_{\pi_j \in s_j(\chi)} e^{-H^{\omega,\cI}(\pi_j)} \geq C \sum_{\pi_j \in
t_j(\chi)}
e^{-H^{\omega,\cI}(\pi_j)}.
\ee
Indeed, since $r\geq \mu_\alpha L$ this yields, via (\ref{excursion}),
\be{excur*}
\P_{\mu_\alpha L}^{\omega,\cI}\big(B_{\gep,L}\big)
\leq \gep \mu_\alpha L \binom{\mu_\alpha L}{\gep \mu_\alpha L}\,\,
(1+C)^{-r\mu_\alpha L}.
\ee
For $\gep$ small enough the r.h.s.\ of \eqref{excur*} tends to zero as
$L\to\infty$ because $C>0$, implying (\ref{Beventlim}) as desired.

\medskip\noindent
{\bf Step 6.}
To prove (\ref{excur*}), we note that, since the paths of $s_j(\chi)$ stay
in the lower halfplane, their Hamiltonian is a constant, namely, $H^{\omega,\cI}
(s_j(\chi))=\sum_{i\in I_j}(\alpha 1\{\omega_i=A\}-\beta 1\{\omega_i=B\})$
(recall (\ref{newham})). A path of $t_j(\chi)$ puts at most $M$ steps of $I_j$
in the upper halfplane, and so $\pi_j \in t_j(\chi)$ implies $H^{\omega,\cI}
(\pi_j)\geq H^{\omega,\cI}(s_j(\chi))-\alpha M$. It therefore remains to
compare the cardinalities of $s_j(\chi)$ and $t_j(\chi)$. The number of
strictly positive excursions of length $\leq M$ is some integer, denoted by
$\sharp(M)$. Moreover, on $I_j$ the possible starting points of the excursion of
type $1$ are at most $M$. Indeed, the excursion has to contain all the $\omega_i$
of $I_j$ that are equal to $A$, and hence it must start less than $M$ steps to the
left of the leftmost $i\in I_j$ such that $\omega_i=A$. Thus, we have at most
$M\sharp(M)$ possible excursions of type $1$ in $I_j$ (if we take into account
their starting point). Next, we note that by fixing the starting point and the
shape of the excursions of type $1$, we can create an injection from $t_j(\chi)$ to
$s_j(\chi)$ as follows (see Figure \ref{fig-inj}). If $2r$ is the number of vertical
steps in the fixed excursion of type $1$, then we associate with each path of
$t_j(\chi)$ a path of $s_j(\chi)$ that begins with $r$ vertical steps down before
performing the preceding non-positive excursion, next makes $s$ horizontal steps,
where $s$ is the number of horizontal steps in the excursion of type $1$, next
performs the subsequent non-positive excursion, and afterwards returns to the
interface with $r$ vertical steps.

\begin{figure}
\vspace{1cm}
\begin{center}
\setlength{\unitlength}{0.35cm}
\begin{picture}(10,5)(13,-8)
\put(-2,-2.5){$0$}
\put(0,-2){\line(1,0){17}}
\put(2,-2){\circle*{.35}}
\put(8,-2){\circle*{.35}}
\put(11,-2){\circle*{.35}}
\put(14,-2){\circle*{.35}}
\put(6.9,0.1){\vector(1,-2){1}}
\put(9.9,-4.3){\vector(1,2){1}}
\put(15.1,-4.3){\vector(-1,2){1}}
\put(1.4,-3.2){{\small $b_1$}}
\put(6.4,0.5){$b_2$}
\put(9.2,-5.2){$d_1$}
\put(14.9,-5.2){$d_2$}
{\qbezier(1,-1)(1.5,-1)(2,-1)}
{\qbezier(2,-1)(2,-1.5)(2,-2)}
{\qbezier(3,-2)(3,-2.5)(3,-3)}
{\qbezier(3,-3)(4,-3)(5,-3)}
{\qbezier(5,-3)(5,-2)(5,-2)}
{\qbezier(6,-2)(6,-4)(6,-4)}
{\qbezier(6,-4)(7,-4)(7,-4)}
{\qbezier(7,-4)(7,-6)(7,-6)}
{\qbezier(7,-6)(8,-6)(8,-6)}
{\qbezier(8,-6)(8,-1)(8,-1)}
{\qbezier(8,-1)(9,-1)(9,-1)}
{\qbezier(9,-1)(10,-1)(10,-1)}
{\qbezier(10,-1)(10,1)(10,1)}
{\qbezier(10,1)(11,1)(11,1)}
{\qbezier(11,1)(11,-2)(11,-2)}
{\qbezier(11,-2)(11,-3)(11,-3)}
{\qbezier(11,-3)(12,-3)(12,-3)}
{\qbezier(12,-3)(12,-4)(12,-4)}
{\qbezier(12,-4)(13,-4)(13,-4)}
{\qbezier(13,-4)(13,-3)(13,-3)}
{\qbezier(13,-3)(14,-3)(14,-3)}
{\qbezier(14,-3)(14,-1.5)(14,-1.5)}
\put(20,-2){\line(1,0){17}}
{\qbezier(21,-1)(21.5,-1)(22,-1)}
{\qbezier(22,-1)(22,-1.5)(22,-2)}
{\qbezier(22,-2)(22,-5)(22,-5)}
{\qbezier(22,-5)(23,-5)(23,-5)}
{\qbezier(23,-5)(23,-6)(23,-6)}
{\qbezier(23,-6)(24,-6)(25,-6)}
{\qbezier(25,-6)(25,-5)(25,-5)}
{\qbezier(25,-5)(26,-5)(26,-5)}
{\qbezier(26,-5)(26,-7)(26,-7)}
{\qbezier(26,-7)(27,-7)(27,-7)}
{\qbezier(27,-7)(27,-9)(27,-9)}
{\qbezier(27,-9)(28,-9)(28,-9)}
{\qbezier(28,-9)(28,-5)(28,-5)}
{\qbezier(28,-5)(31,-5)(31,-5)}
{\qbezier(31,-5)(31,-6)(31,-6)}
{\qbezier(31,-6)(32,-6)(32,-6)}
{\qbezier(32,-6)(32,-7)(32,-7)}
{\qbezier(32,-7)(33,-7)(33,-7)}
{\qbezier(33,-7)(33,-6)(33,-6)}
{\qbezier(33,-6)(34,-6)(34,-6)}
{\qbezier(34,-6)(34,-2)(34,-1.5)}
\put(21.4,-6.2){{\small $b_1$}}
\put(28.4,-6.5){$b_2$}
\put(31.2,-4.5){$d_1$}
\put(34.5,-4.5){$d_2$}
\put(22,-5){\circle*{.35}}
\put(28,-5){\circle*{.35}}
\put(31,-5){\circle*{.35}}
\put(34,-5){\circle*{.35}}
\put(12.4,-0.6){\small $r$ steps}
\put(11.3,-0.8){\bigg \}}
\put(23.2,-3.8){\small $r$ steps}
\put(22.2,-3.8){\bigg \}}
\end{picture}
\end{center}
\caption{Injection from $t_j(\chi)$ to $s_j(\chi)$. Here, $(b_1,b_2)$ and $(d_1,d_2)$
label the endpoints of the preceding and subsequent non-positive excursions.}
\label{fig-inj}
\end{figure}

\noindent
We conclude that $|s_j(\chi)| \geq |t_j(\chi)|/M h(M)$,  which allows
us to estimate
\be{Hsumsest}
\sum_{\pi_j \in s_j(\chi)} e^{-H^{\omega,\cI}(\pi_j)}
= |s_j(\chi)|\, e^{-H^{\omega,\cI}(s_j(\chi))}
\geq  \frac{|t_j(\chi)|}{M \sharp(M)}\,e^{-H^{\omega,\cI}(s_j(\chi))}
= C \sum_{\pi_j \in t_j(\chi)} e^{-H^{\omega,\cI}(\pi_j)}
\ee
with $C=e^{-\alpha M}/M h(M)$, proving (\ref{sumHineq}).


\section{Proof of Theorem \ref{phtr2nd}}
\label{S4}

Section \ref{S4ext} states two propositions providing the lower, respectively,
upper bound for $f$ near the critical curve. These two propositions are proved
in Sections \ref{S4.2} and \ref{S4.3}, respectively, and together yield Theorem
\ref{phtr2nd}. Section \ref{Smaxext} contains several lemmas about the maximisers
of the variational problem for $\psi_{AB}$, which are needed in the proofs.

\subsection{Lower and upper bounds on the free energy}
\label{S4ext}

Recall \eqref{newham}. Fix $p\geq p_c$, $\alpha\in (\alpha^*,\infty)$ and $\delta_0>0$
small enough (depending on $p$ and $\alpha$). Abbreviate $I_0=(0,\delta_0] \cap
(0,\alpha-\beta_{c}(\alpha)]$, and for $\delta \in I_0$ define
\be{frenabbr}
\begin{aligned}
\psi_{kl}(a,\delta) &= \psi_{kl}(\alpha,\beta_c(\alpha)+\delta;a),
&a \geq 2,\\
\phi^\cI(\mu,\delta) &= \phi^\cI(\alpha,\beta_c(\alpha)+\delta;\mu),
&\mu\geq 1,
\end{aligned}
\ee
and
\be{Hdifdef}
T_{\alpha}(\delta) =
f(\alpha,\beta_c(\alpha)+\delta;p)-f(\alpha,\beta_c(\alpha);p).
\ee

\bp{p:lb}
There exists a $C_1>0$ such that
\be{lb}
T_\alpha(\delta) \geq C_1 \delta^2 \qquad \forall\,\delta\in I_0.
\ee
\ep

\bp{p:ub}
There exists a $C_2<\infty$ such that
\be{ub}
T_\alpha(\delta) \leq C_2 \delta^2 \qquad \forall\,\delta\in I_0.
\ee
\ep

\subsection{Maximisers of the block pair free energy}
\label{Smaxext}

Lemmas \ref{I}--\ref{acdellim} below are elementary assertions about the existence
and the limiting behaviour of the maximisers in the variational expression for
$\psi_{AB}$ in \eqref{psiinflink}. These lemmas will be needed in the proof of
Propositions \ref{p:lb}--\ref{p:ub} in Sections \ref{S4.2}--\ref{S4.3}.

\medskip\noindent
{\bf Step 1.}
We first show that $a\mapsto\psi_{AB}(a,\delta)$ has a maximiser for
$\delta$ small enough.

\bl{I}
For every $\delta_0>0$ there exists an $a_0>2$ such that, for every
$\alpha>\alpha^*$ and $\delta \in I_0(\alpha)$, there exists an
$a_{\alpha}(\delta)\in (2,a_0]$ satisfying
\be{eq9}
\sup_{a\geq 2}\
\psi_{AB}(a,\delta)=\psi_{AB}(a_{\alpha}(\delta),\delta).
\ee
\el

\bpr
Recall (\ref{frenabbr}). In Lemma \ref{mulim} we showed that, for every $\beta_0>0$,
$\psi_{AB}(a,\alpha,\beta)$ tends to zero as $a\to\infty$ uniformly in $\alpha\geq\beta$
and $\beta\leq \beta_0$. Since $\beta_{c}(\alpha)\leq \beta^*$ for all $\alpha\geq 0$,
there therefore exists an $a_0>2$ such that $\psi_{AB}(a,\delta)<\kappa(a^*,1)$ for
all $a\geq a_0$, $\alpha>\alpha^*$ and $\delta\in I_{0}(\alpha)$. By \cite{dHW06},
Theorem 1.4.2, we have $\sup_{a\geq 2}\ \psi_{A,B}(a,\delta)>\kappa(a^*,1)$
for all $\delta> 0$ and $\alpha>\alpha*$. This implies
\be{fest1}
\sup_{a\geq 2}\ \psi_{AB}(a,\delta)
= \sup_{2\leq a\leq a_0} \psi_{AB}(a,\delta) \qquad
\forall\,\alpha>\alpha^*,\,
\delta \in I_0(\alpha).
\ee
For $\delta$ fixed, $a\mapsto\psi_{AB}(a,\delta)$ is continuous on $[2,\infty)$ and
$\psi_{AB}(2,\delta)=0$. Therefore there exists an $a_{\alpha}(\delta)\in (2,a_0]$ such that
the l.h.s.\ of \eqref{fest1} is equal to $\psi_{A,B}(a_{\alpha}(\delta),\delta)$.
\epr

\medskip\noindent
{\bf Step 2.} Let
\be{Qdef}
\cQ^{\alpha}_{\delta,\mu_0}=\big\{(c,\mu)\colon\, 0\leq c\leq \mu,\,
\mu\geq \mu_0,\, a_{\alpha}(\delta)-c \geq 2-c/\mu\big\}
\ee
and
\be{eq2*}
H(c,a,\mu,\delta)
=\frac{1}{a}\big[c \phi^\cI(\mu,\delta)+(a-c)\kappa(a-c,1-c/\mu)\big].
\ee
Then, by Lemma \ref{l:ka}(ii), we can assert that there exists a unique pair
$(c_\alpha(\delta),\mu_{\alpha}(\delta))\in\cQ^{\alpha}_{\delta,1}$ satisfying
$\psi_{AB}(a_{\alpha}(\delta),\delta)=H(c_\alpha(\delta),a_{\alpha}(\delta),
\mu_{\alpha}(\delta),\delta)$.

\bl{II}
For every $\delta_0>0$ there exists a $\mu_0>1$ such that $(c_\alpha(\delta),
\mu_{\alpha}(\delta))\in\cQ^{\alpha}_{\delta,1}\setminus\cQ^{\alpha}_{\delta,\mu_0}$
for all $\alpha>\alpha^*$ and $\delta\in I_{0}(\alpha)$.
\el

\bpr
Prior to (\ref{fest1}) we noted that $\psi_{AB}(a_{\alpha}(\delta),\delta)>\kappa(a^*,1)$.
We will show that there exists a $\mu_0>1$ such that $H(c,a_{\alpha}(\delta),\mu,\delta)
\leq \kappa(a^*,1)$ for all $\alpha>\alpha^*$, $\delta\in I_{0}(\alpha)$ and $(c,\mu)\in
\cQ^{\alpha}_{\delta,\mu_0}$. This goes as follows. In Lemma \ref{mulim}(i) we showed that
$\phi^\cI(\mu,\delta)$ tends to zero as $\mu\to\infty$, uniformly in $\alpha>\alpha^*$
and $\delta\in I_{0}(\alpha)$. Therefore there exists a $\mu_0\geq 1$ such that $\phi^\cI
(\mu,\delta)<\frac12 \kappa(a^*,1)$ for all $\mu\geq \mu_0$, $\alpha>\alpha^*$ and
$\delta\in I_{0}(\alpha)$.

\bl{tfix}
There exists an $M>0$, depending on $a_0$, such that $\kappa(a,b)\leq \kappa(a^*,1)+M (1-b)$
for all $(a,b)\in \DOM$ (recall {\rm (\ref{DOMdef})}) satisfying $a\leq a_0$ and $\frac12
\leq b\leq 1$.
\el

\bpr
This is easily proved via Lemma \ref{l:ka}(ii), which says that $(a,b)\mapsto\kappa(a,b)$
is analytic on the interior of $\DOM$, and the equality $\kappa(a,a-1)=0$ for all $a\geq 2$.
\epr

\noindent
We now choose $\mu_0$ large enough so that $\mu>2a_0$ and $M a_0/\mu \leq \frac12\kappa(a^*,1)$.
Thus, for $(c,\mu)\in\cQ^{\alpha}_{\delta,\mu_0}$ we have $c/\mu \leq a_0/\mu_0\leq \frac12$,
which entails $\frac12\leq 1-c/\mu\leq 1$. Therefore, $(a_{\alpha}(\delta)-c,1-c/\mu)$ satisfies
the assumptions of Lemma \ref{tfix} and
\begin{equation}
\begin{aligned}
H(c,a_{\alpha}(\delta),\mu,\delta)
&\leq \frac{1}{a_\alpha(\delta)}\, \left[c\tfrac12\kappa(a^*,1)
+(a_\alpha(\delta)-c)\big(\kappa(a^*,1)+Mc/\mu\big)\right]\\
&\leq \kappa(a^*,1)+ \frac{1}{a_\alpha(\delta)}\,
c\,\left[M a_0/\mu - \tfrac12\kappa(a^*,1)\right]
\leq \kappa(a^*,1).
\end{aligned}
\end{equation}
\epr

\noindent
{\bf Step 3.}
We next show that $a\mapsto\psi_{AB}(a,0)$ has a unique maximiser.

\bl{III}
For every $\alpha\geq\alpha^*$, $\sup_{a\geq 2}\psi_{AB}(a,0)=\kappa(a^*,1)$ and is achieved
uniquely at $a=a^*$. Consequently, for $\alpha\geq\alpha^*$ and $\beta=\beta_c(\alpha)$, the
supremum in \eqref{psiinflink} is achieved uniquely at $c=0$.
\el

\bpr
Since $(\alpha,\beta_{c}(\alpha))\in \cL$, \cite{dHW06}, Theorem 1.4.2, tells us that
$\sup_{a\geq 2}\psi_{AB}(a,0)\leq \kappa(a^*,1)$. Moreover, $\psi_{AB}(a^{*},0)\geq
\kappa(a^*,1)$, and therefore
\be{int}
\sup_{a\geq 2}\psi_{AB}(a,0)=\kappa(a^*,1)=\psi_{AB}(a^*,0).
\ee
Now, pick $a\geq 2$ such that $\psi_{AB}(a,0)=\kappa(a^*,1)$ and recall that $\DOM(a)$ in
(\ref{DOMadef}) is the domain of the variational problem for $\psi_{AB}(a,0)$. We argue
by contradiction. Suppose that there exist $c,b>0$ such that $(c,b)\in \DOM(a)$ and
\be{fest9}
\psi_{AB}(a,0)=\kappa(a^*,1)=\frac{1}{a}\,
\left[c\phi^\cI(c/b,0)+(a-c)\kappa(a-c,1-b)\right].
\ee
Then
\be{eq4}
\frac{1}{a}\,\Big\{(c/b)\,\left[\phi^\cI(c/b,0)-\kappa(a^*,1)\right]-(a/b-c/b)\,
\left[\kappa(a^*,1)-\kappa(a-c,1-b)\right]\Big\}=0.
\ee
However, $(c/b)\,[\phi^\cI(c/b,0)-\kappa(a^*,1)]\leq\varsigma$ by Proposition
\ref{p:phtrinfchar}. Moreover, by \cite{dHW06}, Equation (2.3.3), we have
\be{fest10}
g(\nu)=\nu\,\left[\kappa(a^*,1)
- \sup_{2/(\nu+1)\leq b\leq 1}\kappa(b\nu,1-b)\right] > \varsigma
\qquad \forall\,\nu\geq 1.
\ee
Pick $\nu=(a-c)/b$ to make the l.h.s.\ of (\ref{eq4}) strictly negative.
Then the equality in (\ref{eq4}) cannot occur with $b>0$ and $c>0$.
Consequently, the only way to obtain (\ref{eq4}) is to take $c=0$ and $a=a^*$.
\epr

\medskip\noindent
{\bf Step 4.}
Fix $\alpha>\alpha^*$ and $\delta_0>0$. For $\delta \in I_0(\alpha)$, the quantity
$a_{\alpha}(\delta)$ may not be unique, which is why from now on we take its minimum value.
We next prove that $(a_{\alpha}(\delta),c_\alpha(\delta))$ tends to $(a^*,0)$ as
$\delta\da 0$. In what follows, $(\delta_n)_{n\in\N}$ is a sequence in $I_{0}(\alpha)$
such that $\lim_{n\to\infty}\delta_n=0$.

\bl{Psietphi1}
Let $(a_n)_{n\in\N}$ and $(\mu_n)_{n\in\N}$ be such that $\lim_{n\to\infty} a_n=a\geq 2$
and $\lim_{n\to\infty} \mu_n=\mu\geq 1$. Then $\lim_{n\to\infty} \psi_{AB}(a_n,\delta_n)
=\psi_{A,B}(a,0)$ and $\lim_{n\to\infty}\phi^\cI(\mu_{n},\delta_{n})=\phi^\cI(\mu,0)$.
\el

\bpr
A simple computation gives that $\psi_{AB}(a,\delta)-\psi_{AB}(a,0)\leq \delta$ for all
$a\geq 2$ (recall \eqref{frenabbr}). This allows us to write the inequality
\be{psietphi}
\begin{aligned}
|\psi_{AB}(a_{n},\delta_{n})-\psi_{AB}(a,0)|
&=|\psi_{AB}(a_{n},\delta_{n})-\psi_{AB}(a_{n},0)|+|\psi_{AB}(a_{n},0)-\psi_{AB}(a,0)|\\
&\leq \delta_{n}+|\psi_{AB}(a_{n},0)-\psi_{AB}(a,0)|.
\end{aligned}
\ee
Since $a\mapsto \psi_{A,B}(a,0)$ is continuous (recall Lemma \ref{psiprop}(i)), the r.h.s.\
of \eqref{psietphi} tends to zero as $n\to\infty$. This yields the claim for $\psi_{AB}$.
The same proof gives the claim for $\phi^\cI$.
\epr

\medskip\noindent
{\bf Step 5.} Finally, we obtain the convergence of  $a_\alpha(\delta)$ and $c_\alpha(\delta)$
as $\delta\downarrow 0$.

\bl{acdellim}
(i) $\lim_{\delta \downarrow 0} a_\alpha(\delta) = a^*$.\\
(ii) $\lim_{\delta \downarrow 0} c_\alpha(\delta) = 0$.
\el

\bpr
(i) The family $(a_{\alpha}(\delta))_{\delta\in I_0(\alpha)}$ is bounded. We show that the only
possible limit of its subsequences is $a^*$. Assume that $a_{\delta_{n}} \to a_{\infty}$ as
$n\to\infty$, with $a_{\infty}\in [2,a_{0}]$. Since $\delta \mapsto\psi_{A,B}(a_{\alpha}
(\delta),\delta)$ is non-decreasing, we get
\be{inter}
\psi_{AB}(a_{\delta_{n}},\delta_{n})-\psi_{AB}(a^*,0) \geq 0.
\ee
Lemma \ref{Psietphi1} tells us that the r.h.s.\ of \eqref{inter} tends to
$\psi_{AB}(a_{\infty},0)-\psi_{AB}(a^*,0)$ as $n\to\infty$. Thus, $\psi_{AB}(a_{\infty},0)
\geq \psi_{AB}(a^*,0)$ and, since  $a^*$ is the unique maximiser of $\psi_{A,B}(a,0)$
(by Lemma \ref{III}), we obtain that $a_{\infty}=a^{*}$. This implies that
$a_{\alpha}(\delta)$ tends to $a^*$ as $\delta\downarrow 0$.

\medskip\noindent
(ii) The family $(c_\alpha(\delta))_{\delta\in I_{0}}$ is bounded, because $c_\alpha
(\delta)\leq a_{\alpha}(\delta)-1\leq a_{0}-1$ for every $\delta\in I_0$. Assume that
$c_\alpha(\delta_{n}) \to c_\infty$ as $n\to\infty$. Since $a_\alpha(\delta_{n})
\to a^{*}$, we necessarily have $c_\infty\leq a^* -1$. Moreover, $(\mu_{\alpha}
(\delta_{n}))_{n\in\N}$ is bounded above by $\mu_{0}$ (by Lemma \ref{II}). Therefore,
we can pick a subsequence satisfying $\mu_{\alpha}(\delta_{n})\to \mu_{\infty}$ as
$n\to\infty$. We now recall (\ref{eq2*}) and write
\be{nouv}
\begin{aligned}
\psi_{AB}(a_\alpha(\delta_n),\delta_n)
&= \frac{1}{a_\alpha(\delta_n)}\,c_\alpha(\delta_n)
\phi^\cI(\mu_{\alpha}(\delta_n),\delta_n)\\
&\qquad + \frac{1}{a_\alpha(\delta_n)}
\Big[\big(a_{\delta_{n}}-c_\alpha(\delta_{n})\big)\,\kappa\Big(a_{\alpha}
(\delta)-c_\alpha(\delta_n),1-c_\alpha(\delta_n)/\mu\Big)\Big].
\end{aligned}
\ee
Let $n\to\infty$. Then Lemma \ref{Psietphi1} tells us that
\be{nouv1}
\psi_{AB}(a^*,0)
= \frac{1}{a^*} \Big[c_\infty \phi^\cI(\mu_{\infty},0)+
(a^*-c_\infty)\,\kappa\Big(a^*-c_\infty,1-c_\infty/\mu_{\infty}\Big)\Big].
\ee
Therefore Lemma \ref {III} gives that $c_\infty=0$ and consequently $c_\alpha(\delta)$
tends to $0$ as $\delta\da 0$.
\epr

\subsection{Proof of Proposition \ref{p:lb}}
\label{S4.2}

\bpr
Along the way we need the following. Let $\partial\phi^\cI/\partial\beta^+$ and
$\partial\phi^\cI/\partial \beta^-$ denote the right- and left-derivative of
$\phi^\cI$, respectively.

\bl{Psietphi1*}
For all $\mu\geq 1$ and $\alpha,\beta\geq 0$ such that $\phi^{\cI}(\alpha,\beta;\mu)
>\hat\kappa(\mu)$,
\be{philoc}
\frac{\partial \phi^{\cI}}{\partial \beta^{+}}(\alpha,\beta;\mu)>
\frac{\partial \phi^{\cI}}{\partial \beta^{-}}(\alpha,\beta;\mu)>0.
\ee
\el

\bpr
Use that $\phi^{\cI}(\alpha,\beta;\mu)$ is convex in $\beta$ and that
$\phi^{\cI}(\alpha,\beta;\mu) \geq\phi^{\cI}(\alpha,0;\mu)=\hat\kappa(\mu)$
for all $\beta\geq 0$.
\epr

\noindent
What Lemma \ref{Psietphi1*} says is that the localized phase of $\phi^{\cI}(\alpha,
\beta;\mu)$ for fixed $\mu$ corresponds to pairs $(\alpha,\beta)$ satisfying
$\phi^{\cI}(\alpha,\beta;\mu)> \hat\kappa(\mu)$.

\medskip\noindent
{\bf Step 1.} Recall \eqref{Rpa} and pick a $\gamma\in (0,1)$ for which $M_\gamma\in\cR(p)$.
By picking $a_{AA}=a_{AB}=a^*=\frac52$ and $(\rho_{kl})=M_\gamma$ in (\ref{fevar}), and
noting that $\psi_{AA}(a^*)=f(\alpha,\beta_c(\alpha);p)=\kappa(a^*,1)=\varpi$, we get
\be{fest14extext}
T_\alpha(\delta) \geq
\gamma\big[\psi_{AB}(a^*,\delta)-\kappa(a^*,1)\big].
\ee
Since $\mu\mapsto\phi^\cI(\mu,0)$ is continuous and $\phi^\cI(1,0)=0$, Proposition
\ref{p:phtrinfchar} allows us to choose a $\mu_\alpha\geq 1$ that is a solution of the
equation $\phi^\cI(\mu,0)=\varpi+(1/\mu)\varsigma$ (recall (\ref{kldefs})). Pick $C\in
(0,1)$ and, in the variational formula for $\psi_{AB}(a^*,\delta)$ in Lemma \ref{l:linkinf},
pick $c=C\delta$ and $c/b=\mu_\alpha$, to obtain the lower bound
\be{seq10}
T_{\alpha}(\delta) \geq \frac{\gamma}{a^*}
\Big[C \delta \phi^\cI(\mu_\alpha,\delta)+(a^*-C\delta)
\kappa\big(a^*-C\delta,1-C\delta/\mu_\alpha\big)-a^*\kappa(a^*,1)\Big].
\ee
Use Lemma \ref{l:ka}(iv-vi) to Taylor expand
\be{fest15}
\begin{aligned}
\kappa\big(a^*-C\delta,1-C\delta/\mu_\alpha\big)
&= \kappa(a^*,1)-(\varsigma/a^*)\,
C\delta/\mu_\alpha + B_\alpha C^2\delta^2\\
&\qquad + \zeta\big(C\delta,C\delta/\mu\big)\,
C^2\delta^2\left(1+1/\mu^2_\alpha\right),
\quad \delta\da 0,
\end{aligned}
\ee
for some $B_{\alpha}\in\R$ and $\zeta$ a function on $\R^2$ tending to zero at $(0,0)$. Since
$\beta_c(\alpha)\leq\beta^*$ for $\alpha\geq\alpha^*$, Lemma \ref{mulim} tells us that
$\phi^\cI(\alpha,\beta_{c}(\alpha);\mu)$ tends to $0$ as $\mu\to\infty$ uniformly in
$\alpha\geq \alpha^*$. Consequently, $\mu_{\alpha}$ is bounded uniformly in $\alpha\geq
\alpha^*$, and therefore so is $B_\alpha$. By inserting \eqref{fest15} into \eqref{seq10},
we obtain that there exist $M\in\R$ and $\delta_0>0$ such that
\be{eq11}
T_{\alpha}(\delta) \geq \frac{\gamma}{a^*}
\left[C\delta\left\{\phi^\cI(\mu_\alpha,\delta)-\phi^\cI(\mu_\alpha,0)\right\}
+ Ma^*C^2 \delta^2\right]
\qquad \forall\,\alpha>\alpha^*,\,\delta\in I_0(\alpha).
\ee
Since, by Lemma \ref{l:kamu}(iv) and Proposition \ref{p:phtrinfchar}, $\phi^\cI(\mu_\alpha,0)
>\hat\kappa(\mu_\alpha)$, Lemma \ref{Psietphi1*} gives that $(\alpha,\beta_c(\alpha))$ lies
in the localized phase of $(\alpha',\beta')\rightarrow \phi^\cI(\mu_\alpha,\alpha',\beta')$.
Therefore
\be{phider}
\phi^\cI(\mu_\alpha,\delta)-\phi^\cI(\mu_\alpha,0)
\geq C'_{\alpha}\delta \quad \mbox{ with }
\quad C'_{\alpha}=\frac{\partial \phi^{\cI}}{\partial \beta^{+}}
(\alpha,\beta_{c}(\alpha);\mu_{\alpha})\in (0,1].
\ee
Hence (\ref{eq11}) becomes
\be{fest16}
T_{\alpha}(\delta)\geq \frac{\gamma}{a^*}
(CC'_{\alpha}+Ma^*C^2)\,\delta^2
\qquad \forall \alpha> \alpha^*,\,\delta\in I_0(\alpha).
\ee
Now pick $C$ small enough so that $Ma^*C>-\frac12 C'_{\alpha}$, to get the inequality
in \eqref{lb} with $C_1=\frac{\gamma}{2a^*}CC'_{\alpha}$.

\medskip\noindent
{\bf Step 2.} To complete the proof of Proposition \ref{p:lb} it suffices to show that $C'_\alpha$
can be bounded from below by a strictly positive constant. The latter is done as follows.
Suppose that there exists a sequence $(\alpha_n)_{n\in\N}$ in $(\alpha^*,\infty]$ such that
$\lim_{n\to\infty} C'_{\alpha_{n}}=0$. By considering a subsequence of $(\alpha_n)_{n\in\N}$,
we may assume that $\alpha_n$ and $\mu_{\alpha_n}$ converge, respectively, to $\alpha_\infty
\in[\alpha^*,\infty]$ and $\mu_\infty$. Moreover, as proved in Lemma \ref{Psietphi1},
\be{eq20}
\lim_{n\to \infty} \phi^{\cI}(\alpha_n,\beta,\mu_{\alpha_n})
=\phi^{\cI}(\alpha_\infty,\beta,\mu_\infty) \qquad \forall\,\beta>0,
\ee
and $\beta\mapsto\phi^{\cI}(\alpha_n,\beta;\mu_{\alpha_n})$ is convex
for every $n\in\N$. Consequently,
\be{eq21}
\frac{\partial\phi^{\cI}}{\partial \beta^{-}}
(\alpha_\infty,\beta_{c}(\alpha_\infty);\mu_\infty)
\leq \limsup_{n\to \infty} \frac{\partial \phi^{\cI}}{\partial
\beta^{+}}
(\alpha_n,\beta_{c}(\alpha_n);\mu_{\alpha_{n}})
= \limsup_{n\to \infty} C'_{\alpha_{n}} = 0
\ee
and
\be{eq22}
\phi^{\cI}(\alpha_\infty,\beta;\mu_\infty)
= \varpi +\frac{1}{\mu_\infty}\varsigma >\hat\kappa(\mu_\infty).
\ee
But \eqref{eq21} yields $\frac{\partial \phi^{\cI}}{\partial\beta^{-}}(\alpha_\infty,
\beta_{c}(\alpha_\infty);\mu_\infty)\leq0$, which contradicts the statement in Lemma
\ref{Psietphi1}, because of \eqref{eq22}.
\epr

\subsection{Proof of Proposition \ref{p:ub}}
\label{S4.3}

{\bf Step 1.} Since $\psi_{AB}\geq \psi_{kl}$ for all $kl\in\{A,B\}^2$, we can write
\be{fest11}
f(\alpha,\beta_c(\alpha)+\delta;p)-f(\alpha,\beta_c(\alpha);p)
\leq \psi_{AB}(a_{\alpha}(\delta),\delta)-\varpi.
\ee
Because of Lemma \ref{III} we also have
\be{fest111}
f(\alpha,\beta_c(\alpha)+\delta;p)-f(\alpha,\beta_c(\alpha);p)
\leq
\psi_{AB}(a_{\alpha}(\delta),\delta)-\psi_{AB}(a_{\alpha}(\delta),0).
\ee
Since
\be{fest12}
\begin{aligned}
&\psi_{AB}(a_{\alpha}(\delta),\delta)-\psi_{AB}(a_{\alpha}(\delta),0)\\
&\qquad \leq \frac{1}{a_\alpha(\delta)}\,
\Big\{c_\alpha(\delta)
\Big[\phi^\cI\big(\mu_{\alpha}(\delta),\alpha,\beta_c(\alpha)+\delta\big)
-\phi^\cI\big(\mu_{\alpha}(\delta),\alpha,\beta_c(\alpha)\big)\Big]\Big\}
\end{aligned}
\ee
and, for $\delta$ fixed, $\beta\mapsto\phi^\cI(\alpha,\beta;\mu_{\alpha}(\delta))$ is
convex with slope bounded by 1, we obtain
\be{fest13}
\begin{aligned}
\psi_{AB}(a_{\alpha}(\delta),\delta)-\psi_{AB}(a_{\alpha}(\delta),0)
&\leq \frac{1}{a_0}\,\left[\left(\frac{\partial}{\partial \beta}\phi^\cI\right)
\big(\alpha,\beta_c(\alpha)+\delta;\mu_{\alpha}(\delta)\big)\right]\,c_\alpha(\delta)\,\delta\\
&\leq \frac{1}{a_0}\,c_\alpha(\delta)\,\delta.
\end{aligned}
\ee

\medskip\noindent
{\bf Step 2.}
The proof of \eqref{ub} is now completed by the following.

\bl{VI}
For every $\alpha>\alpha^*$ there exist $C_{\alpha}<\infty$ and $\delta_0>0$ such that
$c_\alpha(\delta)\leq C_{\alpha}\delta$ for all $\delta\in I_0(\alpha)$.
\el

\bpr
Recall the statement of Lemma \ref{II}, i.e., for every $\delta\in I_0(\alpha)$
there exists a $\mu_{\alpha}(\delta)\in [1,\mu_{0}]$ such that
\be{eq8ext}
\psi_{AB}(a_{\alpha}(\delta),\delta)
=\sup_{c\leq \min\{a_{\alpha}(\delta)-1,\mu_{\alpha}(\delta)
(a_{\alpha}(\delta)-2)/(\mu_{\alpha}(\delta)-1)\}}
H(c,a_{\alpha}(\delta),\mu_{\alpha}(\delta),\delta)
\ee
with
\be{eq7ext}
\begin{aligned}
H(c,a_{\alpha}(\delta),\mu_{\alpha}(\delta),\delta)
= \frac{1}{a_\alpha(\delta)}\,\Big[c\phi^\cI(\mu_{\alpha}(\delta),\delta)
+(a_{\alpha}(\delta)-c)\kappa\big(a_{\alpha}(\delta)-c,
1-c/\mu_{\alpha}(\delta)\big)\Big].
\end{aligned}
\ee
We proved in Lemma \ref{acdellim} that the supremum is attained in a point $c_\alpha(\delta)>0$
that tends to zero as $\delta\da 0$. Since $H$ is differentiable w.r.t.\ its first variable,
we have
\be{eq9ext}
\frac{\partial H}{\partial 1}\big(c_\alpha(\delta),a_{\alpha}(\delta),
\mu_{\alpha}(\delta),\delta\big) = 0.
\ee
Moreover, since $H$ is also differentiable w.r.t.\ its second variable, and since the maximum
of $\psi_{AB}(a,\delta)$ over $a\in [2,\infty)$ is attained in $a_{\alpha}(\delta)$, we have
\be{eq10}
\frac{\partial H}{\partial 2}\big(c_\alpha(\delta),a_{\alpha}(\delta),
\mu_{\alpha}(\delta),\delta\big) = 0.
\ee
In what follows, we consider three functions $(\delta\mapsto\xi_{i,\alpha}(\delta))_{i=1,2,3}$
that tend to zero as $\delta\da 0$. Since $a_{\alpha}(\delta)$ tends to $a^*$ by Lemma
\ref{acdellim}(i), we use the notation $a_{\alpha}(\delta)=a^*+\hat{a}_\alpha(\delta)$.
For simplicity, when we do not indicate the point at which a derivative is taken, this
point is $(a^*,1)$ by default.

Computing the derivative in \eqref{eq9ext} from (\ref{eq7ext}), we obtain a relation between
$c_\alpha(\delta)$ and $a_{\alpha}(\delta)$. We may simplify this relation by using a first
order Taylor expansion of the quantities
\be{quanTay}
\kappa\big(a_{\alpha}(\delta),1-c_\alpha(\delta)/\mu_{\alpha}(\delta)\big),\quad
\frac{\partial \kappa}{\partial 2}\big(a_{\alpha}(\delta),1-c_\alpha
(\delta)/\mu_{\alpha}(\delta)\big),\quad
\frac{\partial \kappa}{\partial 2}\big(a_{\alpha}(\delta),1-c_\alpha
(\delta)/\mu_{\alpha}(\delta)\big),
\ee
in the neighbourhood of $(a^{*},1)$. This gives, after some straightforward but tedious
computations,
\be{eq13}
\begin{aligned}
&\big[\phi^\cI(\mu_{\alpha}(\delta),\delta)-\kappa(a^{*},1)
-\textstyle{\frac{5}{2\mu_{\alpha}(\delta)}}\frac{\partial
K}{\partial2}\big]\\
&\qquad + c_\alpha(\delta)\,
A_{\alpha,\delta}+\hat{a}_\alpha(\delta)\,
B_{\alpha,\delta}+\xi_{1,\alpha}(\delta)\,(|c_\alpha(\delta)|+|\hat{a}_\alpha(\delta)|)=0
\end{aligned}
\ee
with
\be{eq13ext}
\begin{aligned}
A_{\alpha,\delta} &= \textstyle{\frac{1}{\mu_{\alpha}(\delta)}
\big[2\frac{\partial \kappa}{\partial 2}
+ 5 \frac{\partial^{2} \kappa}{\partial 1 \partial 2}
+\frac{5}{2 \mu_{\alpha}(\delta)} \frac{\partial^{2} \kappa}{\partial
2^2}
+\frac{5 \mu_{\alpha}(\delta)}{2}\frac{\partial^{2} \kappa}{\partial
2^2}\big]},\\
B_{\alpha,\delta} &=\textstyle{-\frac{1}{\mu_{\alpha}(\delta)}
\big[\frac{\partial \kappa}{\partial 2}
+\frac52 \frac{\partial^{2} \kappa}{\partial 1 \partial 2}
+\frac{5 \mu_{\alpha}(\delta)}{2} \frac{\partial^{2} \kappa}{\partial
1^2}\big]}.
\end{aligned}
\ee
The same type of computation applied to \eqref{eq10} gives
\be{eq12}
\hat{a}_\alpha(\delta)+\xi_{2,\alpha}(\delta) \hat{a}_\alpha(\delta)
=c_\alpha(\delta)C_{\alpha,\delta} + \xi_{3,\alpha}(\delta)
c_\alpha(\delta)
\ee
with
\be{Jdefext}
C_{\alpha,\delta}=\textstyle{-(\frac25)^2\,\, \frac{\kappa(a^*,1)
-\phi^\cI(\mu_{\alpha}(\delta),\delta)}{\frac{\partial^2\kappa}
{\partial 1^2}}+1+\frac{\frac{\partial\kappa^2}{\partial
1 \partial 2}}{\mu_{\alpha}(\delta) \frac{\partial^2\kappa}{\partial
1^2}}}.
\ee

Recalling that $c_\alpha(\delta)$ and $\hat{a}_\alpha(\delta)$ tend to zero as $\delta
\da 0$ (by Lemma \ref{acdellim}), we obtain from \eqref{eq12} that $\hat{a}_\alpha(\delta)
\in [(C_{\alpha,\delta} -\gep) c_\alpha(\delta),(C_{\alpha,\delta}+\gep)c_\alpha(\delta)]$
for all $\gep>0$ and $\delta$ small enough. From this last inclusion and \eqref{eq13}, we
get that there exists a $\delta_{1}>0$ such that, for all $\gep>0$ and $\delta\leq \delta_{1}$,
\be{eq14}
\big[\phi^\cI(\mu_{\alpha}(\delta),\delta)-\kappa(a^{*},1)
-\textstyle{\frac{5}{2\mu_{\alpha}(\delta)}}\frac{\partial
K}{\partial2}\big]
+c_\alpha(\delta)\, \big(A_{\alpha,\delta}+ B_{\alpha,\delta}
C_{\alpha,\delta}+\gep\big)\geq 0.
\ee
Abbreviate
\be{Deldeldef}
\Delta(\delta) = \phi^\cI(\mu_{\alpha}(\delta),\delta)
-\kappa(a^{*},1)-\textstyle{\frac{5}{2\mu_{\alpha}(\delta)}}
\frac{\partial K}{\partial2}.
\ee
Since $(\alpha,\beta_{c}(\alpha))$ lies in the delocalized region,
Proposition \ref{p:phtrinfchar} tells us that $\phi^\cI(\mu_{\alpha}(\delta),0)\leq
\kappa(a^*,1)+\textstyle{\frac{5}{2\mu_{\alpha} (\delta)}}\frac{\partial K}{\partial2}$.
Therefore we can write
\be{eq15}
\Delta(\delta)\leq
\phi^\cI(\mu_{\alpha}(\delta),\delta)-\phi^\cI(\mu_{\alpha}(\delta),0).
\ee
A simple computation gives that $\phi^\cI(\mu,\delta)-\phi^\cI(\mu,0)\leq \delta$ for all $\mu\geq 1$ (recall \eqref{frenabbr}). Hence $\Delta(\delta)\leq \delta$.

From \eqref{eq13ext} and \eqref{Jdefext}, we have
\be{eq14ext}
A_{\alpha,\delta}+
B_{\alpha,\delta}C_{\alpha,\delta}=\textstyle{\frac{A}{\mu_{\alpha}(\delta)^{2}}
+\Delta(\delta)\,\Big[\frac{B}{\mu_{\alpha}(\delta)}-\frac{2}{5}\Big]}
\ee
with
\be{eq15ext}
A=\textstyle{\frac{1}{\frac{\partial^{2}\kappa}{\partial 1^2}}
\Big[\frac52 \frac{\partial^{2}\kappa}
{\partial 2^2}-\frac25 \big(\frac{\partial \kappa}
{\partial 2}\big)^{2}-2\frac{\partial\kappa}
{\partial 2} \frac{\partial^{2}\kappa}
{\partial 1\partial 2}-\frac{5}{2}
\big(\frac{\partial^{2}\kappa}{\partial1\partial2}\big)^{2}\Big]}
\quad \text{and} \quad
B=\textstyle{\frac{1}{\frac{\partial^{2} \kappa}{\partial 1^2}}
\Big[-\big(\frac25\big)^{2}\frac{\partial \kappa}{\partial2}-\frac25
\frac{\partial^{2} \kappa}{\partial 1\partial 2}\Big]}.
\ee
By inserting the values of the derivatives given in Lemma \ref{l:ka}(v--vi),
we find that $A<0$. Thus, recalling that $1\leq\mu_{\alpha}(\delta)\leq \mu_{0}$
for all $\delta\in I_{0}(\alpha)$ (by Lemma \ref{II}), we can rewrite \eqref{eq14}
as
\be{eq16}
A_{\alpha,\delta}+
B_{\alpha,\delta}C_{\alpha,\delta}\leq\textstyle{\frac{A}{\mu_{0}^{2}}
+\Delta(\delta)\,\big[|B|+\frac{2}{5}\big]}.
\ee
Since $\Delta(\delta)\leq \delta$, we can now assert that there exists a $\delta_{2}>0$
such that $0<\delta\leq \delta_{2}$ implies $A_{\alpha,\delta}+B_{\alpha,\delta}
C_{\alpha,\delta}\leq 3A/2 \mu_0^2$. Therefore \eqref{eq14} becomes $\delta
+c_\alpha(\delta)\,3A/2 \mu_{0}^{2}\geq 0$ and, consequently, for $\delta_0
=\min\{\delta_1,\delta_2\}$ there exists a $C_\alpha>0$ such that for
all $\delta\in I_{0}(\alpha)$,
\be{eq17}
c_\alpha(\delta)\leq C_\alpha \delta.
\ee
This completes the proof of Lemma \ref{VI}.
\epr


\section{Proof of Theorem \ref{locinfdiff}}
\label{S5}

In Section \ref{S5.1} we study a variation of the single linear interface model in
which the variable $\mu$ is replaced by a dual variable $\lambda$, which enters into
the Hamiltonian rather than in the set of paths. We show that the free energy for
this dual model is smooth. In Section \ref{Sustrconv} we show that the dual free
energy has a non-zero curvature. In Sections \ref{S5.2} and \ref{S5.3} we use this
to prove that $\phi^\cI$ and $\psi_{AB}$ are smooth on their localized phases and
have a non-zero curvature too. The latter in turn are used in Section \ref{S5.4} to
prove the smoothness of $f$ on $\cL$. Key ingredients in the proofs are the implicit
function theorem, the exponential tightness of the excursions in the localized phases,
and the uniqueness of the maximisers in the variational formulas for $\phi^\cI$,
$\psi_{AB}$ and $f$.

\subsection{Fenchel-Legendre transform of $\phi^\cI$}
\label{S5.1}

We begin by defining the dual of the single interface model. Let $\cW_L$ be the set
of $L$-step directed self-avoiding paths that start at $(0,0)$ and end at $(x,0)$ for
some $x\in\{1,\dots,L\}$. For $\pi\in\cW_L$, let $h(\pi)$ be the number of horizontal
steps in $\pi$. For $\lambda \geq 0$, define (recall (\ref{newham}))
\begin{equation}
\label{newmod}
\begin{aligned}
U_L^{\omega,\cI}(\alpha,\beta;\lambda)
&= \sum_{\pi\in \cW_L} e^{-\lambda h(\pi) - H_L^{\omega,\cI}(\pi)}\\
u^\cI(\alpha,\beta;\lambda)
&= \lim_{L\to \infty} \frac{1}{L}\log U_L^{\omega,\cI}(\alpha,\beta;\lambda)
\quad \omega-a.s.
\end{aligned}
\end{equation}
and
\begin{equation}\label{equa4}
\tilde{\kappa}(\lambda)=\lim_{L\to \infty} \frac{1}{L}\log \sum_{\pi\in \cW_L}
e^{-\lambda h(\pi)}.
\end{equation}
The convergence $\omega$-a.s.\ and in mean and the constantness $\omega$-a.s.\ of
$u^\cI(\alpha,\beta;\lambda)$ follow from the subadditive ergodic theorem (Kingman
\cite{King}). Set
\begin{equation}
\cL_u = \big\{(\alpha,\beta,\lambda) = \CONE \times [0,\infty)\colon\,
u^\cI(\alpha,\beta;\lambda)>\tilde{\kappa}(\lambda)\big\},
\end{equation}
i.e., the region where the dual of the single linear interface model is localized.

\begin{proposition}
\label{uidiff}
The function $(\alpha,\beta,\lambda)\mapsto u^\cI(\alpha,\beta;\lambda)$ is infinitely
differentiable on $\cL_u$.
\end{proposition}

\bpr
The proof is similar to that of the infinite differentiability of the free energy
for the single interface model, proved in Giacomin and Toninelli \cite{GT06b}.
Therefore, we only sketch the main steps in the proof and refer to \cite{GT06b}
for further details.

\medskip\noindent
{\bf Step 1.} The claim follows from the Arzela-Ascoli theorem as soon as we prove
that for all $(\alpha_0,\beta_0,\lambda_0)\in\cL_u$ there exists $\cV\subset \cL_u$ a neighborhood of
$(\alpha_0,\beta_0,\lambda_0)$ such that for all $k\in\mathbb{N}$,
the $k$-th derivative of $L^{-1}\E(\log U_L^{\omega,\cI}(\alpha,
\beta;\lambda))$ w.r.t.\ any of the parameters $\alpha,\beta,\lambda$ is bounded
uniformly in $L$ and $(\alpha,\beta,\lambda)\in\cV$, where $\E$ denotes expectation w.r.t.\ $\omega$.

For $a,b \in\N$ with $a<b$, let $\mathcal{H}_{a,b}$ be the set of bounded functions
that are measurable w.r.t.\ the $\sigma$-algebra $\sigma(\pi_j\colon\,j\in
\{a,\dots,b\})$. As explained in \cite{GT06b}, the conditions of the Arzela-Ascoli
theorem are satisfied once we show that for all $(\alpha_0,\beta_0,\lambda_0)\in \cL_u$
there exist $C_1,C_2>0$ and $\cV\subset \cL_u$ such that, for all $a_1,b_1,a_2,b_2\in\N$ with
$a_1<b_1<a_2<b_2\leq L$ and
$(f_1,f_2)\in \cH_{a_1,b_1} \times \cH_{a_2,b_2}$ and $(\alpha,\beta,\lambda)\in \cV$,
the following inequality holds:
\begin{equation}
\label{correq}
\E\Big(E_L^{\omega,\cI}(f_1 f_2)-E_L^{\omega,\cI}(f_1)E_L^{\omega,\cI}(f_2)\Big)\leq C_1\,
\|f_1\|_\infty\, \|f_2\|_\infty\, e^{-C_2 (a_2-b_1)}.
\end{equation}
Here, $E_L^{\omega,\cI}$ is expectation w.r.t.\ the law of the $L$-step copolymer at
fixed $\omega$ given by (recall (\ref{newmod}))
\be{newmodlaw}
P_L^{\omega,\cI}(\pi) = \frac{1}{U_L^{\omega,\cI}}\,
e^{-\lambda h(\pi) - H_L^{\omega,\cI}(\pi)}.
\ee

\begin{figure}
\vspace{1cm}
\begin{center}
\setlength{\unitlength}{0.35cm}
\begin{picture}(7,5)(10,-8)
\put(1,-1){\circle*{.35}}
\put(3,-4){\circle*{.35}}
\put(11,1){\circle*{.35}}
\put(16,-1){\circle*{.35}}
\put(0.4,-2.2){{\small $a$}}
\put(2.2,-4.8){{\small $a$}}
\put(11.65,1){{\small $b$}}
\put(16.5,-0.5){{\small $b$}}
{\qbezier(1,-1)(1.5,-1)(2,-1)}
{\qbezier(2,-1)(2,-1.5)(2,-2)}
{\qbezier(2,-2)(3,-2)(3,-2)}
{\qbezier(3,-2)(3,-2.5)(3,-3)}
{\qbezier(3,-3)(4,-3)(5,-3)}
{\qbezier(5,-3)(5,-2)(5,-2)}
{\qbezier(5,-2)(6,-2)(6,-2)}
{\qbezier(6,-2)(6,-4)(6,-4)}
{\qbezier(6,-4)(7,-4)(7,-4)}
{\qbezier(7,-4)(7,-6)(7,-6)}
{\qbezier(7,-6)(8,-6)(8,-6)}
{\qbezier(8,-6)(8,-1)(8,-1)}
{\qbezier(8,-1)(9,-1)(9,-1)}
{\qbezier(9,-1)(10,-1)(10,-1)}
{\qbezier(10,-1)(10,1)(10,1)}
{\qbezier(10,1)(11,1)(11,1)}
{\qbezier[10](3,-5)(3,-4.5)(3,-4)}
{\qbezier[20](3,-5)(4,-5)(5,-5)}
{\qbezier[20](5,-5)(5,-6)(5,-7)}
{\qbezier[20](5,-7)(6,-7)(7,-7)}
{\qbezier[10](7,-7)(7,-6.5)(7,-6)}
{\qbezier[20](7,-6)(8,-6)(9,-6)}
{\qbezier[20](7,-6)(8,-6)(9,-6)}
{\qbezier[20](9,-4)(9,-5)(9,-6)}
{\qbezier[10](9,-4)(9.5,-4)(10,-4)}
{\qbezier[10](10,-4)(10,-3.5)(10,-3)}
{\qbezier[20](10,-3)(11,-3)(12,-3)}
{\qbezier[30](12,-3)(12,-1.5)(12,0)}
{\qbezier[20](12,0)(13,0)(14,0)}
{\qbezier[10](14,0)(14,-0.5)(14,-1)}
{\qbezier[20](14,-1)(15,-1)(16,-1)}
{\qbezier(7.8,-3)(8.1,-3)(8.2,-3)}
{\qbezier(7.8,-2)(8.1,-2)(8.2,-2)}
\put(6.8,-2.8){ \{}
{\qbezier(7.2,-2.3)(6.3,-0.5)(5.8,0.5)}
{\qbezier(5.7,0.1)(5.8,0.5)(5.8,0.5)}
{\qbezier(6.15,0.4)(5.8,0.5)(5.8,0.5)}
\put(11.9,-2.8){ \}}
{\qbezier(11.8,-2)(12,-2)(12.2,-2)}
{\qbezier(20,-5)(21,-5)(22,-5)}
{\qbezier[10](20,-7)(21,-7)(22,-7)}
\put(22,-5){\small\  :\ $\pi_1$}
\put(22,-7){\small\ :\ $\pi_2$}
\put(2.7,0.8){\small $j$-step of $\pi_1$}
{\qbezier(12.8,-2.55)(14,-2.55)(16,-2.55)}
{\qbezier(15.7,-2.7)(16,-2.55)(16,-2.55)}
{\qbezier(15.7,-2.4)(16,-2.55)(16,-2.55)}
\put(16.4,-2.8){\small $j$-th step of $\pi_2$}
\end{picture}
\end{center}
\caption{A pair of paths $(\pi_1,\pi_2)$ whose $j$-th steps are the same
and occur at the same height.}
\label{figbab1}
\end{figure}
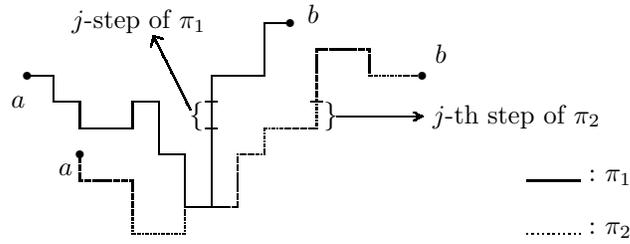

Next, the correlation inequality in (\ref{correq}) will follow once we show that there exist
$C_1,C_2>0$ and $\cV\subset \cL_u$ (depending on $\alpha_0,\beta_0,\lambda_0$) such that, for all $a,b,L\in\N$ with
$a \leq b\leq L$, we have
\begin{equation}
\label{pivest}
\E\big([P_L^{\omega,\cI}]^{\otimes 2}(B_{a,b})\big) \leq C_1 e^{-C_2 (b-a)},
\end{equation}
where $[P_L^{\omega,\cI}]^{\otimes 2}$ is the joint law of two independent copies of
the $L$-step copolymer at fixed $\omega$, and
\begin{equation}
\begin{aligned}
B_{a,b}
&= \{(\pi^1,\pi^2)\colon\, \nexists\, j\in \{a,\dots,b\}
\text{ such that the $j$-th steps}\\
&\qquad \text{of $\pi_1$ and $\pi_2$ are the same and
occur at the same height}\}.
\end{aligned}
\end{equation}
Indeed, on $[B_{a,b}]^c$ the two paths can be coupled as soon as they make the common
step. An example of a pair of paths $(\pi_1,\pi_2)$ not in $B_{a,b}$ is displayed in Figure
\ref{figbab1}.

\medskip\noindent
{\bf Step 2.} For $i=1,2$ and $M\in\mathbb{N}$, let $l_{i,M}$ be the number of excursions
of $\pi_i$ (either strictly positive or non-positive) that are included in $\{a,\dots,b\}$
and are smaller than or equal to $M$. Let
\begin{equation}
\cE_M(\pi_i)=\{(b^i_1,e^i_1),\dots,(b^i_{l_{i,M}},
e^i_{l_{i,M}})\},
\end{equation}
where $(b^i_j,e^i_j)$ denote the end-steps of the $j$-th excursion. Put $\tau_j^i
=e^i_j-b^i_j+1$, and for $\gamma\in(0,1)$ let
\begin{equation}
\label{AiMdef}
\cA_{i,\gamma,M} = \left\{\pi_i\colon\,\sum_{j=1}^{l_{i,M}} \tau^i_j
\geq \gamma(b-a)\right\}.
\end{equation}

\begin{lemma}
\label{lemM}
(i) For all $\gamma_0\in (0,1)$ and $(\alpha_0,\beta_0,\lambda_0)\in \cL_u$ there
exist $M\in\N$, an open neighborhood $\cV$ of\, $(\alpha_0,\beta_0,\lambda_0)$ in
$\cL_u$ and $C_1,C_2>0$ such that, for $L \geq b \geq a$ and $(\alpha,\beta,\lambda)\in \cV$,
\be{tight1}
\nonumber \E\left(\P_L^{\omega,\cI}(\cA_{i,\gamma_0,M})\right) \geq 1-C_1 e^{-C_2(b-a)},
\qquad i=1,2.
\ee
(ii) For all $T_0\in\mathbb{N}$
and $(\alpha_0,\beta_0,\lambda_0)\in \cL_u$ there exist $\gamma\in(0,1)$, an open neighborhood $\cV$
of $(\alpha_0,\beta_0,\lambda_0)$ in $\cL_u$ and $C_1,C_2>0$ such that,
for all $L \geq b \geq a$ and $(\alpha,\beta,\lambda)\in \cV$,
\be{tight2}
\nonumber\E\left(\P_L^{\omega,\cI}(\cA_{i,\gamma,T_0})\right) \leq C_1 e^{-C_2 (b-a)},
\qquad i=1,2.
\ee
\end{lemma}

\bpr
(i) This part gives the exponential tightness of the excursions of the copolymer in the
localized phase. Compared to Proposition \ref{extight}, both the model and the statement
are different. However, the same tools can be used and for this reason we only give
a sketch of the proof. By the definition of $\cA_{i,\gamma_0,M}^c$, there are two cases.

\noindent[Case 1]
The sum of the lengths of the strictly positive excursions larger than $M$ in $\{a,\dots,b\}$
is $\geq \gamma \tfrac{b-a}{2}$.

\noindent[Case 2]
The sum of the lengths of the non-positive excursions larger than $M$ in $\{a,\dots,b\}$
is $\geq\gamma \tfrac{b-a}{2}$.

\noindent
In Case 1, by concatenating the strictly positive excursions larger than $M$ in $\{a,\dots,b\}$,
we can bound the total entropy carried by these excursions from above by the entropy of a single
excursion large at least $\gamma \tfrac{b-a}{2}$. Therefore, the gain in the free energy
obtained by relaxing this large excursion is, for $b-a$ large enough, of order $\exp[C_2(b-a)]$,
with $C_2=\frac{\gamma}{2}[u(\lambda)-\tilde{\kappa}(\lambda)]$. By choosing a small enough open
neighborhood $\cV$ of $(\alpha_0,\beta_0,\lambda_0)$ in $\cL_u$, we get that there exists a $c>0$
such that, for all $(\alpha,\beta,\lambda)\in \cV$, we have $u(\alpha,\beta;\lambda)
-\tilde{\kappa}(\lambda)\geq c$. Thus, $\frac{c\gamma}{2}$ is a lower bound for $C_2$,
uniform in $\cV$. In Case 2, a similar argument applies.

\medskip\noindent
(ii) Again we only sketch the proof. We partition $\{a,\dots,b\}$ into $\frac{b-a}{R}$ blocks
of size $R$. A block is called "good" if it carries only monomers of type $A$. By the law of
large numbers, there exists a $c_R>0$ such that approximately $c_R(b-a)$ of the blocks are good.
We can therefore choose $\gamma$ close enough to $1$ such that, on $\cA_{1,\gamma,T}$, at least
$\frac{c_R}{2}(b-a)$ of the good blocks are covered only by excursions smaller than $T$. Such
blocks are called "good $T$-blocks". Consequently, more than $\frac{R}{T}$ excursions are
required to cover a good $T$-block and so at least $\frac{R}{T}$ steps in each good $T$-block
are below the interface. Thus, by relaxing the condition $\cA_{1,\gamma,T}$, we can replace on
each good $T$-block the excursions shorter than $T$ by a long strictly positive excursion. This
does not decrease the entropy, but increases the energy by at least $\beta \frac{R}{T}$ on each good
$T$-block. Summed up these energy increases are of order $\frac{c_R}{2}(b-a) \beta \frac{R}{T}$.
\epr

\noindent
{\bf Step 3.} Let $D = \cA_{1,\frac34,M} \cap \cA_{2,\frac34,M}$ and $\cT_M=\{\cE_M(\pi_1)\colon\,\pi_1
\in \cA_{1,\frac34,M}\}$. For $i=1,2$ and $\cE_M\in\cT_M$, let $\cJ^i(\cE_M)=\{\pi_i\colon
\cE_M(\pi_i)=\cE_M\}$. Then Lemma \ref{lemM} applied at $\gamma_0=\frac34$ implies
that there exists $M\in\N$, an open neighborhood $\cV$
of $(\alpha_0,\beta_0,\lambda_0)$ in $\cL_u$ and $C_1, C_2>0$  such that for $L\geq b$ and
$(\alpha,\beta,\lambda)\in \cV$ we have $[P_L^{\omega,\cI}]^{\otimes 2}
(D^c)\leq 2C_1 e^{-C_2(b-a)}$, so that it remains to estimate $[P_L^{\omega,
\cI}]^{\otimes 2}(B_{a,b}\cap D)$.
\be{pn2}
\begin{aligned}
&[P_L^{\omega,\cI}]^{\otimes 2}(B_{a,b} \cap D)
=\sum_{\cE^1_M,\cE^2_M\in \cT_M}
[P_L^{\omega,\cI}]^{\otimes 2}\Big(B_{a,b} \cap \{\cJ^1(\cE^1_M) \times \cJ^2(\cE^2_M)\}\Big)\\
&=\sum_{\cE^1_M,\cE^2_M\in \cT_M}
E_L^{\omega,\cI}\Big(1_{\{\pi_2\in\cJ^2(\cE^2_M)\}}\
P_L^{\omega,\cI}\big( B_{a,b}\cap\{\pi_1\in \cJ^1(\cE^1_M)\} \mid \pi_2\big)\Big).
\end{aligned}
\ee
Next, set $\tilde{i}=2$ if $i=1$ and vice versa, and define
\begin{equation}
\cR^i(\cE^1_M,\cE^2_M) = \Big\{j\in\{1,\dots,l_{i,M}\}\colon\,
b_k^{\tilde{i}}\ \text{or}\  e_k^{\tilde{i}} \in\{b_j^i,e_j^i\}
\mbox{ for some }  k\in\{1,\dots,l_{\tilde{i},M}\}\Big\}.
\end{equation}
By the definition of $\mathcal{A}_{i,\frac34,M}$ in (\ref{AiMdef}), for any $\cE^1_M,\cE^2_M\in
\cT_M$ there are at least $\frac14(b-a)$ steps in $\{a,\dots,b\}$ belonging to excursions
smaller than $M$, in both $\pi_1$ and $\pi_2$. Therefore we can choose a $C>0$ small
enough such that, for all $\cE^1_M,\cE^2_M\in \cT_M$, either $|\cR^i(\cE^1_M,\cE^2_M)|
\geq C (b-a)/M$ or $|\cR^{\tilde{i}}(\cE^1_M,\cE^2_M)|\geq C (b-a)/M$. Without loss
of generality, we may assume that $|\cR^1(\cE^1_M,\cE^2_M)|\geq C (b-a)/M$. Because of the
condition imposed by $B_{a,b}$, for all $j\in\cR^1(\cE^1_M,\cE^2_M)$ the excursion of $\pi_1$
on $\{b_j^1,\dots,e_j^1\}$ has some prohibited parts. Indeed, $\pi_2$ starts or ends
an excursion inside $\{b_j^1,\dots,e_j^1\}$, which restricts the possible excursions
of $\pi_1$, because $\pi_1$ cannot make the same step as $\pi_2$ at the same height.
Moreover, there is only a finite number of possibilities to make an excursion smaller
than $M$ and so, for all $j\in \cR^1(\cE^1_M,\cE^2_M)$, relaxing the condition $B_{a,b}$
on $\{b_j^1,\dots,e_j^1\}$ amounts to increasing the probability in \eqref{pn2} by a
factor $Q>1$ depending only on $M$, i.e.,
\be{losfe}
P_L^{\omega,\cI}\Big(B_{a,b}\cap\{\pi_1\in \cJ^1(\cE^1_M)\} \mid \pi_2\Big)
\leq Q^{-|\cR^1(\cE^1_M,\cE^2_M)|}\,P_L^{\omega,\cI}\big(\{\pi_1\in \cJ^1(\cE^1_M)\}\big).
\ee
Therefore, since $|\cR^1(\cE^1_M,\cE^2_M)|\geq C(b-a)/M$, (\ref{pn2}) becomes
\begin{equation}
[P^{\omega,\cI}]^{\otimes 2}(B_{a,b} \cap D) \leq  e^{-C \frac{b-a}{M} \log Q},
\end{equation}
which proves (\ref{pivest}) and completes the proof of Proposition \ref{uidiff}.
\epr

The following proposition provides the link between $u^\cI$ and $\phi^\cI$.

\begin{proposition}
\label{p:varform}
For $\lambda\geq 0$,
\begin{equation}
\label{eqa}
u^\cI(\lambda)=\sup_{\rho\in (0,1]} \{-\lambda \rho+\phi^\cI(1/\rho)\}.
\end{equation}
\end{proposition}

\bpr
For $\rho\in(0,1]$, let $\cW_L(\rho)=\{\pi\in \cW_L\colon\,h(\pi)=\rho L\}$ and
\begin{equation}
U_L^{\omega,\cI}(\lambda,\rho)=\sum_{\pi \in \cW_L(\rho)}
e^{-\lambda h(\pi)-H_L^{\omega,\cI}(\pi)}.
\end{equation}
By restricting the sum defining $U_L^{\omega,\cI}(\lambda)$ in (\ref{newmod}) to the
set $\cW_L(\rho)$, we obtain $u^\cI(\lambda)\geq \lim_{L\to\infty}\E[L^{-1}\log
U_L^{\omega,\cI}(\lambda,\rho)] = -\lambda \rho + \phi^\cI(1/\rho)$. Therefore,
optimising over $\rho$, we get $u^\cI(\lambda) \geq \sup_{\rho\in(0,1]}\{-\lambda
\rho +\phi^\cI(1/\rho)\}$.

To prove the reverse inequality, we note that an analogue of the concentration inequality
\eqref{concest} gives that there exists a $C>0$ such that, for all $L\in\N$, $\rho\in
(0,1]$ and $\gep>0$,
\begin{equation}
\label{majo}
\P\left(\frac{1}{L} \log U_L^{\omega,\cI}(\lambda,\rho) \geq
\E\left[\frac{1}{L}\log U_L^{\omega,\cI}(\lambda,\rho)\right]+\gep\right)
\leq C \exp[-\gep^2 L/C(\alpha+\beta)^2].
\end{equation}
Next, we define the event
\begin{equation}
J(L) =\left\{\exists\,j\in\{1,\dots,L\}\colon\,\frac{1}{L}
\log U_L^{\omega,\cI}(\lambda,j/L) \geq \E\left[\frac{1}{L}
\log U_L^{\omega,\cI}(\lambda,j/L)\right]+\gep\right\},
\end{equation}
and abbreviate $E(L)=E[L^{-1}\log U_L^{\omega,\cI}(\lambda)]$. Then we can write
\be{dira}
E(L) \leq \E\left(\left(\frac{1}{L}\log U_L^{\omega,\cI}(\lambda)\right)
1_{J(L)}\right)
+ \E\left(\frac{1}{L} \log \left(\sum_{j=1}^{L} U_L^{\omega,\cI}(\lambda,j/L)\right)
1_{[J(L)]^c}\right).
\ee
Trivially, the quantity $L^{-1}\log U_L^{\omega,\cI}(\lambda)$ can be bounded from above
by $\alpha+\tilde{\kappa}(0)$ (recall \eqref{equa4}), uniformly in $L$ and $\omega$. Therefore, with the help
of the inequality in \eqref{majo}, we see that the first term in the r.h.s.\ of
\eqref{dira} is bounded from above by $(\alpha+\tilde{\kappa}(0))CL\exp[-\gep^2 L/(C
(\alpha+\beta)^2)]$, which tends to zero as $L\to\infty$. Moreover, for every
$j\in \{1,\dots,L\}$, a standard subadditivity argument gives that $\E(L^{-1}
\log U_L^{\omega,\cI}(\lambda,j/L))\leq -\lambda j/L+\phi^\cI(L/j)$. Therefore, on
the event $[J(L)]^c$, we have that $L^{-1}\log U_L^{\omega,\cI}(\lambda,j/L)\leq -
\lambda j/L+\phi^\cI(L/j)+\gep$ for all $j\in \{1,\dots,L\}$. Thus, the second
term in the r.h.s.\ of \eqref{dira} is bounded from above by $(\log L)/L+
\max_{\rho\in(0,1]}\{-\lambda\rho+\phi^\cI(1/\rho)\}+\gep$. Letting $L\to\infty$
and $\gep\downarrow 0$, we obtain $\lim_{L\to \infty} E_1(L)\leq \max_{\rho\in(0,1]}
\{-\lambda \rho+\phi^\cI(1/\rho)\}$, which is the reverse inequality we were after.
\epr

Since $\rho\mapsto\phi^\cI(1/\rho)$ is continuous and concave, we can apply the
Fenchel-Legendre duality lemma (see Dembo and Zeitouni \cite{DemZei}, Lemma 4.5.8), to obtain
\begin{equation}
\label{inv1}
\phi^\cI(\mu)=\inf_{\lambda\geq 0}\{\lambda/\mu+u^\cI(\lambda)\}, \qquad \mu \geq 1.
\end{equation}
In the same spirit we have
\begin{equation}
\begin{aligned}
\label{eqa2}
\tilde{\kappa}(\lambda)&=\sup_{\rho\in (0,1]}\{-\lambda \rho+\hat{\kappa}(1/\rho)\},
\qquad \lambda \geq 0,\\
\hat{\kappa}(\mu)&=\inf_{\lambda\geq 0}\{\lambda/\mu+\tilde{\kappa}(\lambda)\},
\qquad \mu \geq 1.
\end{aligned}
\end{equation}

\subsection{Positive and finite curvature of $u^\cI$}
\label{Sustrconv}

In Propositions \ref{uidiff}--\ref{p:varform} we found that $u^\cI$ is smooth
and is the Fenchel-Legendre transform of $\phi^\cI$. In Section \ref{S5.2} we
will exploit these properties to obtain information on $\phi^\cI$. To prepare
for this, we first need to show the following. It is immediate from (\ref{newmod})
that $\lambda \mapsto u^\cI(\alpha,\beta;\lambda)$ is convex. Lemma \ref{assum2}
and Assumption \ref{assum} below state that it has a strictly positive and
finite curvature. To ease the notation, we suppress $\alpha,\beta$ from
some of the expressions.

\bl{assum2}
For all $(\alpha,\beta,\lambda)\in\cL_u$, $\partial^2 u^\cI(\alpha,\beta;\lambda)/\partial
\lambda^2>0$.
\el

\bpr
It suffices to prove that for all $(\alpha,\beta,\lambda_0)\in\cL_u$ there exist
$C,\gep>0$ such that, for all $\lambda\in I_\gep(\lambda_0)=[\lambda_0-\gep,\lambda_0+\gep]$
and $L\geq 1$
\be{dersec2}
\E\Big([E_{L}^{\omega,\cI}]^{\otimes 2}\big([h(\pi_1)-h(\pi_2)]^2\big)\Big)\geq C L,
\ee
where $E_L^{\omega,\cI}$ is expectation w.r.t.\ the law in (\ref{newmodlaw}),
and $\lambda$ is suppressed from the notation.

\medskip\noindent
{\bf Step 1.} By lemma \ref{lemM}(ii), we can assert that
for all $T_0\in\N$ there exist $z_0\in(0,1)$
and $L_0\in\N$ such that, for all $L\geq L_0$ and $\lambda\in I_\gep(\lambda_0)$,
\be{largeexcu2}
\E\bigg(P_L^{\omega,\cI}\Big(\Big\{\sum_{k=1}^{l_L}\tau_k\,
1_{\{\tau_k> T_0\}}\geq z_0 L\Big\}\Big)\bigg)
\geq \frac34.
\ee
where $\tau_k$ is the length of the $k$-th excursion. Similarly, by Lemma \ref{lemM}(i), there
exists $M_0\in\N$ with $M_0>T_0$ and $L_1\in\N$ such that, for all $L\geq L_1$ and
$\lambda\in I_\gep(\lambda_0)$,
\be{largeexcu3}
\E\bigg(P_L^{\omega,\cI}\Big(\Big\{\sum_{k=1}^{l_L}\tau_k\,
1_{\{\tau_k\leq M_0\}}\geq \left(1-\frac{z_0}{2}\right) L\Big\}\Big)\bigg)
\geq \frac34.
\ee
Abbreviate $\Gamma_0=\{T_0+1,\dots,M_0\}\times\{-1,+1\}$. Let $(j,\sigma)\in\Gamma_0$
and $L\geq L_2=\max\{L_0,L_1\}$. Define
\begin{equation}
\label{ABdefs}
A(L)=\Big\{\sum_{k=1}^{l_L}\tau_k 1_{\{T_0<\tau_k\leq M_0\}}\geq
\frac{z_0}{2}L\Big\}\ \ \text{and}\ \
B_{(j,\sigma)}(L)=\Big\{\sum_{k=1}^{l_L}\tau_k 1_{\{\tau_k=j, \sigma_k=\sigma\}}\geq
\frac{z_0}{4(M_0-T_0)}L\Big\},
\end{equation}
where $\sigma_k$ is the sign of the $k$-th excursion.
It follows from (\ref{largeexcu2}--\ref{largeexcu3}) that $\E\big(P_L^{\omega,\cI}
(A(L))\big)\geq \frac12$ and $A(L)\subset\cup_{(j,\sigma)\in \Gamma_0} B_{(j,\sigma)}(L)$.
Since $|\Gamma_0|=2 (M_0-t_0)$, for all $L\geq L_2$ and $\lambda\in I_\gep(\lambda_0)$,
there exists a
$(j_L,\sigma_L)\in \Gamma_0$ such that
\be{estimon1}
\E\big(P_L^{\omega,\cI}(B_{(j_L,\sigma_L)}(L))\big)
\geq \frac{1}{4(M_0-T_0)}.
\ee

\medskip\noindent
{\bf Step 2.} Henceforth, we abbreviate $B^L=B_{(j_L,\sigma_L)}(L)$. We will show that
the quantity
\begin{equation}
\label{Ecoupdef}
\cH^L=\E\Big([E_L^{\omega,\cI}]^{\otimes2}
\Big([h(\pi_1)-h(\pi_2)]^2\,1_{B^L}(\pi_1)\,1_{B^L}(\pi_2)\Big)\Big)
\end{equation}
is bounded from below by $CL$ for some $C>0$, which will complete the proof of
(\ref{dersec2}). For given $\pi$, we let
\be{Tpidef}
T(\pi)=\{(T_1,T'_{1},\sigma_1),\dots,(T_{l_L},T'_{l_L},\sigma_{l_L})\}
\ee
denote the starting points, ending points and signs of the $l_L$ excursions of $\pi$
between $0$ and $L$. For $r\in\N$, we set
\be{Zrdef}
\cZ_r^L=\{T(\pi)\colon\,\pi\in B^L,l_L=r\},
\ee
and we denote by $\cE(T,\sigma)$ the set of excursions of length $T$ and sign $\sigma$.
Futhermore, we write $(\varepsilon_1,\dots,\varepsilon_r)\sim T$ as short hand notation
for $(\varepsilon_1,\dots,\varepsilon_r)\in\cE(T'_1-T_1,\sigma_1)\times\dots\times
\cE(T'_r-T_r,\sigma_r)$. With this notation, we can write the
quantity in (\ref{Ecoupdef}) as
\be{secder}
\cH^L=\sum_{r,\tilde{r}}\sum_{T\in\cZ_r^L}
\sum_{\tilde{T}\in\cZ_{\tilde{r}}^L}
\E\Bigg[\frac{1}{(Z_L^{\omega})^2}\,
\bigg(\prod_{s=1}^r\  \prod_{\tilde{s}=1}^{\tilde{r}}\,
Z_{T,s}^{\omega}\ Z_{\tilde{T},\tilde{s}}^{\omega}\bigg)\,
R_{r,T,\tilde{r},\tilde{T}}^L\Bigg],
\ee
with $Z_L^{\omega}$ the total partition sum,
\be{Rdef*}
R_{r,T,\tilde{r},\tilde{T}}^L = \sum_{(\varepsilon_1,\dots,\varepsilon_r)\sim T}\,
\sum_{(\tilde{\varepsilon}_1,\dots,\tilde{\varepsilon}_{\tilde{r}})\sim\tilde{T}}\,
\prod_{s=1}^r\  \prod_{\tilde{s}=1}^{\tilde{r}}\frac{e^{-\lambda h(\varepsilon_s)}
e^{-\lambda h(\tilde{\varepsilon}_{\tilde{s}})}}
{Z_{T,s}\, Z_{\tilde{T},\tilde{s}}}\Big[\sum_{s=1}^r
h(\varepsilon_s)-\sum_{\tilde{s}=1}^{\tilde{r}}
h(\tilde{\varepsilon}_{\tilde{s}})\Big]^2
\ee
and (recall \eqref{Zinf})
\begin{equation}
\begin{aligned}
Z_{T,s}^{\omega} &=\sum_{\varepsilon_s\in \cE(T,s)}
e^{-\lambda h(\varepsilon_s)-H^{\omega,\cI}(\varepsilon_s)},\\
Z_{T,s} &=\sum_{\varepsilon_s\in \cE(T,s)}
e^{-\lambda h(\varepsilon_s)}.
\end{aligned}
\end{equation}
Note that $R_{r,T,\tilde{r},\tilde{T}}^L$ does not depend on $\omega$.

\medskip\noindent
{\bf Step 3.} Putting
\begin{equation}
X_s=h(\varepsilon_s),\qquad \tilde{X}_{\tilde{s}}=h(\tilde{\varepsilon}_{\tilde{s}}),
\qquad t_0=z_0/4M_0(M_0-T_0),
\end{equation}
we note that in $R_{r,T,\tilde{r},\tilde{T}}^L$ the random variables
\begin{equation}
\label{XtilXdef}
(X_1,\dots,X_r,\tilde{X}_1,\dots,\tilde{X}_{\tilde{r}})
\end{equation}
are independent, and that the law of $X_s$ depends on $(T'_s-T_s,\sigma_s)$. Since
$(T,\tilde{T})\in\cZ_r^L\times\cZ_{\tilde{r}}^L $, there are at least $t_0 L$
excursions of length $j_L$ and sign $\sigma_L$ in $T$ and $\tilde{T}$. Let
$(s_1,\dots,s_{t_0 L})$ and $(\tilde{s}_1,\dots,\tilde{s}_{t_0 L})$ be the indices
of the $t_0 L$ first such excursions in $T$ and $\tilde{T}$, put
\begin{equation}
Y_{r,T,\tilde{r},\tilde{T}}^L
=\sum_{s\in\{1,\dots,r\}\setminus\{s_1,\dots,s_{t_0 L}\}} X_s
\quad -\quad
\sum_{\tilde{s}\in\{1,\dots,\tilde{r}\}\setminus
\{\tilde{s}_1,\dots,\tilde{s}_{t_0 L}\}} \tilde{X}_{\tilde{s}},
\end{equation}
and write (\ref{Rdef*}) as
\begin{equation}
\label{RTTrep}
R_{r,T,\tilde{r},\tilde{T}}^L=E_{T,\tilde{T}}
\Big(\Big[\sum_{k=1}^{t_0 L} W_k + Y_{r,T,\tilde{r},\tilde{T}}^L\Big]^2\Big).
\end{equation}
where $W_k = X_{s_k}-\tilde{X}_{\tilde{s}_k}$ and $E_{T,\tilde{T}}$ denotes expectation
w.r.t.\ the law of (\ref{XtilXdef}). Clearly, $W=(W_k)_{k\in\{1,\dots,t_0 L\}}$ are i.i.d.,
symmetric and bounded random variables. Denote their variance by $v_L$. We can choose
$T_0$ large enough so that the $W_k$ are not constant. Moreover, since the $W_k$ have
only a finite number of laws, there exists an $a>0$ such that $v_L>a$ for all $\lambda
\in I_\gep(\lambda_0)$ and $L\geq L_2$.

\medskip\noindent
{\bf Step 4.}
At this stage, we may assume without loss of generality that $P_{T,\tilde{T}}(Y_{r,T,\tilde{r},\tilde{T}}^L\geq 0)
\geq \frac12$. Then (\ref{RTTrep}) gives
\be{minor}
R_{r,T,\tilde{r},\tilde{T}}^L\geq P_{T,\tilde{T}}(Y_{r,T,\tilde{r},\tilde{T}}^L\geq 0)\
\tfrac12\ E_{T,\tilde{T}}\Big(\Big[\sum_{k=1}^{t_0 L}W_k\Big]^2\Big)
\geq \tfrac14 E_{(j_L,\sigma_L)}\Big(\Big[\sum_{k=1}^{t_0 L} W_k\Big]^2\Big),
\ee
where $E_{(j_L,\sigma_L)}$ is expectation w.r.t.\ the law of $W$. Since the $W_k$ take
only values smaller than $2M_0$, their third moments are bounded by some finite $N$
uniformly in $\lambda\in I_\gep(\lambda_0)$ and $(j,\sigma)\in\Gamma_0$. Therefore we
can apply the Berry-Esseen theorem and, writing $\xi(u)=P(\cN(0,1)\leq u)$, $u \in \mathbb{R}$ with $N(0,1)$
a standard normal random variable, can assert that, for all $u\in\R$, $\lambda\in I_\gep
(\lambda_0)$ and $(j,\sigma)\in\Gamma_0$,
\be{minor2}
\bigg|P_{(j,\sigma)}\Big(\sum_{k=1}^{t_0 L} W_k \leq u \sqrt{t_0 L v_L}\Big)
-\xi(u)\bigg|\leq \frac{3N}{a^{3/2} \sqrt{t_0 L}},
\ee
where $P_{(j,\sigma)}$ is the law of $W$ when $(j_L,\sigma_L)=(j,\sigma)$.
Taking the restriction of the r.h.s.\ of \eqref{minor} to the event $K=\{\sum_{k=1}^{t_0 L}
W_k/\sqrt{t_0 L v_L}\in[1,2]\}$, we obtain
\be{minor3}
R_{r,T,\tilde{r},\tilde{T}}^L\geq\frac{v_Lt_0 L}{4}
P_{(j,\sigma)}(K)\geq \frac{a t_0 L}{4}
\Big(\xi(2)-\xi(1)-\frac{6 N}{a^{3/2} \sqrt{t_0 L}}\Big),
\ee
which implies that $R_{r,T,\tilde{r},\tilde{T}}^L\geq t'_0 L$ for $L$ large enough and some $t'_0>0$.
Recalling \eqref{secder}, we can now estimate
\begin{equation}
\cH^L \geq  t'_0L\, \E\big([P_L^{\omega,\cI}]^{\otimes 2}(B^L)\big)
\geq t'_0 L/4(M_0-T_0),
\end{equation}
which yields (\ref{dersec2}) with $C=t'_0 L/4(M_0-T_0)$.
\epr

\begin{assumption}
\label{assum}
For all $(\alpha,\beta)\in\CONE$ and $\lambda> 0$ there exist $C(\lambda)>0$
and $\delta_0>0$ such that, for all $\delta\in(0,\delta_0]$,
\be{asu}
u^\cI(\lambda-\delta)+u^\cI(\lambda+\delta)-2u^\cI(\lambda)\leq C(\lambda) \delta^2.
\ee
\end{assumption}

\noindent
Although we are not able to prove this assumption, we believe it to be true for the
following reason. First, as a consequence of Proposition \ref{uidiff}, we have that,
for all $(\alpha,\beta)\in \CONE$, $\lambda\mapsto u(\alpha,\beta; \lambda)$ is infinitely
differentiable on the set $\{\lambda\in[0,\infty)\colon\,u(\alpha,\beta;\lambda)>
\tilde{\kappa}(\lambda)\}$. Since $\lambda\mapsto\tilde{\kappa}(\lambda)$ is infinitely
differentiable on $[0,\infty)$, this implies that $\lambda\mapsto u(\alpha,\beta;\lambda)$
is infinitely differentiable on the interior of the set $\{\lambda\in[0,\infty)\colon\,
u(\alpha,\beta;\lambda)=\tilde{\kappa}(\lambda)\}$. Thus, the assumption only concerns
the values of $\lambda$ located at the boundary of the latter. For these values, proving the
assumption amounts to proving the reverse of inequality \eqref{dersec2}, i.e., showing
that the variance of the number of horizontal steps made by the polymer of length $L$ is
of order $L$, which we may reasonably expect to be true. In Remark \ref{Sub5.2} we give
a weaker alternative to Assumption \ref{assum}.

\subsection{Smoothness of $\phi^\cI$ in its localized phase}
\label{S5.2}

Having collected in Section \ref{S5.1}--\ref{Sustrconv} some key properties of the
dual free energy $u^\cI$, we are now ready to look at what these imply for
$\phi^\cI$. We begin by showing that $\phi^\cI$ is strictly concave.

\begin{lemma}
\label{stconc}
Let
\begin{equation}
\label{equiRdef}
D(\delta)=\tfrac{1}{2}\phi^\cI\left(\frac{1}{\rho_0+\delta}\right)
+\tfrac{1}{2}\phi^\cI\left(\frac{1}{\rho_0-\delta}\right)
-\phi^\cI\left(\frac{1}{\rho_0}\right).
\end{equation}
Then, for all $(\alpha,\beta)\in\CONE$ and $\rho_0\in(0,1)$ there exist $C>0$ and
$\delta_0>0$ such that, for all $\delta\in (0,\delta_0]$,
\begin{equation}
\label{equi}
D(\delta) \leq -C \delta^2.
\end{equation}
This inequality implies the strict concavity of $\rho\mapsto\phi^\cI(1/\rho)$ on
$(0,1]$.
\end{lemma}

\begin{proof}
Lemma \ref{assum2} states the strict convexity of $\lambda\mapsto u^\cI(\lambda)$,
which implies the uniqueness of the maximiser in the variational formula \eqref{inv1},
i.e., there exists a unique $\lambda_0=\lambda_0(\rho)\geq 0$ such that $\phi^\cI
(1/\rho_0)=\lambda_0\rho_0+u^\cI(\lambda_0)$. Let $x>0$. By picking $\lambda=\lambda_0
-x\delta$ in \eqref{inv1} with $\mu=1/(\rho_0+\delta)$, and $\lambda=\lambda_0+x \delta$
in \eqref{inv1} with $\mu=1/(\rho_0-\delta)$, we obtain
\begin{equation}
\begin{aligned}
\label{eqa3}
D(\delta)
&\leq \tfrac12[(\lambda_0-x\delta)(\rho_0+\delta)+u^\cI(\lambda_0-x\delta)]\\
&\qquad +\tfrac12[(\lambda_0+x\delta)(\rho_0-\delta)
+u^\cI(\lambda_0+x\delta)]-\lambda_0\rho_0-u^\cI(\lambda_0)\\
&=-x\delta^2+\tfrac12[u^\cI(\lambda_0-x\delta)
+u^\cI(\lambda_0+x\delta)-2u^\cI(\lambda_0)].
\end{aligned}
\end{equation}
Picking $x=1/2C(\lambda_0)$, with $C(\lambda_0)$ the constant in Assumption \ref{assum},
we see that \eqref{eqa3} implies, for $0<\delta<2C(\lambda_0) \delta_0$,
\be{sconc}
D(\delta) \leq -x\delta^2+C(\lambda_0) x^2\delta^2=-\delta^2/4C(\lambda_0),
\ee
which proves (\ref{equi}). To prove the claim made below (\ref{equi}), pick
$1 \leq u < v$ and consider \eqref{equiRdef} at the point $\rho_0=(u+v)/2$.
Then, by (\ref{equiRdef}--\ref{equi}), there exists a $0<\delta<(v-u)/2$ such
that
\begin{equation}
\label{eqimp}
\frac{\phi^\cI(\frac{1}{\rho_0+\delta})-\phi^\cI(\frac{1}{\rho_0})}{\delta}
< \frac{\phi^\cI(\frac{1}{\rho_0})-\phi^\cI(\frac{1}{\rho_0-\delta})}{\delta}.
\end{equation}
Since $v>\rho_0+\delta>\rho_0-\delta>u$, it follows that
\be{partdifineq}
\frac{\partial^-\phi^\cI}{\partial\rho}\left(\frac{1}{\rho}\right)|_{\rho=v}
\leq \mbox{ {\rm l.h.s.} (\ref{eqimp}) } < \mbox{ {\rm r.h.s.} (\ref{eqimp}) }
\leq \frac{\partial^+\phi^\cI}{\partial\rho}\left(\frac{1}{\rho}\right)|_{\rho=u},
\ee
with $-$ and $+$ denoting the left- and the right-derivative.
\end{proof}

We are now ready to prove that $\phi^\cI$ is smooth. Let
\begin{equation}
\cL_\phi=\big\{(\alpha,\beta,\mu) = \CONE \times [1,\infty)\colon\,
\phi^\cI(\alpha,\beta;\mu)>\hat{\kappa}(\mu)\big\},
\end{equation}
i.e., the region where the single linear interface model is localized.

\begin{proposition}
\label{phicinf}
$(\alpha,\beta,\mu)\mapsto\phi^\cI(\alpha,\beta;\mu)$ is infinitely differentiable on
$\cL_\phi$.
\end{proposition}

\bpr
Let $(\alpha,\beta,\mu)\in\cL_\phi$. Lemma \ref{assum2} states the strict convexity of
$\lambda\mapsto u^\cI(\lambda)$ on $\{\lambda : u(\lambda)>\tilde{\kappa}(\lambda)\}$
and it can be shown that $\lambda \mapsto \tilde{\kappa}(\lambda)$ is strictly convex
on $[0,\infty)$. This entails that $\lambda\mapsto u^\cI(\lambda)$ is strictly convex
on $[0,\infty)$. Therefore, the variational formula in (\ref{inv1}) attains its maximum
at a unique point $\lambda(\mu)\geq 0$, so that the variational formula in \eqref{eqa}
allows us to write
\begin{equation}
\label{eqimpo}
\phi^\cI(\mu)=\lambda(\mu)/\mu+\sup_{\rho\in (0,1]}\{-\lambda(\mu)\rho+\phi^\cI(1/\rho)\},
\end{equation}
after which the strict concavity of $\rho\mapsto \phi^\cI(1/\rho)$ (recall Lemma
\ref{stconc}) implies that this supremum is attained uniquely at $\rho=1/\mu$. Since
$\phi^\cI(\rho)\geq\hat{\kappa}(\rho)$ for all $\rho$, and $\phi^\cI(\mu)>\hat{\kappa}
(\mu)$, the variational formula in \eqref{eqa2} allows us to write $u^\cI(\lambda(\mu))>
\tilde{\kappa}(\lambda(\mu))$, and therefore $(\alpha,\beta,\lambda(\mu))\in
\mathcal{L}_u$.

Next, let
\begin{equation}
\cS =\left\{ (\alpha,\beta,\mu,\lambda)\in \CONE \times [1,\infty) \times [0,\infty)
\colon (\alpha,\beta,\mu)\in\cL_\phi, (\alpha,\beta,\lambda)\in\cL_u\right\},
\end{equation}
and define $\Upsilon_1$ as
\be{ups3}
\Upsilon_1\colon\,(\alpha,\beta,\mu,\lambda)\in \cS\mapsto
\frac{\partial (\lambda/\mu+u^\cI(\lambda))}{\partial\lambda}.
\ee
We want to apply the implicit function theorem in Bredon \cite{Bred}, Chapter II, Theorem 1.5,
to $\Upsilon_1$. This requires checking three properties:
\begin{itemize}
\item[(i)]
$\Upsilon_1$ is infinitely differentiable on $\cS$.
\item[(ii)]
For all $(\alpha,\beta,\mu) \in \cL_\phi$, $\lambda(\mu)$ is the unique $\lambda\in [1,\infty)$
such that $(\alpha,\beta,\lambda)\in \cL_u$ and $\Upsilon_1(\alpha,\beta,\mu,\lambda(\mu))$
$=0$.
\item[(iii)]
For all $(\alpha,\beta,\mu) \in \cL_\phi$, $\frac{\partial \Upsilon_1}{\partial\lambda}
(\alpha,\beta,\mu,\lambda(\mu))\neq 0$.
\end{itemize}
Property (i) holds because $u^\cI$ is infinitely differentiable on $\cL_u$ (by
Proposition \ref{uidiff}). Property (ii) holds because $\lambda\mapsto u^\cI(\lambda)$
is strictly convex (by Lemma \ref{assum2}). Moreover, Lemma \ref{assum2} gives that
\be{derpar}
\frac{\partial \Upsilon_1}{\partial\lambda}(\alpha,\beta,\mu,\lambda(\mu))
= \frac{\partial^2 u^\cI}{\partial \lambda^2}(\alpha,\beta,\lambda(\mu))>0,
\ee
so property (iii) holds too. We can therefore indeed use the implicit function theorem,
obtaining that $(\alpha,\beta,\mu)\mapsto\lambda(\mu)$ and $(\alpha,\beta,\mu) \mapsto
\phi^\cI(\alpha,\beta;\mu)$ are infinitely differentiable on $\cL_\phi$.
\epr

\br{Sub5.2}
Assumption {\rm \ref{uidiff}} can be weakened. Namely, instead of assuming finite curvature
of $\lambda\mapsto u(\alpha,\beta;\lambda)$, we may assume strict concavity of $\mu
\mapsto\mu \phi^\cI(\mu)$ (which is already known to be concave). This strict concavity
is implied by Assumption {\rm \ref{assum}}, Lemma {\rm \ref{equiRdef}} and \eqref{dersecphi},
and is sufficient to guarantee, in the proof of Proposition {\rm \ref{phicinf}}, that
$\lambda(\mu)$ in \eqref{eqimpo} is unique and satisfies $(\alpha,\beta,\lambda(\mu))
\in\cL_\mu$. This in turn is enough to carry out the rest of the proof.
\er

\subsection{Smoothness of $\psi_{AB}$ in its localized phase}
\label{S5.3}

In this section we transport the properties of $\phi^\cI$ obtained in Section
\ref{S5.2} to $\psi_{AB}$. We begin with some elementary observations. Fix
$(\alpha,\beta) \in \CONE$ and recall (\ref{DOMadef}). By Lemma \ref{stconc}
and Lemma \ref{l:ka}(ii), for all $a\geq 2$, $(c,b) \mapsto c \phi^{\cI}(c/b)$
and $(c,b)\mapsto (a-c) \kappa(a-c,1-b)$ are strictly concave on $\DOM(a)$.
Consequently, for all $a \geq 2$, the supremum of the variational formula in
\eqref{psiinflink} is attained at a unique pair $(c,b)\in\DOM(a)$ (use that
$\DOM(a)$ is a convex set).

Next, note that Lemma \ref{stconc} and Proposition \ref{phicinf} imply that
for all $(\alpha,\beta,\rho_0)\in\cL_\phi$ there exists a $C>0$ such that
\begin{equation}\label{dersecphi}
\tfrac{\partial^2}{\partial \rho^2}[\rho\phi^\cI(\rho)]
(\rho_0)=\tfrac{1}{\rho_0^3}\tfrac{\partial^2}{\partial \rho^2}
\big[\phi^\cI \big(1/\rho\big)\big]
\big(\tfrac{1}{\rho_0}\big) \leq -C.
\end{equation}

Let
\begin{equation}
\cL_{\psi} = \{(\alpha,\beta,a)\in \CONE \times [2,\infty)\colon\,
\psi_{AB}(\alpha,\beta;a)>\varpi\},
\end{equation}
i.e., the region where $\psi_{AB}$ is localized. Our main result in this section
is the following.

\bp{psiklsmooth}
$(\alpha,\beta,a)\mapsto\psi_{AB}(\alpha,\beta;a)$ is infinitely differentiable
on $\cL_{\psi}$.
\ep

\bpr
Define
\begin{equation}
\cL_{\alpha,\beta,a} = \{(c,b)\in\DOM(a)\colon\,
\phi^\cI(\alpha,\beta;c/b)>\hat{\kappa}(c/b)\}.
\end{equation}
As noted above, the variational formula in \eqref{psiinflink} attains its maximum at
a unique pair $(c(\alpha,\beta;a),b(\alpha,\beta;a))\in\DOM(a)$. We write $(c(a),b(a))$,
suppressing $(\alpha,\beta)$ from the notation. Since $(\alpha,\beta)\in\cL$ (recall
\eqref{DL}), Lemma \ref{l:kamu}(iv) and Proposition \ref{p:phtrinfchar} imply that
$(c(a),b(a))\in\cL_{\alpha,\beta,a}$. Let
\be{partialder}
F(c,b)=c\phi(c/b),\qquad \tilde{F}(c,b)=(a-c)\kappa(a-c,1-b),
\ee
and denote by $\{F_{c},F_b,F_{cc},F_{cb},F_{bb}\}$ the partial derivatives of
order $1$ and $2$ of $F$ with respect to the variables $c$ and $b$ (and similarly
for $\tilde{F}$). By the strict concavity of $(c,b)\mapsto F(c,b)+\tilde{F}(c,b)$
in $\DOM(a)$, we know that $(c(a),b(a))$ is also the unique pair in $\cL_{\alpha,\beta,a}$
at which $F_{c}+\tilde{F}_{c}=0$ and $F_{b}+\tilde{F}_{b}=0$.

We need to show that $(c(a),b(a))$ is infinitely differentiable w.r.t.\ $(\alpha,\beta,a)$.
To that aim we again use the implicit function theorem. Define
\be{subse}
\cR=\{(\alpha,\beta,a,c,b)\colon\,(\alpha,\beta,a)\in\cL_{\psi},\,
(c,b)\in\cL_{\alpha,\beta,a}\}
\ee
and
\be{eq18}
\Upsilon_2\colon\,(\alpha,\beta,a,c,b)\in \cR\mapsto
(F_{c}+\tilde{F}_{c},F_{b}+\tilde{F}_{b}).
\ee
Let $J_2$ be the Jacobian determinant of $\Upsilon_2$ as a function of $(c,b)$. Applying
the implicit function theorem to $\Upsilon_2$ requires checking three properties:
\begin{itemize}
\item[(i)]
$\Upsilon_2$ is infinitely differentiable on $\cR$.
\item[(ii)] For all $(\alpha,\beta,a) \in \cL_{\psi}$, $(c(a),b(a))$ is the
only pair in $\cL_{\alpha,\beta,a}$ satisfying $\Upsilon_2=0$.
\item[(iii)]
 For all $(\alpha,\beta,a) \in \cL_{\psi}$, $J_2 \neq 0$ in $(c(a),b(a))$.
\end{itemize}
As explained below (\ref{partialder}), property (ii) holds. Proposition \ref{phicinf}
and Lemma \ref{l:kamu}(ii) show that also property (i) holds. Computing the Jacobian
determinant $J_2$, we get
\be{eq19}
J_2=(F_{cc}+\tilde{F}_{cc})(F_{bb}+\tilde{F}_{bb})-(F_{cb}+\tilde{F}_{c,b})^2.
\ee
Since $F_{cc}F_{bb}-F_{cb}^2=0$, $F_{bb}=\mu^2 F_{cc}$ and $F_{cb}=\mu F_{cc}$,
\eqref{eq19} becomes
\be{eq20*}
J_2 = \tilde{F}_{cc}\tilde{F}_{bb}-\tilde{F}_{cb}^2+F_{cc}
[\tilde{F}_{bb}+2\mu\tilde{F}_{cb}+\mu^2 \tilde{F}_{cc}].
\ee
By the concavity of $c\mapsto F(c,b)$ and $c\mapsto \tilde{F}(c,b)$, we have
$F_{cc}\leq 0$ and $\tilde{F}_{cc}\leq 0$. Moreover, by the concavity of $(c,b)
\mapsto\tilde{F}(c,b)$, its Hessian matrix necessarily has two non-positive eigenvalues.
Therefore, the determinant of this matrix is non-negative, i.e., $\tilde{F}_{cc}\tilde{F}_{bb}
-\tilde{F}_{cb}^2\geq 0$. This, together with the inequality $\tilde{F}_{cc}
\leq 0$, implies that $\mu \mapsto \tilde{F}_{bb}+2\mu\tilde{F}_{cb}+ \mu^2 \tilde{F}_{cc}$
is non-positive on $\R$. Hence $J_2\geq 0$.

\bl{MAPLEstrconv}
$\tilde{F}_{cc}\tilde{F}_{bb}-\tilde{F}_{cb}^2 > 0$.
\el

\bpr
The strict inequality can be checked with MAPLE. In \cite{dHW06}, an explicit variational
formula is given for the entropy function in (\ref{kappa}), which is easily implemented.
\epr

\noindent
It follows from Lemma \ref{MAPLEstrconv} that $J_2>0$, which proves property (iii).
We know from Lemma \ref{l:ka}(ii) and Proposition \ref{phicinf} that $\tilde{F}$
and $F$ are infinitely differentiable on $\DOM(a)$ for all $a \in [2,\infty)$. Hence,
the claim indeed follows the implicit function theorem.
\epr

We close this section with the following observations needed in Section \ref{S5.4}.

\bl{l:c'} Fix $(\alpha,\beta) \in \CONE$.\\
(i)  For all $k,l\in\{A,B\}$, $a\mapsto\psi_{kl}(a)$ is strictly concave on
$[2,\infty)$.\\
(ii) For all $k,l\in\{A,B\}$ with $kl\neq BB$, $\lim_{a\to\infty} a\psi_{kl}(a)=\infty$.\\
$(iii)$  For all $k,l\in\{A,B\}$, $\lim_{a\to\infty}\partial
[a\psi_{kl}(a)]/\partial a\leq 0$.
\el

\bpr
(i) This is a straightforward consequence of the observations made at the beginning
of this section, together with the strict concavity of $\mu\mapsto\mu\phi^\cI(\mu)$
proved in Lemma \ref{stconc}.

\medskip\noindent
(ii) Because $\psi_{AB}\geq\psi_{AA}$, it suffices to consider $kl\in\{AA,BA\}$.
For $kl=AA$, the claim is immediate from Lemma \ref{l:ka}(iii) and (\ref{psiAABB}).
For $kl=BA$, we use the fact that $\phi^{\cI}(\mu)\geq\hat{\kappa}(\mu)$ (recall
(\ref{phiublb})) in combination with the variational formula of Lemma \ref{l:linkinf}
with $c=a-\frac32$ and $b=\frac12$. This gives
\be{liminfi}
a\psi_{BA}(a)\geq\tfrac12\,(2a-3)\,\hat\kappa(2a-3)+\tfrac32\,
\left[\kappa(\tfrac32,\tfrac12)+\tfrac12(\beta-\alpha)\right],
\ee
which yields the claim because $\mu\hat\kappa(\mu)\sim\log\mu$ as $\mu\to \infty$
by Lemma \ref{l:kamu}(iii).

\medskip\noindent
(iii) Since, for all $k,l\in \{A,B\}$, $\psi_{AB}\geq\psi_{kl}$ and $a\mapsto a\psi_{kl}(a)$
is concave, it suffices to prove that $\limsup_{a\to\infty}\psi_{AB}(a)\leq 0$. The latter
is immediate from the variational formula in (\ref{psiinflink}) and the fact that
$\lim_{a\to\infty}\phi^{\cI}(a)=0$ (Lemma \ref{acdellim}(i)) and $\lim_{a\to\infty}
\kappa(a,1)=0 ((\ref{ka}))$.
\epr

\subsection{Smoothness of $f$ on $\cL$}
\label{S5.4}

We begin by proving the uniqueness of the maximisers in the variational formula in
\eqref{fevar}. For $(\alpha,\beta)\in\CONE$, $p\in(0,1)$ and $(\rho_{kl})\in
\cR(p)$, let
(recall (\ref{fctV}))
\be{maxdefs}
\begin{aligned}
f_{(\rho_{kl})} &= \sup_{(a_{kl})\in \cA} V\big((\rho_{kl}),(a_{kl})\big),\\
\cO_{(\rho_{kl})} &= \{kl\in\{A,B\}^2\colon\,\rho_{kl}>0\},\\
\cR^f(p) &= \{(\rho_{kl})\in\cR(p)\colon\,f=f_{(\rho_{kl})}\},\\
\cP(p) &= \bigcup_{(\rho_{kl})\in \cR^f(p)} \cO_{(\rho_{kl})}.
\end{aligned}
\ee

\bp{p:unimaw}
(i) For every $(\alpha,\beta)\in\CONE$, $p\in (0,1)$ and $\rho=(\rho_{kl})\in\cR(p)$,
there exists a unique family $a^\rho=(a_{kl}^\rho)_{kl\in \cO_\rho}\in\cA$ satisfying
\be{maxi}
f_\rho=\frac{\sum_{kl\in\cO_\rho}
\rho_{kl} a^{\rho}_{kl} \psi_{kl}(a^{\rho}_{kl})}
{\sum_{kl\in\cO_\rho} \rho_{kl} a^{\rho}_{kl}}
= V(\rho,a^\rho).
\ee
(ii) For every $(\alpha,\beta)\in\CONE$ and $p\in(0,1)$, $\cR^f(p)\neq\emptyset$ and there
exists a unique family $(a^*_{kl})_{(k,l) \in\cP(p)}$ such that $a^\rho_{kl}=a^*_{kl}$ for
all $\rho\in\cR^f(p)$ and $kl\in\cO_\rho$.
\ep

\bpr
Recall Theorem \ref{feiden}.

\medskip\noindent
(i) The case $\rho_{BB}=1$ is trivial. In that case we have $f_\rho=\sup_{a_{BB}\geq 2}
\psi_{BB}(a_{BB}) = \psi_{BB}(a^*)=\frac12\beta+\varpi$ (by Lemma \ref{l:ka}(iv)), and
so $a^{\rho}_{BB}=a^*=\frac52$. Therefore assume that $\rho_{BB}<1$. Then at least one
pair $k_1l_1\in\{AA,AB,BA\}$ satisfies $\rho_{k_1l_1}>0$, and since $\lim_{u\to\infty}
u\psi_{k_1l_1}(u)=\infty$ by Lemma \ref{l:c'}(ii), we have $f_\rho>0$. The latter is needed
in what follows.

To prove existence of $a^\rho$, for $R>0$ let
\be{GrhoRdef}
f_{\rho,R} = \sup_{a\in [2,R]^{\cO_\rho}} V(\rho,a).
\ee
We prove that for $R$ large enough the supremum in \eqref{maxi} is attained in
$[2,R]^{\cO_{\rho}}$, i.e., $f_\rho=f_{\rho,R}$. Indeed, for $a\in\cA$, $\rho\in\cR(p)$
and $k_2l_2\in\{A,B\}^{2}$ we have (recall (\ref{fctV}))
\be{deriv}
\frac{\partial V}{\partial a_{k_2l_2}}(\rho,a)
= \frac{\rho_{k_2l_2}}{\sum_{kl} \rho_{kl} a_{kl}}
\Big\{\frac{\partial [u\psi_{k_2l_2}(u)]}{\partial u}|_{u=a_{k_2l_2}}-V(\rho,a)\Big\}.
\ee
Moreover, for every $kl\in\{A,B\}^{2}$, $u\mapsto u\psi_{kl}(u)$ is strictly concave and
$u\mapsto\partial[u\psi_{kl}(u)]/\partial u$ is strictly decreasing (by Lemma
\ref{l:c'}(i)) and converges to a limit $\leq 0$ as $u\to\infty$ (by Lemma \ref{l:c'}(iii)).
Pick $R>0$ large enough so that $\partial[u\psi_{kl}(u)]/\partial u \leq f_\rho/2$ for all
$u\geq R$ and $kl\in\{A,B\}^2$. We will show that $f_\rho>f_{\rho,R}$ implies that
$V(\rho,a)\leq\max\{f_\rho/2,f_{\rho,R}\}$ for all $a\in \cA\setminus [2,R]^{\cO_\rho}$,
and this will provide a contradiction.

To achieve the latter, assume that $AA\in\cO_\rho$ and consider, for instance, $a\in\cA$
such that $a_{AA}>R$ and $a_{kl}\leq R$ for $kl\in\cO_\rho\setminus\{AA\}$. Fix $x\geq R$
and denote by $a^x$ the element of $\cO_\rho$ given by $a_{AA}^{x}=x$ and $a^{x}_{kl}=a_{kl}$,
$kl\in\cO_\rho\backslash\{AA\}$. Since $a^{R}\in [2,R]^{\cO_\rho}$, we have $V(\rho,a^R)
\leq f_{\rho,R}<f_\rho$ and
\be{inte}
V(\rho,a^x)-V(\rho,a^R)=\int_{R}^{x} \frac{\partial V}{\partial a_{AA}}(\rho,a^{u})\,du.
\ee
Since, by \eqref{deriv}, the sign of $(\partial V/\partial a_{AA})(\rho,a^{u})$ is equal
to the sign of $\partial [u\psi_{AA}(u)]/\partial u -V(\rho,a^u)$, it follows that
$V(\rho,a^x)$ decreases with $x$ whenever $V(\rho,a^x)\geq f_\rho/2$. Since $V(\rho,a^R)
<f_\rho$, we therefore have $V(\rho,a^x)\leq \max\{f_\rho/2,f_{\rho,R}\}$ for all
$x\geq R$ and, consequently, $V(\rho,a) \leq \max\{f_\rho/2, f_{\rho,R}\}$. Therefore
the supremum of \eqref{maxi} is attained in $[2,R]^{\cO_\rho}$.

The uniqueness of $a^\rho$ realising $f_\rho=V(\rho,a^{\rho})$ follows from \eqref{deriv},
because for each $kl\in\{A,B\}^{\cO_\rho}$ we must have $(\partial V/\partial a_{kl})
(\rho,a^{\rho})=0$. This means that for each $kl\in\cO_\rho$ we must have
\be{muha}
\frac{\partial [u \psi_{kl}(u)]}{\partial u}|_{u=a_{kl}^\rho}
= V(\rho,a^\rho) = \sup_{a\in\cA} V(\rho,a),
\ee
and, since $u\mapsto u \psi_{kl}(u)$ is strictly concave (by Lemma \ref{l:c'}(i)),
there is only one such $a_{kl}$ for each $kl\in\cO_\rho$.

\medskip\noindent
(ii) As shown in \cite{dHW06}, Proposition 3.2.1, $\rho\mapsto f_\rho$ is continuous
on $\cR(p)$. Therefore, the compactness of $\cR(p)$ entails $\cR^f(p)\neq \emptyset$.
Consider $(\rho_1,\rho_2)\in\cR^f(p)$ and $kl\in\cO_{\rho_1}\cap\cO_{\rho_2}$. Then
(\ref{deriv}) also gives
\be{algi}
\frac{\partial [u\psi_{kl}(u)]}{\partial u}|_{u=a_{kl}^{\rho_1}}
= f = \frac{\partial [u\psi_{kl}(u)]}{\partial u}|_{u=a_{kl}^{\rho_2}},
\ee
which, by the strict concavity of $u\mapsto u\psi_{kl}(u)$, implies that $a_{kl}^{\rho_1}
=a_{kl}^{\rho_2}$.
\epr

We are now ready to prove the smoothness of $f$ on $\cL$. Because of the inequalities
$\psi_{AA}\geq \psi_{BB}$ and $\psi_{AB}\geq \psi_{BA}$, the concavity of $a\mapsto a
\psi_{AA}(a)$ and $a\mapsto a \psi_{AB}(a)$ implies that the variational problem in
\eqref{fevar} reduces to the matrices $\{M_\gamma,\gamma\in C\}$, with $M_\gamma$
the matrix and $C$ the set defined in \eqref{Rpa}. Write $V(\gamma,a_{AB},a_{AA})$
for the quantity $V(M_\gamma,(a_{AB},a_{AA},0,0))$ defined in \eqref{fctV}, put
$\gamma^*=\max C$ and let $(x^*(\alpha,\beta),y^*(\alpha,\beta))$ be the unique maximisers
$(a_{AB}^*,a_{AA}^*)$ defined in Proposition \ref{p:unimaw}. By differentiating
the quantity $V(\gamma,x^*,y^*)$ with respect to $\gamma$, we easily get that
$\cR^f(p)$ contains only the matrix $M_{\gamma^*}$. Thus, we have
the equality
\be{fe}
f(\alpha,\beta)=V(\gamma^*,x^*,y^*)=\frac{\gamma^* x^*
\psi_{AB}(x^*)+(1-\gamma^*) y^* \kappa(y^*,1)}{\gamma^* x^*+(1-\gamma^*) y^*}.
\ee
Since $(\alpha,\beta)\in \cL$, we have $\psi_{AB}(x^*)>\varpi$ and therefore
$(\alpha,\beta,x^*)\in \cL_{\psi}$. To show that $f$ is infinitely differentiable on
$\cL$, we once more use the implicit function theorem. For that we define
\be{subse2}
\mathcal{N}=\{(\alpha,\beta,x,y)\colon\,(\alpha,\beta)\in\cL,\,
(\alpha,\beta,x)\in\cL_{\psi}, y> 2\}
\ee
and
\be{eq20**}
\Upsilon_3\colon\,(\alpha,\beta,x,y)\in\cN \mapsto
\Big(\frac{\partial V}{\partial x}(\gamma^*,x,y),
\frac{\partial V}{\partial y}(\gamma^*,x,y)\Big).
\ee
Let $J_3$ be the Jacobian determinant of $\Upsilon_3$ as a function of
$(\alpha,\beta,x,y)$. To apply the implicit  function theorem we must check three properties:
\begin{itemize}
\item[(i)]
$\Upsilon_3$ is infinitely differentiable on $\cN$.
\item[(ii)]
For all $(\alpha,\beta) \in \cL$, $(x^*,y^*)$ is the only pair in
$[2,\infty)^2$ satisfying
$(\alpha,\beta,x,y)\in \cN$
and $\Upsilon_3(\alpha,\beta,x,y)=0$.
\item[(iii)]
For all $(\alpha,\beta) \in \cL$, $J_3 \neq 0$ in $(\alpha,\beta,x^*,y^*)$.
\end{itemize}
It follows from Lemma \ref{l:ka}(ii), Proposition \ref{psiklsmooth} and (\ref{fe})
that property (i) and (ii) hold. To get property (iii), abbreviate $x \psi_{AB}(x)
=\psi(x)$, $y \kappa(y,1)=\kappa(y)$. From Lemma \ref{l:ka}(ii) and Proposition
\ref{psiklsmooth}, we know that $\psi$ and $\kappa$ are infinitely differentiable.
By (\ref{eq20**}),
\be{21}
J_3=\frac{\partial^2 V}{\partial x^2}
\frac{\partial^2 V}{\partial y^2}-\Big(\frac{\partial^2 V}{\partial x\partial y}\Big)^2.
\ee
Taking into account that $(\partial V/\partial x)(x^*,y^*)=(\partial V/\partial y)
(x^*,y^*)=0$, we deduce from (\ref{fe}) that $\psi'(x^*)=\kappa'(y^*)$  and
$J_3=c^* \psi''(x^*) \kappa''(y^*)$, where $c^*>0$ is a constant depending on
$(x^*,y^*)$. We already know from Lemma \ref{l:ka}(iii) that $\kappa''(y^*)<0$.

\bl{LGstrconv}
$\psi''(x^*)<0$.
\el

\bpr
For $x>2$ satisfying $(\alpha,\beta,x)\in\cL_\psi$, we will show that $(x\psi_{AB}(x))''<0$.
For this it suffices to show
that there exists a $C>0$ such that, for $\delta$ small enough,
\begin{equation}
\label{equi2}
T(\delta)=\tfrac12\big[(x+\delta)\psi_{AB}(x+\delta)
+(x-\delta)\psi_{AB}\big(x-\delta\big)-2x\psi_{AB}(x)\big]
\leq -C \delta^2.
\end{equation}
Set $x_{-\delta}=x-\delta$ and $x_{\delta}=x+\delta$, and let $(e_{-\delta},b_{-\delta})$ and $(e_{\delta},b_{\delta})$ be the
unique maximisers of \eqref{psiinflink} at $x_{-\delta}$ and $x_{\delta}$. Pick $(c,b)=(\frac12(e_{-\delta}+e_{\delta}),
\frac12(b_{-\delta}+b_{\delta}))$ in \eqref{psiinflink}. Since $x=\frac12(x_{-\delta}+x_{\delta})$, we obtain $T(\delta)\leq
V_1(\delta)+V_2(\delta)$ with
\begin{equation}
\label{equi2*}
\begin{aligned}
V_1(\delta) &=\tfrac12\left[e_{-\delta}\phi^\cI(\tfrac{e_{-\delta}}{b_{-\delta}})+e_{\delta}
\phi^\cI(\tfrac{e_{\delta}}{b_{\delta}})
-(e_{-\delta}+e_{\delta})\phi^\cI\left(\tfrac{e_{-\delta}+e_{\delta}}{b_{-\delta}+b_{\delta}}\right)\right]\\
V_2(\delta) &= (x_{-\delta}-e_{-\delta})\,\kappa(x_{-\delta}-e_{-\delta},1-b_{-\delta})+(x_{\delta}-e_{\delta})\,
\kappa(x_{\delta}-e_{\delta},1-b_{\delta})\\
&\qquad -(x_{-\delta}+x_{\delta}-e_{-\delta}-e_{\delta})\,\kappa\left(\tfrac12(x_{-\delta}+x_{\delta}-e_{-\delta}-e_{\delta}),
1-\tfrac12(b_{-\delta}+b_{\delta})\right).
\end{aligned}
\end{equation}
\bl{jackap}
The determinant of the Jacobian matrix of $(a,b)\mapsto a\kappa(a,b)$ is strictly positive everywhere on $\DOM$.
\el
\bpr
The non-negativity of the Jacobian determinant is a consequence of the concavity of
$(a,b)\mapsto a\kappa(a,b)$ (recall Lemma \ref{l:ka}(ii)). The strict positivity can be checked with MAPLE
via the explicit expression $\kappa(a,b)$ given in den Hollander and Whittington \cite{dHW06}.
\epr

Since $(a,b)\mapsto a\kappa(a,b)$ is concave and twice
differentiable, Lemma \ref{jackap} allows us to assert that on $\DOM$
the Jacobian matrix of $(a,b)\mapsto a\kappa(a,b)$ has two strictly negative eigenvalues.
The second derivatives of $\kappa$ are continuous. Moreover,
the uniqueness of $(e_{-\delta},b_{-\delta})$ and $(e_{\delta},b_{\delta})$
imply their continuity in $\delta$, and so
there exists a $\widetilde{C}>0$ such that, for $\delta$ small enough,
\be{major1}
V_2(\delta)\leq -\widetilde{C}
\big[\big((x_{-\delta}-x_{\delta})-(e_{-\delta}-e_{\delta})\big)^2+(b_{-\delta}-b_{\delta})^2\big].
\ee
In what follows, we set $Y(\frac{e}{b})=
(\partial^2/\partial^2 \mu)[\mu\phi^\cI(\mu)](\frac{e}{b})$.
To bound $V_1(\delta)$ from above, we compute the Jacobian matrix of
$(e,b)\mapsto e\phi^\cI(e/b)$:
$$\tfrac{1}{b}\, Y(\tfrac{e}{b}) \left(%
\begin{array}{cc}
1 & -\tfrac{e}{b}\\
-\tfrac{e}{b} & \tfrac{e^2}{b^2}\\
\end{array}%
\right).$$
Thus, if for $t\in[0,1]$ and $u\in [0,t]$ we set $e_{u,t}=\tfrac{e_{-\delta}+e_{\delta}}{2}+t(u-\tfrac12)
(e_{-\delta}-e_{\delta})$ and
$b_{u,t}=\tfrac{b_{-\delta}+b_{\delta}}{2}+t(u-\tfrac12)(b_{-\delta}-b_{\delta})$, then a
Taylor expansion gives us
\be{major2}
V_1(\delta)=\tfrac14 \int_0^1 dt\, t\int_0^t du\, \tfrac{1}{b_{u,t}}
Y\big(\tfrac{e_{u,t}}{b_{u,t}}\big)
\big[(e_{-\delta}-e_{\delta})-\tfrac{e_{u,t}}{b_{u,t}} (b_{-\delta}-b_{\delta})\big]^2.
\ee
As explained in the proof of Proposition \ref{psiklsmooth}, the fact that $(\alpha,\beta,x)\in \cL_\psi$
implies $(e_0,b_0)\in\cL_{\alpha,\beta,x}$ and therefore
$(\alpha,\beta,\tfrac{e_0}{b_0})\in\cL_\phi$. Moreover,
$\cL_\phi$ is an open subset of $\CONE \times [1,\infty)$
and
$(e_{\delta},b_{\delta})$ is continuous in
$\delta$, so that for $\delta$ small enough, $t\in[0,1]$ and $u\in[0,t]$, we have
$(\alpha,\beta,\tfrac{e_{u,t}}{b_{u,t}})\in \cL_\phi$. This implies,
by Lemma \ref{stconc} and by the continuity of the second derivative of $\phi^\cI$ on $\cL_\phi$, that
there exists a $\widehat{C}>0$ such that, for $\delta$ small enough,
$\tfrac{1}{b_{u,t}} Y(\tfrac{e_{u,t}}{b_{u,t}})\leq -\widehat{C}$.
At this stage, we need to consider the following three cases:

\noindent[Case 1] $|b_{-\delta}-b_{\delta}|\geq \tfrac{b_0}{e_0}\tfrac{\delta}{4}$. Then,
\eqref{major1} gives
$V_2(\delta)\leq -\tfrac{\widetilde{C} b_0^2}{4^2 e_0^2} \delta^2$.

\noindent[Case 2] $|e_{-\delta}-e_{\delta}|\leq \delta$.
Then, since $x_{\delta}-x_{-\delta}=2\delta$,
\eqref{major1} gives
$V_2(\delta)\leq -\widetilde{C} \delta^2$.

\noindent[Case 3] $|e_{-\delta}-e_{\delta}|> \delta$ and
$|b_{-\delta}-b_{\delta}|< \tfrac{b_0}{e_0}\tfrac{\delta}{4}$.
By continuity of $e_\delta$ and $b_\delta$,
$\tfrac{e_{u,t}}{b_{u,t}}\leq \tfrac{2e_0}{b_0}$ for $\delta$ small enough and therefore
\be{major3}
|(e_{-\delta}-e_{\delta})-\tfrac{e_{u,t}}{b_{u,t}} (b_{-\delta}-b_{\delta})|
\geq |e_{-\delta}-e_{\delta}|-\tfrac{2e_0}{b_0} |b_{-\delta}-b_{\delta}|
\geq \delta-\tfrac{2e_0}{b_0} \tfrac{b_0}{e_0}\tfrac{\delta}{4}=\tfrac{\delta}{2}.
\ee
Thus, \eqref{major2} and \eqref{major3}  give
$V_1(\delta)\leq -\tfrac{\widehat{C}}{48}\delta^2$.

We conclude by setting
$C=\min\{\tfrac{\widetilde{C} b_0^2}{4^2 e_0^2},\widetilde{C},\tfrac{\widehat{C}}{48}\}$, so that
Cases 1,2
and 3 give $T(\delta)\leq -C\delta^2$ for $\delta$ small enough, which proves \eqref{equi2}.
\epr

\noindent
Lemma \ref{LGstrconv} implies that $J_3>0$. Hence, the implicit function theorem can indeed
be applied to (\ref{fe}), and it follows that $f$ is infinitely differentiable on $\cL$.


\end{document}